\documentclass{article}

\usepackage[main, final]{neurips_2026}

\usepackage[utf8]{inputenc} 
\usepackage[T1]{fontenc}    
\usepackage{hyperref}       
\usepackage{url}            
\usepackage{booktabs}       
\usepackage{amsfonts}       
\usepackage{nicefrac}       
\usepackage{microtype}      
\usepackage{xcolor}         
\usepackage{amsmath}
\usepackage{mathtools}
\usepackage{amssymb}
\usepackage{amsthm}
\usepackage{bbm}
\usepackage[ruled]{algorithm2e}
\usepackage{graphicx}
\usepackage{minitoc}
\usepackage{subcaption}
\usepackage{enumitem}

\DeclareMathOperator*{\argmax}{argmax}
\DeclareMathOperator*{\argmin}{argmin}

\title{Mirror Descent-Ascent for mean-field min-max problems}

\author{Razvan-Andrei Lascu\\
  RIKEN AIP\\
  \texttt{razvan-andrei.lascu@riken.jp} \\
  \And
   Mateusz B. Majka\\
   Heriot-Watt University \\
   \texttt{m.majka@hw.ac.uk} \\
   \AND
  \L{}ukasz Szpruch \\
   University of Edinburgh\\
   \texttt{l.szpruch@ed.ac.uk} \\
}

\newtheorem{theorem}{Theorem}[section]
\newtheorem{proposition}[theorem]{Proposition}
\newtheorem{lemma}[theorem]{Lemma}
\newtheorem{corollary}[theorem]{Corollary}
\newtheorem{example}[theorem]{Example}
\newtheorem{definition}[theorem]{Definition}
\newtheorem{assumption}[theorem]{Assumption}
\newtheorem{remark}[theorem]{Remark}
\def\floor#1{\lfloor #1 \rfloor}

\begin{document}

\maketitle

\begin{abstract}
  We study two variants of the mirror descent-ascent (MDA) algorithm for solving min-max problems on the space of measures: simultaneous and alternating. We work under assumptions of convexity-concavity and relative smoothness of the payoff function with respect to a suitable Bregman divergence, defined on the space of measures via flat derivatives. We establish non-asymptotic convergence rates to mixed Nash equilibria, measured in the Nikaid\^{o}-Isoda error, proving an $\mathcal{O}(N^{-1/2})$ rate for simultaneous MDA and an improved $\mathcal{O}(N^{-2/3})$ rate for alternating MDA. The main technical contribution is an infinite-dimensional dual space analysis that relates Bregman divergences on measures to dual Bregman divergences on spaces of bounded continuous functions, allowing us to control asymmetric commutator terms created by alternating updates. The results substantially generalize prior analyses restricted to bilinear objectives and also apply to nonlinear convex-concave problems on measure spaces, thereby providing a unified theoretical foundation for MDA in mean-field min-max optimization.
\end{abstract}

\section{Introduction}
\label{sec:Introduction}
Numerous tasks in machine learning can be framed as optimization problems for functions defined on the space of probability measures. For instance, in supervised learning, \cite{Chizat2018OnTG,Mei2018AMF,Rotskoff2018NeuralNA} showed that training a shallow neural network (NN) in the mean-field regime, i.e., an infinite-width one-hidden-layer NN, can be viewed as minimizing a convex function over the space of probability distributions of the network parameters. This observation has become a central tool for analysing the convergence of training algorithms for mean-field NNs; see, e.g., \cite{10.1214/20-AIHP1140, chizat2022meanfield, Nitanda2022ConvexAO, suzuki2023meanfield}. The same mean-field optimization paradigm has also been extended to min-max settings, where the objective is to compute mixed Nash equilibria (MNEs) of games over spaces of probability measures; see, e.g., \cite{pmlr-v97-hsieh19b, NEURIPS2020_e97c864e, wang2023exponentially, yulong, trillos2023adversarial, kim2024symmetric,pmlr-v291-cai25a}.

In this work, we study the convergence of mirror descent-ascent (MDA) algorithms for convex-concave min-max games over probability measures. The central question we address is how the order of play affects convergence in the infinite-dimensional mean-field setting. In games, learning dynamics depend crucially on whether players update simultaneously, with both players moving from the same current state, or alternately, with one player updating after observing the other player's move. To our knowledge, existing discrete-time convergence guarantees for mean-field min-max games focus on simultaneous updates; see, e.g., \cite{pmlr-v97-hsieh19b, wang2023exponentially}. We provide a rigorous comparison between simultaneous and alternating MDA and prove that alternating play can achieve a strictly faster convergence rate.

The mean-field perspective offers a tractable route to analyzing particle algorithms whose finite-particle implementations evolve in high-dimensional spaces. Rather than tracking particles directly, one studies the infinite-particle limit and derives guarantees at the level of probability measures. This viewpoint makes it possible to isolate the effect of algorithmic choices, such as simultaneous versus alternating updates, on the evolution of the players' mixed strategies. Our analysis shows that this distinction is not merely algorithmic. In the alternating scheme, the second player uses additional information from the first player's updated strategy, and this changes the structure of the error terms. Controlling these terms requires a dual-space argument for Bregman divergences on measures and leads to an improved rate for alternating MDA.

Our results provide a theoretical justification for alternating updates in mean-field min-max optimization. In Section \ref{example:GAN-example}, beyond GANs, we discuss applications to learning MNEs in preference models related to Reinforcement Learning from Human Feedback (RLHF). We also show that the analysis is not restricted to bilinear games. It applies to nonlinear convex-concave min-max problems over measures, including adversarial training of mean-field neural networks.

\subsection{Notation and setup}
For any $\mathcal{X} \subseteq \mathbb R^d,$ let $\mathcal{P}(\mathcal{X})$ denote the set of probability measures on $\mathcal{X}.$ We endow $\mathcal{P}(\mathcal{X})$ with the narrow convergence topology, i.e., convergence against continuous bounded functions in $C_b(\mathcal X)$. In game theory, if $\mathcal{X}$ is the set of \textit{(pure) strategies} available to the players, then $\mathcal{P}(\mathcal{X})$ is known as the set of \textit{mixed strategies}. Let $\mathcal{C}, \mathcal{D} \subseteq \mathcal{P}(\mathcal{X})$ be nonempty and convex. We consider a convex-concave (cf. Assumption \ref{def: def-F-conv-conc}) payoff function $F:\mathcal{C} \times \mathcal{D} \to \mathbb R$ and the associated min-max game
\begin{equation}
\label{eq:game}
    \min_{\nu \in \mathcal{C}} \max_{\mu \in \mathcal{D}} F(\nu, \mu).
\end{equation}
We are interested in finding \textit{mixed Nash equilibria} (MNEs) for game \eqref{eq:game}, i.e., pairs of strategies $(\nu^*, \mu^*) \in \mathcal{C} \times \mathcal{D}$ such that, for any $(\nu, \mu) \in \mathcal{C} \times \mathcal{D},$ we have
    \begin{equation}
    \label{eq:saddle}
        F(\nu^*, \mu) \leq F(\nu^*, \mu^*) \leq F(\nu, \mu^*).
    \end{equation}
Throughout, we assume that there exists at least one MNE for game \eqref{eq:game}.\footnote{\label{fnote}If $F$ is continuous and convex-concave, and both $\mathcal C$ and $\mathcal D$ are narrowly compact, then an MNE of \eqref{eq:game} exists by Sion’s minimax theorem \cite{sion}. For the integral payoff $F(\nu, \mu) = \iint_{\mathcal{X}\times \mathcal{X}} f(x,y) \nu(\mathrm{d}x)\mu(\mathrm{d}y),$ an MNE exists if $\mathcal C$ and $\mathcal D$ are narrowly compact and $f\in C_b(\mathcal X\times\mathcal X)$. If $\mathcal X$ is compact, it suffices to assume that $f$ is continuous.} In min-max games, the distance between a pair of strategies $(\nu, \mu)$ and an MNE is typically measured using the Nikaid\^{o}-Isoda (NI) error \cite{Nikaid1955NoteON}, which, for all $(\nu, \mu) \in \mathcal{C}\times \mathcal{D},$ is defined by
\begin{equation*}
    \text{NI}(\nu,\mu) \coloneqq \max_{\mu' \in \mathcal{D}} F(\nu, \mu') - \min_{\nu' \in \mathcal{C}} F(\nu', \mu).
\end{equation*}
Straight from the definition, we see that $\text{NI}(\nu,\mu) \geq 0$ for all $(\nu, \mu) \in \mathcal{C} \times \mathcal{D},$ and from \eqref{eq:saddle} it follows that $\text{NI}(\nu,\mu) = 0$ if and only if $(\nu, \mu)$ is an MNE. Our notions of MNE and NI error agree with the classical definitions in the bilinear setting, where $F(\nu,\mu) = \int_{\mathcal{X}} \int_{\mathcal{X}} f(x,y) \nu(\mathrm{d}x)\mu(\mathrm{d}y),$ for some $f:\mathcal{X} \times \mathcal{X} \to \mathbb R.$ We therefore adopt the same terminology for the more general class of convex-concave games over measures.

Following \cite[Definition 5.43]{Carmona2018ProbabilisticTO} and \cite[Definition 7.12]{santambrogio2015optimal}, we introduce the notion of differentiability on the space of measures that we utilize throughout the paper.
\begin{definition}
\label{def:fderivative}
Let $\mathcal{X} \subseteq \mathbb{R}^d$, and let $\mathcal{K} \subseteq \mathcal{P}(\mathcal{X})$ be convex. A functional $F: \mathcal{P}(\mathcal{X}) \to \mathbb{R}$ is in $\mathfrak{C}^1(\mathcal{K})$ if there exists a continuous bounded function $\frac{\delta F}{\delta \nu}: \mathcal{K} \times \mathcal{X} \to \mathbb{R}$ satisfying
\begin{equation*}
\lim_{\varepsilon \to 0} \frac{F(\nu + \varepsilon (\nu'-\nu)) - F(\nu)}{\varepsilon} = \int_{\mathcal{X}} \frac{\delta F}{\delta \nu}(\nu,x) (\nu' - \nu)(\mathrm{d}x) \quad \text{for all } \nu, \nu' \in \mathcal{K}.
\end{equation*}
We call $\frac{\delta F}{\delta \nu}$ the flat derivative (or first variation) of $F$, normalized such that $\int_{\mathcal{X}} \frac{\delta F}{\delta \nu}(\nu,x) \nu(\mathrm{d}x) = 0$.
\end{definition}
We now introduce the notion of Bregman divergence and its commutator on the space of probability measures using the flat derivative (Definition \ref{def:fderivative}), following \cite[Definition 3]{korba}, who defined it via directional derivatives. We assume flat differentiability so that, for each fixed $\nu$, the first-order term is linear in the measure increment $\nu'-\nu$, a property not guaranteed by directional differentiability alone. Let $\mathcal{E}\subseteq \mathcal{P}(\mathcal{X})$ be nonempty and convex.
\begin{definition}[Bregman divergence and commutator]
\label{def:bregman-div}
     Assume that $h \in \mathfrak{C}^1(\mathcal{E})$. The $h$-Bregman divergence and its commutator are the maps $D_h,\mathrm{Com}_h: \mathcal{E} \times \mathcal{E} \to \mathbb R,$ respectively, given by
    \begin{align*}
        &D_h(\nu', \nu) := h(\nu') - h(\nu) - \int_{\mathcal{X}} \frac{\delta h}{\delta \nu}(\nu,x)(\nu'-\nu)(\mathrm{d}x), \\ 
        &\mathrm{Com}_h(\nu',\nu):=D_h(\nu', \nu) - D_h(\nu, \nu').
    \end{align*}
\end{definition}
Bregman divergence plays a key role in the definition of mirror descent/ascent algorithms (see Section \ref{subsec:sim and alt MDA}). Its use as the update penalty allows the algorithm to adapt naturally to the geometry of the underlying space. Since $D_h$ is generally asymmetric, the role of $\operatorname{Com}_h$ is to measure this asymmetry.

\subsection{Simultaneous and alternating MDA}
\label{subsec:sim and alt MDA}
With the definition of the Bregman divergence in place, we now present the simultaneous and alternating MDA algorithms. For a given stepsize $\tau > 0,$ and fixed initial pair of strategies $(\nu_0, \mu_0) \in \mathcal{C} \times \mathcal{D},$ for $n \geq 0,$ the \textit{simultaneous} and \textit{alternating} MDA algorithms are respectively defined in Algorithm \ref{eq:mirror-sim-explicit} and Algorithm \ref{eq:mirror-alt}.
\begin{algorithm}[!ht] 
   \caption{\textsc{Simultaneous MDA}}
   \label{eq:mirror-sim-explicit}
   \KwIn{Objective function $F,$ initial measures $(\nu^0, \mu^0)$, stepsize $\tau > 0$}
   \For{$n =0, 1, \dots ,N-1$}
   {
    $\nu^{n+1} = \argmin\limits_{\nu \in \mathcal{C}} \{\int_{\mathcal{X}} \frac{\delta F}{\delta \nu}(\nu^n, \mu^{n}, x)(\nu-\nu^n)(\mathrm{d}x) + \frac{1}{\tau} D_h(\nu, \nu^n)\},$
    
    $\mu^{n+1} = \argmax\limits_{\mu \in \mathcal{D}} \{\int_{\mathcal{X}} \frac{\delta F}{\delta \mu}(\nu^{n}, \mu^{n}, y)(\mu-\mu^n)(\mathrm{d}y) - \frac{1}{\tau} D_h(\mu, \mu^n)\}$
   }
\KwOut{$\left(\frac{1}{N}\sum_{n=0}^{N-1}\nu^n,\frac{1}{N}\sum_{n=0}^{N-1}\mu^n\right)$}
\end{algorithm}
\begin{algorithm}[!ht] 
   \caption{\textsc{Alternating MDA}}
   \label{eq:mirror-alt}
   \KwIn{Objective function $F,$ initial measures $(\nu^0, \mu^0)$, stepsize $\tau > 0$}
   \For{$n =0, 1, \dots ,N-1$}
   {
    $\nu^{n+1} = \argmin\limits_{\nu \in \mathcal{C}} \{\int_{\mathcal{X}} \frac{\delta F}{\delta \nu}(\nu^n, \mu^{n}, x)(\nu-\nu^n)(\mathrm{d}x) + \frac{1}{\tau} D_h(\nu, \nu^n)\},$
    
    $\mu^{n+1} = \argmax\limits_{\mu \in \mathcal{D}} \{\int_{\mathcal{X}} \frac{\delta F}{\delta \mu}(\nu^{n+1}, \mu^{n}, y)(\mu-\mu^n)(\mathrm{d}y) - \frac{1}{\tau} D_h(\mu, \mu^n)\}$
   }
\KwOut{$\left(\frac{1}{N}\sum_{n=0}^{N-1}\nu^{n+1},\frac{1}{N}\sum_{n=0}^{N-1}\mu^n\right)$}
\end{algorithm}
The distinction between the two algorithms is important for the convergence rate. In Algorithm \ref{eq:mirror-sim-explicit}, both players evaluate their updates at the same old pair $(\nu^n,\mu^n)$, so the analysis leaves a first-order interaction error that leads to the $\mathcal{O}(N^{-1/2})$ rate. In Algorithm \ref{eq:mirror-alt}, the second player observes the first player’s updated move $\nu^{n+1}$. This extra information creates an asymmetric term, but once that term is controlled, the remaining error is of higher order. As we will discuss in the sequel, the proof of Theorem \ref{thm:conv-alt-bregman} shows that this asymmetric term can be rewritten as a Bregman commutator, transported to the dual function space, and bounded cubically by $\tau^3$ using Assumption \ref{assumption:lipschitz-h^*2}. This is the mechanism behind the improved $\mathcal{O}(N^{-2/3})$ rate of the alternating MDA, compared to the simultaneous MDA.

\begin{remark}[Entropy MDA and the Fisher--Rao gradient flow]
    When $h$ is the relative entropy, we can view Algorithms \ref{eq:mirror-sim-explicit} and \ref{eq:mirror-alt} as Euler discretizations of a Fisher--Rao gradient flow, whose continuous-time convergence with explicit rates for mean-field min-max games was proved in \cite{lascu2024fisherrao} (cf. also \cite{https://doi.org/10.48550/arxiv.2206.02774} for single-agent convex optimization). 
\end{remark}
\begin{remark}[Connection to the continuous-time Bregman gradient flow]
    In Appendix \ref{sec:proof-implicit-game}, we provide an implicit MDA algorithm that achieves a convergence rate of $\mathcal{O}(N^{-1})$, which matches the rate $\mathcal{O}(t^{-1})$ of the continuous-time gradient flow obtained by letting $\tau \to 0.$ However, this scheme is not implementable in practice, as it requires each player to know the next move of their opponent at every iteration. Consequently, while it provides a useful theoretical benchmark, practical algorithms rely on explicit (simultaneous or alternating) mirror schemes.
\end{remark}

\subsection{Our contribution}
We provide a theoretical contribution to the convergence analysis of MDA for mean-field min-max problems. We establish non-asymptotic rates for simultaneous and alternating MDA on spaces of probability measures under convexity-concavity and relative smoothness assumptions. Our contribution can be summarized as follows:
\begin{itemize}[leftmargin = 5mm]
    \item From the perspective of measure space MD, our work extends the single-player results of \cite{korba} to two-player min-max games. This extension is nontrivial because, unlike in convex minimization, the objective is not monotonically decreasing along the iterates. Consequently, the analysis must be carried out in terms of the Nikaid\^o-Isoda error and requires new proof techniques based on convex conjugates and Bregman divergences on dual function spaces.
    \item From the perspective of mean-field min-max optimization, our work generalizes \cite{pmlr-v97-hsieh19b} in several directions. We go beyond bilinear GAN objectives by covering general convex-concave payoffs, including RLHF and nonlinear adversarial training of mean-field NNs. Our analysis also applies to mirror geometries beyond relative entropy. Finally, whereas \cite{pmlr-v97-hsieh19b} establishes an explicit $\mathcal{O}(N^{-1/2})$ rate only for simultaneous MDA, we prove that alternating MDA achieves the faster rate $\mathcal{O}(N^{-2/3})$.
    \item At the technical level, the main novelty is a dual space commutator argument for alternating MDA on spaces of measures. Alternating updates generate asymmetric Bregman terms. In finite dimensions, such terms can be controlled using the inner product identification between iterates and the mirror map since both live in the same finite-dimensional space, but this identification is unavailable in the measure setting. Iterates are measures, while the mirror map naturally lives in a function space. We therefore pass to the primal-dual pair $(\mathcal{M}(\mathcal{X}),C_b(\mathcal{X}))$, define the convex conjugate $h^*$ on $C_b(\mathcal{X})$, and prove a transport identity showing that primal Bregman commutators on measures coincide with dual Bregman commutators on functions. This allows us to apply Fréchet smoothness of the mirror map to control the update error and obtain the improved $\mathcal{O}(N^{-2/3})$ rate. To our knowledge, the use of this dual formulation of MDA on a function space is novel and may be useful beyond the present setting.
\end{itemize}

\subsection{Motivating examples}
\paragraph{Training of GANs.}
\label{example:GAN-example}
Let $\hat{\xi} \in \mathcal{P}(\mathcal{Y})$ be the empirical measure of the i.i.d. sampled particles $\{x_i\}_{i=1}^M \subseteq \mathcal{Y},$ and let $\xi \in \mathcal{P}(\mathcal{Z})$ be a source measure. Consider the measurable parametrized transport map $T_{\theta}: \mathcal{Z} \to \mathcal{Y}$ (which typically can be viewed as a neural network with parameters $\theta \in \Theta \subseteq \mathbb R^d$). The \textit{pushforward} of the measure $\xi$ on $\mathcal{Z}$ via $T_{\theta}$ is the measure $T_{\theta}{\#} \xi$ on $\mathcal{Y}$ characterized by $\int_{\mathcal{Y}} \varphi \mathrm{d}\left(T_{\theta}{\#} \xi\right) = \int_{\mathcal{Z}} \left(\varphi \circ T_{\theta}\right)\mathrm{d}\xi,$ for any measurable function $\varphi: \mathcal{Y} \to \mathbb R.$ The aim of a GAN is to search for the optimal set of parameters $\theta^* \in \Theta$ that minimizes the distance between the generated measure $T_{\theta^*} \# \xi$ and the empirical measure $\hat{\xi}.$ In order to evaluate this distance, we define the function $D_{w}: \mathcal{Y} \to \mathbb R$ (which can also be viewed as a neural network with parameters $w \in \mathcal{W} \subseteq \mathbb R^d$), and solve the min-max problem
\begin{equation*}
    \min_{\theta \in \Theta} \max_{w \in \mathcal{W}} \int_{\mathcal{Y}} D_{w}(y) (T_{\theta} \# \xi - \hat{\xi})(\mathrm{d}y).
\end{equation*}
For example, if the family of functions $\left\{D_{w}\right\}_{w \in \mathcal{W}}$ is either 1-Lipschitz continuous or uniformly bounded, the resulting GAN corresponds to the Wasserstein GAN or the Total Variation GAN, respectively \cite{pmlr-v70-arjovsky17a}. On the other hand, if the family of functions $\left\{D_{w}\right\}_{w \in \mathcal{W}}$ belongs to the norm unit ball of a reproducing kernel Hilbert space (RKHS), we recover the Maximum Mean Discrepancy (MMD) GAN \cite{NIPS2017_dfd7468a}. Solving this problem on the finite-dimensional subspaces $\theta, w \subseteq \mathbb R^d$ may pose serious challenges such as the lack of existence of pure Nash equilibria. Instead, we lift the problem to the space of probability measures and search for MNEs, i.e., optimal distributions over the set of parameters. That is, by setting $f(\theta, w) := \int_\mathcal{Y} D_{w}(y) (T_{\theta} \# \xi - \hat{\xi})(\mathrm{d}y),$ we solve the mean-field min-max game
    \begin{equation}
    \label{eq:bilinear-GAN-lifted}
        \min_{\nu \in \mathcal{P}\left(\Theta\right)} \max_{\mu \in \mathcal{P}\left(\mathcal{W}\right)}\int_{\mathcal{W}} \int_{\Theta} f(\theta, w) \nu(\mathrm{d}\theta)\mu(\mathrm{d}w).
    \end{equation}
Note that the lifted problem is bilinear in $\nu$ and $\mu,$ so an MNE for \eqref{eq:bilinear-GAN-lifted} exists under mild conditions (see footnote \ref{fnote}). We illustrate that alternating updates speed up GANs training in simple numerical experiments (see Appendix \ref{section:numerical-example}).

\paragraph{Nash Learning from Human Feedback.}
Reinforcement Learning from Human Feedback (RLHF) has become essential to LLMs training since it provides a principled way to incorporate human preferences into the learning objective and thereby improve alignment beyond pretraining. This example illustrates that our framework applies to the RLHF (see, e.g., \cite{munos24a,calandriello2024humanalignmentlargelanguage} and the references therein for more context). In particular, we show how it can be cast as a convex-concave two-player game. Let $\mathcal{S}$ denote a measurable space of prompts/contexts and $\rho\in\mathcal P(\mathcal{S})$ their distribution. Let $\mathcal{A}$ denote a measurable space of responses. A preference model is a measurable map $P(\cdot \succ \cdot|x): \mathcal{A}\times \mathcal{A} \to [0,1]$ given by $(y,y')\mapsto P(y\succ y'|x)$, and interpreted as the probability that a random human prefers response $y$ over $y'$ given context $x$. Assume that the preference model is anti-symmetric, i.e., $ P(y\succ y'|x) = 1- P(y'\succ y|x)$. Given two policies $\mu,\nu:\mathcal{S}\to\mathcal P(\mathcal{A}|\mathcal{S})$, define the preference between them by
\begin{equation*}
P(\mu\succ \nu):= \int_\mathcal{S} \int_\mathcal{A} \int_\mathcal{A}P(y\succ y'| x)\nu(\mathrm{d}y')\mu(\mathrm{d}y)\rho(\mathrm{d}x).
\label{eq:nlhf-P}
\end{equation*}
Due to anti-symmetry, we have $P(\mu\succ \nu) = 1- P(\nu\succ \mu)$ for any policies $\mu,\nu \in \mathcal P(\mathcal{A}|\mathcal{S})$. Define the convex feasible set of policies $\mathcal C = \mathcal{D}:=\left\{\pi:\ \pi(\cdot|x)\in \mathcal{P}(\mathcal{A}),\ \forall x \in \mathcal{S}\right\}$. The learning objective is to find a policy $\mu$ that is preferred over every alternative policy $\nu$ and thus can be formulated as 
\begin{equation*}
    \max_{\mu \in \mathcal{C}} \min_{\nu \in \mathcal{C}} P(\mu\succ \nu).\footnote{Because $P(\mu\succ \nu)$ is bilinear in $(\mu,\nu)$, it admits a Nash equilibrium $(\nu^*,\mu^*)$ which satisfies $\max_{\mu \in \mathcal{C}} \min_{\nu \in \mathcal{C}} P(\mu\succ \nu) = \min_{\nu \in \mathcal{C}} \max_{\mu \in \mathcal{C}} P(\mu\succ \nu) = P(\mu^*\succ \nu^*)$ \cite[Theorem 1]{rosen}.}
\end{equation*}

\paragraph{Adversarial training of mean-field NNs.}
Let $\mathcal{Y} \subseteq \mathbb R$ and $\mathcal{Z} \subseteq \mathbb R^{d-1}$ be compact and convex with $\hat{\mu} \in \mathcal{P}(\mathcal{Y} \times \mathcal{Z})$ representing the training data $(y,z) \in \mathcal{Y} \times \mathcal{Z}.$ Let $x=(w,b) \in \mathbb R^{d-1} \times \mathbb R$ be the parameters of the neural network, let $\varphi:\mathbb R \to \mathbb R$ be a bounded, continuous, non-constant activation function and define the function $\hat{\varphi}(x,z) \coloneqq \varphi(w \cdot z + b)$. The training of the two-layer mean-field neural network aims to find the optimal distribution of parameters $\nu \in \mathcal{P}(\mathbb R^d)$ which minimizes the lifted $L^2$-loss function (see, e.g.,\ \cite[Section 3]{10.1214/20-AIHP1140} and the references therein)
\begin{equation*}
\min_{\nu \in \mathcal{P}(\mathbb R^d)} F_0(\nu):=\int_{\mathcal{Y} \times \mathcal{Z}} \frac{1}{2}\left|y-\mathbb E^{X \sim \nu}[\hat{\varphi}(X, z)]\right|^2 \hat{\mu}(\mathrm{d}y, \mathrm{d}z).
\end{equation*}
To account for potential attacks by an adversary aiming to manipulate the training data $\hat{\mu},$ we minimize over the parameter distribution $\nu,$ considering the ``worst-case'' perturbation of $\hat{\mu}.$ This leads to the two-player game
\begin{equation}\label{eq:adversarial_minmax}
    \min_{\nu \in \mathcal{P}(\mathbb R^d)} \max_{\mu \in \mathcal{P}(\mathcal{Y} \times \mathcal{Z})} \left( F_0(\nu,\mu) - \operatorname{TV}_\varepsilon^2(\mu,\hat{\mu}) \right),
\end{equation}
where, for any $\varepsilon > 0$, $\operatorname{TV}_\varepsilon$ denotes the smoothed total variation distance, which represents the cost incurred by the adversary to alter the original training data $\hat{\mu}$, and it is defined by
\begin{equation*}
\operatorname{TV}_\varepsilon(\mu,\hat\mu):=\frac12\int_{\mathcal{Y} \times \mathcal{Z}} \sqrt{(m(r)-\hat m(r))^2+\varepsilon^2}\ \mathrm{d}r,
\end{equation*}
for $\mu$ and $\hat\mu$ absolutely continuous with respect to Lebesgue with density $m$ and $\hat m$, respectively. The resulting objective $F_\varepsilon(\nu, \mu) \coloneqq F_0(\nu, \mu) - \operatorname{TV}_\varepsilon^2(\mu,\hat{\mu})$ is a non-linear function covered by our general framework (see Proposition \ref{prop:verification-adversarial}). The choice of the incurred cost in \eqref{eq:adversarial_minmax} is, to an extent, arbitrary, and we focus here on $\operatorname{TV}_\varepsilon^2$ due to its convenience for verifying our assumptions. From a technical point of view, we consider the smoothed TV distance because, unlike the classical TV distance, it is flat differentiable everywhere. Alternative cost functions include the Wasserstein distance \cite{bai2023wasserstein,trillos2023adversarial} and the KL divergence \cite{si2023distributionally}.

\subsection{Related works}
\label{sec:related-works}
Mirror descent (MD) was originally proposed in \cite{nemirovski1983problem} for solving convex optimization problems and has been extensively studied on finite-dimensional vector spaces (see, e.g., \cite{BECK2003167,bubeck2015convex,LuNesterov}). In the finite-dimensional min-max setting, \cite{wibisono2022alternating} analyzed alternating MDA and showed that alternating updates can improve the convergence rate for constrained min-max games. In finite dimensions, the proof relies on Euclidean vector space geometry, which identifies the primal and dual spaces through the Euclidean inner product. In particular, the Bregman potential and the mirror map belong to the same finite-dimensional vector space. 

In the space of probability measures, there is no canonical inner product structure. We therefore first construct the appropriate primal space geometry through flat derivatives and Bregman divergences on measures. We then pass to the duality pair between signed measures and bounded continuous functions in order to define the convex conjugate and the corresponding dual mirror map. This dual space construction is essential for the analysis of the alternating scheme.

The MD algorithm has also been extended to infinite-dimensional settings for studying optimization problems over spaces of measures, with applications in machine learning, including Sinkhorn's and Expectation-Maximization algorithms \cite{korba}, as well as policy optimization for reinforcement learning \cite{tomar2021mirror, KerimkulovLeahySiskaSzpruchZhang2025}. By leveraging results from finite-dimensional optimization \cite{Bauschke2017ADL,LuNesterov}, \cite{korba} extends the convergence proof of \cite{LuNesterov} to infinite-dimensional MD. In particular, they show that convexity and relative smoothness of the objective are sufficient to obtain an $\mathcal{O}(N^{-1})$ convergence rate for single-agent measure optimization. In contrast, our setting is a two-player min-max problem, where monotonic decrease of the objective along the iterates is no longer available and where the alternating scheme creates additional asymmetric terms that must be controlled through the dual space construction. 

Other works such as \cite{pmlr-v97-hsieh19b, dvurechensky2024analysis} studied infinite-dimensional Mirror-Prox algorithms for finding MNEs of two-player zero-sum games. Mirror-Prox methods can achieve $\mathcal{O}(N^{-1})$ rates under Lipschitz-type conditions on the gradients of the payoff function (see, e.g., \cite[Section 5.2.3]{bubeck2015convex}). Existing infinite-dimensional analyses (e.g., entropic Mirror-Prox for convex-concave payoffs \cite{dvurechensky2024analysis}) rely on a $\operatorname{TV}$-Lipschitz assumption on the Fréchet derivatives of the payoff (see \cite[Assumption 4.2]{dvurechensky2024analysis}), which is substantially stronger than our relative smoothness assumption (see Assumption \ref{def:relative-smoothness} and Appendix \ref{app:further-works}). In particular, for bilinear games $F(\nu,\mu)=\iint f(x,y)\nu(\mathrm{d}x)\mu(\mathrm{d}y)$, relative smoothness automatically holds with $L_\nu=L_\mu=0$ for any $f$, whereas the TV-Lipschitz requirement on the derivatives forces $f$ to be bounded and Lipschitz, excluding many relevant settings such as RLHF and adversarial NNs due to lack of Lipschitz continuity. Moreover, Mirror-Prox requires two mirror steps per player per iteration, thus doubling gradient evaluations, and in the stochastic setting (such as in particle implementations) its rate matches $\mathcal{O}(N^{-1/2})$. 

Finally, a convergence theory for an alternating Mirror-Prox scheme (the natural counterpart to our alternating MDA) appears technically challenging and, to our knowledge, is still an open question. We discuss the technical difficulties involved in analyzing Mirror-Prox in greater detail in Appendix \ref{app:further-works}. The appendix also contains a more extensive discussion of related work on optimization over spaces of measures. 

\section{Convergence of the simultaneous MDA Algorithm \texorpdfstring{\ref{eq:mirror-sim-explicit}}{1.6}}
\label{section:2}
In what follows, we state our standing assumptions and the necessary definitions for introducing the convergence results.
\begin{assumption}[Differentiability and convexity of $h$]
\label{assumption:assump-h}
    Assume that $h \in \mathfrak{C}^1(\mathcal{E})$ (cf. Definition \ref{def:fderivative}). Moreover, assume that $h:\mathcal{E} \to \mathbb R$ is $\alpha$-strongly convex relative to the total variation (TV) distance, i.e., there exists $\alpha > 0$ such that for all $\varepsilon \in [0,1]$ and $\nu',\nu \in \mathcal{E},$ with $\nu_\varepsilon := (1-\varepsilon)\nu + \varepsilon \nu'$,
    \begin{align*}
        h\left(\nu_\varepsilon\right) \leq (1-\varepsilon)h(\nu) + \varepsilon h(\nu') - \frac{\alpha}{2}\varepsilon(1-\varepsilon)\operatorname{TV}^2(\nu',\nu).
    \end{align*}
\end{assumption}
If Assumption \ref{assumption:assump-h} holds, then we show in Lemma \ref{lemma:strong-convexity-h} that for any $\nu', \nu \in \mathcal{E},$ we have
\begin{equation}\label{eg:h_strong_convexity_TV}
    h(\nu') - h(\nu) \geq \int_{\mathcal{X}} \frac{\delta h}{\delta \nu}(\nu,x)(\nu'-\nu)(\mathrm{d}x) + \frac{\alpha}{2}\operatorname{TV}^2(\nu',\nu).
\end{equation}
Assumption \ref{assumption:assump-h} implies \eqref{eg:h_strong_convexity_TV} due to Lemma \ref{lemma:strong-convexity-h} and hence, straight from Definition \ref{def:bregman-div}, we obtain $D_h(\nu', \nu) \geq \frac{\alpha}{2}\operatorname{TV}^2(\nu',\nu),$ for all $\nu', \nu \in \mathcal{E},$ and $D_h(\nu', \nu) = 0$ if and only if $\nu'=\nu.$ In Appendix \ref{subsection:examples satifying 1.1}, we provide examples of the corresponding sets $\mathcal{E}$ for the relative entropy and $\chi^2$-divergence cases that satisfy Assumption \ref{assumption:assump-h}. For other examples of functions $h$ that verify the inequality $ D_h(\nu',\nu) \geq \frac{\alpha}{2}\operatorname{TV}^2(\nu',\nu)$ and hence Assumption \ref{assumption:assump-h}, see \cite[Lemma $3.2$]{chizat2023convergence}.

Let $F: \mathcal{C} \times \mathcal{D} \to \mathbb R$ be such that $F(\cdot, \mu) \in \mathfrak{C}^1(\mathcal{C})$ and $F(\nu, \cdot) \in \mathfrak{C}^1(\mathcal{D})$ (cf. Definition \ref{def:fderivative}). Under Definition \ref{def:bregman-div}, the Bregman divergences of $F(\cdot, \mu)$ and $F(\nu, \cdot)$ are respectively given by
\begin{align*}
    D_{F(\cdot,\mu)}(\nu',\nu) &:= F(\nu', \mu) - F(\nu, \mu) - \int_{\mathcal{X}} \frac{\delta F}{\delta \nu}(\nu,\mu,x) (\nu'-\nu)(\mathrm{d}x),\\
    D_{F(\nu,\cdot)}(\mu',\mu) &:= F(\nu, \mu') - F(\nu, \mu) - \int_{\mathcal{X}} \frac{\delta F}{\delta \mu}(\nu,\mu,y) (\mu'-\mu)(\mathrm{d}y).
\end{align*}
\begin{assumption}[Convexity-concavity of $F$]
\label{def: def-F-conv-conc}
Assume that $F$ is convex in $\nu$ and concave in $\mu,$ i.e., for any $\nu, \nu' \in \mathcal{C}$ and $\mu, \mu' \in \mathcal{D},$ we have
\begin{equation*}
    D_{F(\cdot,\mu)}(\nu',\nu) \geq 0, \quad D_{F(\nu,\cdot)}(\mu',\mu) \leq 0.
\end{equation*}
\end{assumption}
\begin{assumption}[Uniform boundedness of  the flat derivatives of $F$]
	\label{assumption:bddflatF}
	Suppose that there exists $C_1 > 0$ and $C_2 > 0$ such that for all $(\nu,\mu)\in \mathcal{C} \times \mathcal{D},$ and all $x,y \in \mathcal{X},$ we have
    \begin{equation*}
        \left|\frac{\delta F}{\delta \nu}(\nu, \mu, x)\right| \leq C_1, \quad \left|\frac{\delta F}{\delta \mu}(\nu, \mu, y)\right| \leq C_2.
    \end{equation*}
\end{assumption}
Assumptions \ref{def: def-F-conv-conc} and \ref{assumption:bddflatF} are standard in the mean-field optimization literature, see, e.g., \cite{zhenjiefict,lascu2023entropic}. In Proposition \ref{prop:verification-GANs} and \ref{prop:verification-adversarial}, we verify that Assumptions \ref{def: def-F-conv-conc} and \ref{assumption:bddflatF} are satisfied by the examples in Section \ref{example:GAN-example}. Another important tool utilized in the proofs of the main results is the three-point inequality that is proved in Appendix \ref{appendix:additional-results} following \cite{korba}[Lemma~$3$].
\begin{lemma}[Three-point inequality]
\label{lemma:Bregman-prox-ineq}
Let Assumption \ref{assumption:assump-h} hold. Let $G:\mathcal{E} \to \mathbb R$ be convex and $G \in \mathfrak{C}^1(\mathcal{E}).$ For all $\mu \in \mathcal{E},$ suppose that there exists $\Bar{\nu} \in \mathcal{E}$ such that $\Bar{\nu} \in \argmin_{\nu \in \mathcal{E}} \{G(\nu) + D_h(\nu,\mu)\}.$
Then, for any $\nu \in \mathcal{E},$ we have $G(\nu) + D_h(\nu,\mu) \geq G(\Bar{\nu}) + D_h(\Bar{\nu},\mu) + D_h(\nu, \Bar{\nu}).$
\end{lemma}
We are ready to state the main result on the convergence of the simultaneous MDA algorithm. 
\begin{theorem}[Convergence of the simultaneous MDA Algorithm \ref{eq:mirror-sim-explicit}]
\label{thm: conv-sim-bregman}
    Let $(\nu^0, \mu^0)$ be such that $\Delta \coloneqq\sup_{\nu \in \mathcal{C}}D_h(\nu, \nu^0) + \sup_{\mu \in \mathcal{D}}D_h(\mu, \mu^0) < \infty.$ Suppose that Assumption \ref{assumption:assump-h}, \ref{def: def-F-conv-conc}, \ref{assumption:bddflatF} hold. If $\tau = \Theta\left(N^{-1/2}\right)\footnote{We say $f(n)=\Theta(g(n))$ if there exists $c_1,c_2, n_0 > 0$ such that $c_1 g(n) \leq f(n) \leq c_2 g(n),$ for all $n \geq n_0$.},$ then there exists a constant $K_\text{sim} \equiv K_\text{sim}(\Delta, C_1, C_2, \alpha) > 0$ such that
    \begin{equation*}
        \operatorname{NI}\left(\frac{1}{N} \sum_{n=0}^{N-1}\nu^{n}, \frac{1}{N}\sum_{n=0}^{N-1}\mu^{n}\right) \leq \frac{K_\text{sim}}{\sqrt{N}}.
    \end{equation*}
\end{theorem}
\begin{remark}[Initialization condition]\label{remark:initialization}
    The initialization requirement in Theorem \ref{thm: conv-sim-bregman}, namely, $\sup_{\nu \in \mathcal{C}} D_h(\nu, \nu^0) + \sup_{\mu \in \mathcal{D}} D_h(\mu, \mu^0) < \infty$ must be verified case by case, depending on the choice of $h$ and the admissible classes $\mathcal{C},\mathcal{D}$. Such verifications for Examples \ref{example:relative-entropy} and \ref{example:chi-squared} are carried out in Lemmas \ref{lemma:verif-init-entropy} and \ref{lemma:verif-init-chi}, respectively.
\end{remark}
\begin{remark}[About the proof of Theorem \ref{thm: conv-sim-bregman}]
\label{remark: proof of thm 2.1}
    In their proof of convergence of the infinite-dimensional MD algorithm for convex $F,$ \cite{korba} show that relative smoothness is sufficient to prove that $F$ is monotonically decreasing along the sequence $(\nu^n)_{n \geq 0}$ generated by MD, i.e., $F(\nu^{n+1}) \leq F(\nu^n),$ for all $n \geq 0.$ The monotonicity property is key to establishing that the MD scheme converges to a minimizer of $F$ with rate $\mathcal{O}(N^{-1}).$ In the case of Algorithm \ref{eq:mirror-sim-explicit}, monotonicity no longer holds, and we therefore prove convergence only for the time-averaged iterates. The corresponding convergence rate is expected to be the slower rate $\mathcal{O}\left(N^{-1/2}\right)$, as suggested by the finite-dimensional \cite[Theorem $5.1$]{bubeck2015convex}. 
    
    Furthermore, in contrast to the proof strategy in \cite{korba}, our argument does not require relative smoothness of $F,$ but only uniform boundedness of its first-order flat derivatives. This is made possible by requiring $h$ to be strongly convex relative to the $\operatorname{TV}^2$ distance, which is stronger than the strict convexity of $h$ used in \cite{korba}, yet still verifiable for typical choices of divergences, see Examples \ref{example:relative-entropy} and \ref{example:chi-squared}. The choice of the $\operatorname{TV}^2$ distance in Assumption \ref{assumption:assump-h} is motivated by the fact that several functions $h,$ as illustrated in Examples \ref{example:relative-entropy} and \ref{example:chi-squared} and noted in \cite[Lemma 3.2]{chizat2023convergence}, satisfy the inequality $D_h(\nu', \nu) \geq \frac{\alpha}{2}\operatorname{TV}^2(\nu',\nu)$. 
    
    By contrast, replacing $\operatorname{TV}^2$ with, for example, the squared $L^2$-Wasserstein distance $\mathcal{W}_2^2(\nu',\nu)$ would reduce the generality of our analysis. Apart from the relative entropy, we are not aware of any divergences satisfying $D_h(\nu', \nu) \geq \frac{\alpha}{2}\mathcal{W}_2^2(\nu',\nu)$, and even this would require a considerably stronger (and difficult to verify) condition that the iterates produced by Algorithm \ref{eq:mirror-sim-explicit} satisfy the Talagrand inequality. Thus, Theorem \ref{thm: conv-sim-bregman} should be viewed as the baseline explicit mean-field MDA guarantee. It requires only first-order boundedness of the flat derivatives but does not exploit the additional information structure of alternating play. The improved rate in Theorem \ref{thm:conv-alt-bregman} comes from analyzing precisely this additional structure.
\end{remark}
\subsection{Sketch of proof of Theorem \ref{thm: conv-sim-bregman}}
    Here, we present the proof sketch of Theorem \ref{thm: conv-sim-bregman}. The full proof is provided in Section \ref{section:proofs}. Applying the Bregman proximal inequality (Lemma \ref{lemma:Bregman-prox-ineq}) to each of the simultaneous mirror steps yields two inequalities: one for the $\nu$-update and one for the $\mu$-update. Each inequality relates the corresponding linearized terms in \eqref{eq:mirror-sim-explicit} to three Bregman terms: $D_h(\nu,\nu^n)$, $D_h(\nu,\nu^{n+1})$ and $D_h(\nu^{n+1},\nu^n),$ and likewise for the $\mu$-update. Combining the linearized terms in \eqref{eq:mirror-sim-explicit} involving $\nu- \nu^n$ and $\mu-\mu^n$ with the convexity-concavity of $F$ via Assumption \ref{def: def-F-conv-conc}, we obtain upper bounds on the differences $F(\nu^n, \mu^n) - F(\nu, \mu^n)$ and $F(\nu^n, \mu) - F(\nu, \mu^n)$, respectively. Assumptions \ref{assumption:assump-h} and \ref{assumption:bddflatF} then convert the remaining linearized terms involving $\nu^{n+1}-\nu^n$, $\mu^{n+1}-\mu^n$ and the Bregman divergences $D_h(\nu^{n+1}, \nu^n)$, $D_h(\mu^{n+1}, \mu^n)$ into quadratic bounds in $\operatorname{TV}$. More precisely, for the $\nu$-part one obtains the expression
    \begin{align*}
        F(\nu^n, \mu^n) - F(\nu, \mu^n) \leq - \frac{\alpha}{2\tau}\left(\operatorname{TV}(\nu^{n+1},\nu^n) - \frac{2\tau C_1}{\alpha}\right)^2 +\frac{2\tau C_1^2}{\alpha} + \frac{1}{\tau} S_n^\nu,
    \end{align*}
    $S_n^\nu := D_h(\nu, \nu^n) - D_h(\nu, \nu^{n+1}),$ and similarly for the $\mu$-part. Combining the two inequalities gives the unified bound
    \begin{align*}
    &F(\nu^n, \mu) - F(\nu, \mu^n) \leq \frac{2\tau}{\alpha}\left(C_1^2 +  C_2^2\right)+ \frac{1}{\tau}(S_n^\nu + S_n^\mu).
    \end{align*}
    Summing over $n=0,...,N-1,$ telescoping $S_n^\nu + S_n^\mu$, dividing by $N$ and using Jensen’s inequality, then yields
    \begin{equation*}
         \operatorname{NI}\left(\frac{1}{N}\sum_{n=0}^{N-1}\nu^{n}, \frac{1}{N}\sum_{n=0}^{N-1}\mu^{n}\right) \leq \frac{2\tau}{\alpha}\left(C_1^2 +  C_2^2\right)  +\frac{1}{N\tau}\Delta.
    \end{equation*}
    Setting $\tau = \Theta\left(N^{-1/2}\right)$ leads to the final bound, establishing the claim. 
\section{Convergence of the alternating MDA Algorithm \texorpdfstring{\ref{eq:mirror-alt}}{1.7}}
\label{subsect:alternating}
The analysis of alternating MDA requires a dual-space viewpoint. In the simultaneous scheme, the error terms can be controlled directly using strong convexity of $h$. Alternating updates, however, produce asymmetric Bregman commutators. In finite-dimensional Euclidean spaces such terms can be handled using self duality of the inner product and derivatives of the dual mirror map on the same space. In the measure setting, the primal variables are signed measures, while the natural dual objects are bounded continuous functions. Therefore, before analyzing the alternating scheme, we introduce the duality pair between signed measures and bounded continuous functions.

Let $\left(\mathcal{M}(\mathcal{X}), \|\cdot\|_{\operatorname{TV}}\right)$ be the Banach space of finite signed measures $\mu$ on $\mathcal{X}$ equipped with the total variation norm $\|\mu\|_{\operatorname{TV}} \coloneqq |\mu|(\mathcal{X}).$ Let $\left(C_b(\mathcal{X}), \|\cdot\|_{\infty}\right)$ be the Banach space of bounded continuous functions from $\mathcal{X} \subseteq \mathbb R^d$ to $\left(\mathbb R, |\cdot|\right),$ where $|\cdot|$ is the Euclidean norm. For any $\left(f,m\right) \in C_b(\mathcal{X}) \times \mathcal{M}(\mathcal{X}),$ we define the duality pairing $\langle \cdot, \cdot \rangle: C_b(\mathcal{X}) \times \mathcal{M}(\mathcal{X}) \to \mathbb R$ by 
\begin{equation}
\label{eq:duality-pairing}
    \langle f, m \rangle \coloneqq \int_{\mathcal{X}} f(x)m(\mathrm{d}x). 
\end{equation}
Next, we define the notion of convex conjugate of $h:\mathcal{M}(\mathcal{X}) \to \mathbb R$ relative to the duality pairing \eqref{eq:duality-pairing}.
\begin{definition}[Convex conjugate]
\label{def: convex-conjugate}
    Let $h:\mathcal{M}(\mathcal{X}) \to \mathbb R$ be a function. Then the map $h^*: C_b(\mathcal{X}) \to \mathbb R$ given by $h^*(f) \coloneqq \sup_{m \in \mathcal{M}(\mathcal{X})} \left\{\langle f, m \rangle - h(m)\right\}$
is called the \textit{convex conjugate} of $h.$
\end{definition}
Regardless of the convexity of $h$, it follows from \cite[Theorem $2.112$]{Bonnans2000PerturbationAO} that $h^*$ is convex on $C_b(\mathcal{X}),$ i.e., for all $\lambda \in [0,1]$ and all $f',f \in C_b(\mathcal{X}),$ we have that $h^*\left((1-\lambda)f + \lambda f'\right) \leq (1-\lambda)h^*(f) + \lambda h^*(f').$ In Examples \ref{example:entropy1} and \ref{example:chi1}, we provide in Examples \ref{example:relative-entropy} and \ref{example:chi-squared} the explicit form of $h^*$ when $h$ is the relative entropy and the chi-squared divergence, respectively. Analogous to the characterization of the convexity of $h$ on $\mathcal{P}(\mathcal{X})$ via flat derivatives, we can characterize the convexity of $h^*$ on $C_b(\mathcal{X})$ using its Fréchet derivative (cf. Definition \ref{definition:frechet-diff}). 

We say $h^*:C_b(\mathcal{X}) \to \mathbb R$ is \textit{Fréchet convex} if for any $f,f' \in C_b(\mathcal{X}),$ we have $h^*(f') - h^*(f) \geq \nabla_\mathcal{F} h^*(f)[f'-f]$. As shown in Examples \ref{example:first-Frechet-entropy} and \ref{example:first-Frechet-chi}, when $h$ is chosen as the relative entropy or the chi-squared divergence, its convex conjugate $h^*$ admits the Fréchet derivatives $\nabla_\mathcal{F} h^*(f).$

Furthermore, using the Fréchet characterization of convexity, we can define the Bregman divergence between $f$ and $f'$ and its commutator on the convex subset $\mathcal{U}$ of the dual space $C_b(\mathcal{X})$.
\begin{definition}[Dual Bregman divergence and commutator]
\label{def:dual-Bregman}
    Let $h^*:C_b(\mathcal{X}) \to \mathbb R$ be the convex conjugate of $h.$ Assume $h^*$ is Fréchet differentiable on $\mathcal{U}.$ The \textit{dual $h^*$-Bregman divergence} and its commutator are the maps $D_{h^*}, \mathrm{Com}^*_{h^*}: \mathcal{U} \times \mathcal{U} \to \mathbb R$, respectively, given by
    \begin{align*}
        &D_{h^*}(f',f) := h^*(f') - h^*(f) - \nabla_\mathcal{F} h^*(f)[f'-f],\\ 
        &\mathrm{Com}^*_{h^*}(f',f):= D_{h^*}(f',f)-  D_{h^*}(f,f').
        \end{align*}
\end{definition}
Before stating the next assumption on the dual space, we introduce two additional assumptions on $F$. These conditions are required only for the analysis of the alternating scheme, as they allow us to control the extra asymmetric terms arising from the updates in \eqref{eq:mirror-alt}. 
\begin{assumption}[Relative smoothness of $F$]
\label{def:relative-smoothness}
     Assume the function $F$ is $(L_{\nu},L_\mu)$-smooth relative to $h,$ i.e., there exist $L_{\nu}, L_{\mu} > 0$ such that, for any $\nu, \nu' \in \mathcal{C}$ and $\mu, \mu' \in \mathcal{D},$ we have
    \begin{align*}
    \label{smooth-nu}
       D_{F(\cdot,\mu)}(\nu',\nu)  \leq L_{\nu}D_h(\nu', \nu), \
        D_{F(\nu,\cdot)}(\mu',\mu) \geq  - L_{\mu}D_h(\mu', \mu).
    \end{align*}
\end{assumption}
In Proposition \ref{prop:verification-GANs}, \ref{prop:nlhf-assumptions} and \ref{prop:verification-adversarial}, we verify that Assumptions \ref{def:relative-smoothness} is satisfied by the examples in Section \ref{example:GAN-example}. 

The following assumption on the dual mirror map will turn out to be crucial for showing the improvement in the convergence rate of Algorithm \ref{eq:mirror-alt} compared to the simultaneous algorithm.
\begin{assumption}[Third-order smoothness of the dual mirror map]
\label{assumption:lipschitz-h^*2}
    Suppose that $\left(C_b(\mathcal{X}) \times C_b(\mathcal{X}) \times C_b(\mathcal{X})\right) \ni (g,g,g) \mapsto \nabla_\mathcal{F}^3 h^*(f)[g,g,g] \in \mathbb R$ is uniformly bounded, i.e, there exists $L_{h^*} > 0$ such that for all $g \in C_b(\mathcal{X})$, 
    \begin{equation*}
        \left|\nabla_\mathcal{F}^3 h^*(f)[g,g,g]\right| 
        \leq L_{h^*}\|g\|^3_{\infty}.
    \end{equation*}
\end{assumption}
Assumption \ref{assumption:lipschitz-h^*2} is used only in the analysis of the alternating scheme. Its role is to quantify how far the dual mirror geometry is from being locally symmetric. Indeed, alternating updates create Bregman commutators of the form $D_{h^*}(g',g)-D_{h^*}(g,g')$. If the geometry were locally symmetric, this difference would vanish as for $\chi^2$-divergence (see Proposition \ref{eq:verification-assumption-chi}). Assumption \ref{assumption:lipschitz-h^*2} ensures that the remaining asymmetry is cubic in $\|g-g'\|_{\infty}$. Concretely, the Bregman commutators admit the cubic bound $\left|D_{h^*}(g',g)-D_{h^*}(g,g')\right|\leq cL_{h^*}\|g'-g\|_\infty^3$,
for some constant $c>0$. Since the mirror update gives $g-g' = \mathcal{O}(\tau)$, the bound becomes \(\mathcal{O}(\tau^3)\) along the iterates. 

It is also worth comparing Assumption \ref{assumption:lipschitz-h^*2} with the notion of self-concordance in convex analysis (see, e.g., \cite[Chapter 2]{nesterov1994}). Self-concordance provides a local control of third-order derivatives in terms of the Hessian, namely, it requires that for all $g \in C_b(\mathcal{X})$, 
\[
\left|\nabla_{\mathcal{F}}^3 h^*(f)[g,g,g]\right| \le 2\left(\nabla_{\mathcal{F}}^2 h^*(f)[g,g]\right)^{3/2}.
\]
Importantly, Assumption \ref{assumption:lipschitz-h^*2} is not an artificial technical condition since it holds for two standard mirror geometries, namely relative entropy and the $\chi^2$-divergence. In Propositions \ref{eq:verification-assumption} and \ref{eq:verification-assumption-chi}, we verify Assumption \ref{assumption:lipschitz-h^*2} in the case where $h$ is the relative entropy and the $\chi^2$-divergence, respectively. On the other hand, self-concordance holds only for the $\chi^2$ mirror map as Proposition \ref{eq:verification-assumption-chi} immediately shows. In Example \ref{example:self-concordance-fail}, we explain why self-concordance may fail for the entropy mirror map. In this case, it is possible to prove Theorem \ref{thm:conv-alt-bregman} for the $\chi^2$ mirror map only by assuming self-concordance. 

Now, we are ready to state the second main result of the paper.  
\begin{theorem}[Convergence of the alternating MDA Algorithm \ref{eq:mirror-alt}]
\label{thm:conv-alt-bregman}
Let $(\nu^0, \mu^0)$ be such that $\Delta \coloneqq \sup_{\nu \in \mathcal{C}} D_h(\nu, \nu^0) + \sup_{\mu \in \mathcal{D}} D_h(\mu, \mu^0) < \infty$ (cf.\ Remark \ref{remark:initialization}). Let Assumptions \ref{assumption:assump-h}, \ref{def: def-F-conv-conc}, \ref{assumption:bddflatF}, \ref{def:relative-smoothness} and \ref{assumption:lipschitz-h^*2} hold. Suppose that $\tau L \leq \frac{1}{2},$ with $L \coloneqq \max\{L_{\nu}, L_{\mu}\}$. If $\tau = \Theta\left(N^{-1/3}\right)$, then there exists a constant $K_{\text{alt}} \equiv K_{\text{alt}}(\Delta, C_1, C_2, \alpha, L, L_h^*, |F(\nu^0,\mu^0)|) > 0$ such that
\begin{equation}
\label{eq:seq-conv-NI}
    \operatorname{NI}\left(\frac{1}{N}\sum_{n=0}^{N-1} \nu^{n+1}, \frac{1}{N}\sum_{n=0}^{N-1} \mu^n\right) \leq \frac{K_\text{alt}}{N^{2/3}}.
\end{equation}
\end{theorem}
\begin{remark}[Bilinear games]
\label{remark:bilinear-games}
In particular, if $F(\nu, \mu) = \iint f(x,y) \nu(\mathrm{d}x)\mu(\mathrm{d}y),$ for a bounded function $f$, then Assumptions \ref{def: def-F-conv-conc}, \ref{assumption:bddflatF}, \ref{def:relative-smoothness} are satisfied and in Assumption \ref{def:relative-smoothness} we have $L_{\nu} = L_{\mu} = 0.$ Therefore, $L = 0$ in \eqref{eq:seq-conv-NI}, and hence Theorem \ref{thm:conv-alt-bregman} is consistent with the already known rate $\mathcal{O}\left({N^{-2/3}}\right)$ of the MDA algorithm in finite-dimensional constrained bilinear min-max games; see \cite[Theorem $3.2$ and Corollary $3.3$]{wibisono2022alternating}. Since we work in an infinite-dimensional setting with a non-linear convex-concave objective function $F,$ Theorem \ref{thm:conv-alt-bregman} substantially generalizes the results of \cite{wibisono2022alternating}. Moreover, the proof also differs substantially from the Euclidean argument because the asymmetric alternating terms must be transported from the primal measure space to the dual function space before Assumption \ref{assumption:lipschitz-h^*2} can be applied.
\end{remark}
\subsection{Sketch of proof of Theorem \ref{thm:conv-alt-bregman}}
\label{subsec:proof-sketch-alt}
The proof of Theorem \ref{thm:conv-alt-bregman} is where the main infinite-dimensional difficulty appears. The alternating update creates a non-symmetric error term that has no analogue in the simultaneous analysis. In finite-dimensional Euclidean space, the corresponding term can be controlled using the self-duality of the inner product and the third derivative of the mirror map. In the measure space setting, one must first identify the correct dual object and prove that the asymmetric terms can be represented as dual Bregman commutators on $C_b(\mathcal{X})$. This commutator transport identity in Lemma \ref{lemma:primal-dual-iterates} is the key technical step that makes Assumption \ref{assumption:lipschitz-h^*2} applicable. 

Using Assumption \ref{def:relative-smoothness}, we combine the asymmetric term $F(\nu^{n+1}, \mu^n) - F(\nu^{n}, \mu^n)$ with $\int_{\mathcal{X}} \frac{\delta F}{\delta \nu}(\nu^n, \mu^n, x)(\nu^n - \nu^{n+1})(\mathrm{d}x)$ and $\int_{\mathcal{X}} \frac{\delta F}{\delta \mu}(\nu^{n+1}, \mu^n, y)(\mu^{n+1} - \mu^{n})(\mathrm{d}y),$ which yields the Bregman commutator $\text{Com}_{\nu^n} := D_h(\nu^n, \nu^{n+1}) - D_h(\nu^{n+1}, \nu^{n})$, with $\text{Com}_{\mu^n}$ being defined analogously. Proceeding as in the simultaneous case leads to
\begin{align*}
    \operatorname{NI}\left(\frac{1}{N}\sum_{n=0}^{N-1} \nu^{n+1}, \frac{1}{N}\sum_{n=0}^{N-1} \mu^n\right) 
    \leq \frac{\Delta}{N\tau} + \mathcal{O}(\tau^2)+ \mathcal{O}\left(\frac{1}{N}\right) + \frac{1}{2N\tau}\sum_{n=0}^{N-1}\left(\text{Com}_{\nu^n} + \text{Com}_{\mu^n}\right),
\end{align*}
where the $\mathcal{O}\left(\frac{1}{N}\right)$ term uses the uniform boundedness of $F$ proved in Lemma \ref{lem:boundedness-from-flat-derivatives}. Let $f^n := \frac{\delta h}{\delta \nu}(\nu^{n}, \cdot)$, $g^n := \frac{\delta h}{\delta \mu}(\mu^{n}, \cdot)$ and define the dual Bregman commutator $\text{Com}^*_{\nu^n} := D_{h^*}\left(f^{n+1}, f^n\right) - D_{h^*}\left(f^n, f^{n+1}\right)$, with $\text{Com}^*_{\mu^n}$ being defined analogously. We transport the commutators from the primal space to the dual space via Lemma \ref{lemma:primal-dual-iterates} and obtain 
\begin{align*}
   \text{Com}_{\nu^n} = \text{Com}^*_{\nu^n}, \ \text{Com}_{\mu^n} = \text{Com}^*_{\mu^n}.
\end{align*} 
Using the third-order Fréchet derivative of $h^*$ (Definition \ref{def:third-Frechet}) and its uniform boundedness (Assumption \ref{assumption:lipschitz-h^*2}), we obtain $\text{Com}^*_{\nu^n} \leq \frac{L_{h^*}}{4}\left\|f^{n+1}- f^n\right\|^3_{\infty}$. By Proposition \ref{prop:foc}, the first-order optimality condition for the $\nu$-update yields (up to an additive constant) $f^{n+1}-f^n = -\tau \frac{\delta F}{\delta \nu}(\nu^{n}, \mu^n, \cdot)$. Hence, by Assumption \ref{assumption:bddflatF}, 
\begin{align*}
     \text{Com}^*_{\nu^n} \leq \frac{L_{h^*}}{4} \tau^3 C_1^3,\quad
     \text{Com}^*_{\mu^n} \leq \frac{L_{h^*}}{4} \tau^3 C_2^3,
\end{align*}
where the second inequality follows from applying the same argument to the $\mu$-commutator. Therefore,
\begin{align*}
    \operatorname{NI}\left(\frac{1}{N}\sum_{n=0}^{N-1} \nu^{n+1}, \frac{1}{N}\sum_{n=0}^{N-1} \mu^n\right) \leq \frac{\Delta}{N\tau} + \mathcal{O}\left(\frac{\tau^3}{\tau}\right) + \mathcal{O}(\tau^2) + \mathcal{O}\left(\frac{1}{N}\right).
\end{align*}
Choosing $\tau = \Theta\left(N^{-1/3}\right)$ and noting that $N^{-1} \leq N^{-2/3}$, for all $N \geq 1$, yields the conclusion.

\section*{Acknowledgments and Disclosure of Funding}
We thank the anonymous reviewers for their insightful comments, which helped improve the paper. R-AL acknowledges the support of the EPSRC Centre for Doctoral Training in Mathematical Modelling, Analysis and Computation (MAC-MIGS), funded by the UK Engineering and Physical Sciences Research Council (grant EP/S023291/1), Heriot-Watt University, and the University of Edinburgh. This work was partly carried out while R-AL was a PhD student at Heriot-Watt University. The paper was finalized during R-AL's postdoctoral appointment at the RIKEN Center for Advanced Intelligence Project (RIKEN AIP). R-AL also acknowledges support from JSPS KAKENHI Grant Number JP26K25544 for travel to NeurIPS 2026 to present this work. \L{}S acknowledges the support of the UKRI Prosperity Partnership Scheme (FAIR) under EPSRC Grant EP/V056883/1 and the Alan Turing Institute.

\bibliographystyle{abbrv}
\bibliography{referencesFP} 

\newpage
\appendix
\section*{Appendix}
The appendix is organized into several sections (\ref{section:proofs}--\ref{app:further-works}), each providing supporting technical material, extended proofs and further context for the main text. Section \ref{section:proofs} contains the complete proofs of the paper’s two main results, namely Theorem \ref{thm: conv-sim-bregman} and Theorem \ref{thm:conv-alt-bregman}. Section \ref{appendix:additional-results} contains proofs for supplementary lemmas and propositions referenced earlier in the paper. In Section \ref{appendix:verification-GAN}, we check that Assumptions \ref{def: def-F-conv-conc}, \ref{assumption:bddflatF} and \ref{def:relative-smoothness} hold in for the examples in Section \ref{example:GAN-example}. In Section \ref{section:numerical-example} we describe all implementation details, experimental settings and parameter choices. Section \ref{appendix:B} collects standard definitions concerning differentiability with respect to probability measures that are used throughout the paper. Section \ref{appendix:differentiability-dual-space} complements the previous section by treating differentiability in the dual space and supporting the duality arguments used in the proof of Theorem \ref{thm:conv-alt-bregman}. Section \ref{appendix:technical-duality} gathers duality identities used particularly in the proof of Theorem \ref{thm:conv-alt-bregman}. Section \ref{sec:proof-implicit-game} shows that an implicit MDA scheme achieves the same sublinear convergence rate as the continuous-time dynamics under convex-concave assumptions on $F$. Finally, Section \ref{app:further-works} collects references and related literature not included in the main text.

\section{Proofs of Theorem \ref{thm: conv-sim-bregman} and Theorem \ref{thm:conv-alt-bregman}}\label{section:proofs}
This section is dedicated to the proofs of the main results, namely Theorem \ref{thm: conv-sim-bregman} and Theorem \ref{thm:conv-alt-bregman}. Before proving Theorem \ref{thm: conv-sim-bregman}, we summarize the assumptions on $h$ and $F$, outlining how each applies to the simultaneous and alternating schemes.

\subsection{Summary of assumptions and applicability}
\label{app:assumption-summary}
The assumptions distinguish between properties of the mirror geometry and properties of the payoff. Both the simultaneous and alternating schemes require a differentiable mirror map that is strongly convex in
total variation (Assumption \ref{assumption:assump-h}), together with a convex-concave payoff with bounded flat derivatives (Assumptions \ref{def: def-F-conv-conc} and \ref{assumption:bddflatF}). The alternating scheme additionally requires relative smoothness of the payoff with respect to the mirror map (Assumption \ref{def:relative-smoothness}) and a bounded third Fr\'echet derivative of the dual mirror map (Assumption \ref{assumption:lipschitz-h^*2}). Both analyses also require finite initial Bregman divergences. Table~\ref{tab:assumption-summary} summarizes these conditions and their verification in the settings considered in this paper.

\begin{table}[htbp]
\centering
\small
\renewcommand{\arraystretch}{1.25}
\setlength{\tabcolsep}{4pt}
\begin{tabular}{@{}p{0.12\linewidth}p{0.36\linewidth}p{0.15\linewidth}p{0.29\linewidth}@{}}
\hline
\textbf{Component}
&
\textbf{Condition}
&
\textbf{Scheme}
&
\textbf{Verified settings}
\\
\hline
Mirror
&
Assumption \ref{assumption:assump-h}
&
Both
&
Relative entropy and $\chi^2$
\\
&
Assumption \ref{assumption:lipschitz-h^*2}
&
Alternating
&
Relative entropy and $\chi^2$
\\
\hline
Payoff
&
Assumptions \ref{def: def-F-conv-conc} and \ref{assumption:bddflatF}
&
Both
&
Propositions \ref{prop:verification-GANs}, \ref{prop:nlhf-assumptions}, \ref{prop:verification-adversarial}
\\
&
Assumption \ref{def:relative-smoothness}
&
Alternating
&
Propositions \ref{prop:verification-GANs}, \ref{prop:nlhf-assumptions}, \ref{prop:verification-adversarial}
\\
\hline
Initial data
&
Finite initial Bregman divergences
&
Both
&
Lemmas~\ref{lemma:verif-init-entropy} and \ref{lemma:verif-init-chi}
\\
\hline
\end{tabular}
\caption{Summary of the assumptions and their applicability.
The listed settings satisfy the corresponding conditions on the
admissible domains and under the hypotheses specified in the
respective results.}
\label{tab:assumption-summary}
\end{table}

\paragraph{Mirror geometries.}
The relative entropy and $\chi^2$ mirror geometries considered here both satisfy the TV-strong convexity condition in Assumption~\ref{assumption:assump-h} and the dual third-order condition in Assumption~\ref{assumption:lipschitz-h^*2}. For the $\chi^2$ geometry, the third derivative of $h^*$ vanishes on the dual domain.

\paragraph{Bounded bilinear payoffs.}
Consider $F(\nu,\mu)=\int_{\mathcal X}\int_{\mathcal X}f(x,y)\nu(\mathrm dx)\mu(\mathrm dy),$ where $f$ is a bounded measurable kernel. The payoff is affine in each argument and therefore convex-concave. Its flat derivatives are uniformly bounded by $\|f\|_\infty$. Moreover, affinity in each argument makes the corresponding first-order expansion exact. Thus the separate-variable relative smoothness conditions hold with $L_\nu=L_\mu=0.$ Hence, for bilinear objectives, the payoff assumptions reduce to boundedness of the kernel. This applies to the lifted GAN objective in Proposition~\ref{prop:verification-GANs} and the NLHF objective in Proposition~ \ref{prop:nlhf-assumptions}.

\paragraph{Nonlinear payoffs.}
The payoff assumptions do not require bilinearity. Proposition~\ref{prop:verification-adversarial} verifies them for the nonlinear adversarial training problem involving a mean-field neural network. The conditions stated there include a compact data domain, a bounded activation, clipping of the output layer, and a smoothed total variation penalty.

\paragraph{Initialization.}
Finiteness of the initial Bregman divergences is a separate requirement on the initial measures and the chosen mirror geometry. Lemmas~\ref{lemma:verif-init-entropy} and \ref{lemma:verif-init-chi} verify this condition for the relative entropy and $\chi^2$ geometries, respectively.

\subsection{Proof of Theorem \ref{thm: conv-sim-bregman}}
\begin{proof}[Proof of Theorem \ref{thm: conv-sim-bregman}]
Since $\nu \mapsto \tau \int_\mathcal{X} \frac{\delta F}{\delta \nu}(\nu^n,\mu^n,x)(\nu-\nu^n)(\mathrm{d}x)$ is convex, applying Lemma \ref{lemma:Bregman-prox-ineq} with $\Bar{\nu} = \nu^{n+1}$ and $\mu = \nu^n$ implies that, for any $\nu \in \mathcal{C},$ we have
\begin{multline*}
    \tau \int_\mathcal{X} \frac{\delta F}{\delta \nu}(\nu^n,\mu^n,x)(\nu-\nu^n)(\mathrm{d}x) + D_h(\nu, \nu^n) \geq \tau \int_\mathcal{X} \frac{\delta F}{\delta \nu}(\nu^n,\mu^n,x)(\nu^{n+1}-\nu^n)(\mathrm{d}x)\\ + D_h(\nu^{n+1}, \nu^n) + D_h(\nu, \nu^{n+1}),
\end{multline*}
or, equivalently,
\begin{multline}
\label{eq: bregpp-nu}
    -\tau \int_\mathcal{X} \frac{\delta F}{\delta \nu}(\nu^n,\mu^n,x)(\nu-\nu^n)(\mathrm{d}x) - D_h(\nu, \nu^n) \leq -\tau \int_\mathcal{X} \frac{\delta F}{\delta \nu}(\nu^n,\mu^n,x)(\nu^{n+1}-\nu^n)(\mathrm{d}x)\\ - D_h(\nu^{n+1}, \nu^n) - D_h(\nu, \nu^{n+1}).
\end{multline}
Similarly, since $\mu \mapsto -\tau \int_\mathcal{X} \frac{\delta F}{\delta \mu}(\nu^n,\mu^n,y)(\mu-\mu^n)(\mathrm{d}y)$ is convex, applying Lemma \ref{lemma:Bregman-prox-ineq} with $\Bar{\nu} = \mu^{n+1}$ and $\mu = \mu^n$ implies that, for any $\mu \in \mathcal{D},$ we have 
\begin{multline}
\label{eq: bregpp-mu}
    \tau \int_\mathcal{X} \frac{\delta F}{\delta \mu}(\nu^n,\mu^n,y)(\mu-\mu^n)(\mathrm{d}y) - D_h(\mu, \mu^n) \leq \tau \int_\mathcal{X} \frac{\delta F}{\delta \mu}(\nu^n,\mu^n,y)(\mu^{n+1}-\mu^n)(\mathrm{d}y)\\ 
    - D_h(\mu^{n+1}, \mu^n) - D_h(\mu, \mu^{n+1}).
\end{multline}
Using the convexity of $\nu \mapsto F(\nu,\mu)$ in \eqref{eq: bregpp-nu}, with $\nu= \nu^n$ and $\mu = \mu^n,$ we have that
\begin{multline}
\label{eq:eq3}
    F(\nu^n, \mu^n) - F(\nu, \mu^n) - \frac{1}{\tau}D_h(\nu, \nu^n) \leq \int_\mathcal{X} \frac{\delta F}{\delta \nu}(\nu^n,\mu^n,x)(\nu^{n}-\nu^{n+1})(\mathrm{d}x)\\ - \frac{1}{\tau}D_h(\nu^{n+1}, \nu^n) - \frac{1}{\tau} D_h(\nu, \nu^{n+1}).
\end{multline}
From Assumptions \ref{assumption:assump-h} and \ref{assumption:bddflatF}, it follows from \eqref{eq:eq3} that
\begin{equation}
    \begin{aligned}
    \label{eq:Lnu-smooth}
        F(\nu^n, \mu^n) &- F(\nu, \mu^n) - \frac{1}{\tau}D_h(\nu, \nu^n)\\ 
        &\leq 2C_1\operatorname{TV}(\nu^{n+1},\nu^n)- \frac{\alpha}{2\tau}\operatorname{TV}^2(\nu^{n+1}, \nu^n) - \frac{1}{\tau} D_h(\nu, \nu^{n+1})\\
        &= - \frac{\alpha}{2\tau}\left(\operatorname{TV}(\nu^{n+1},\nu^n) - \frac{2\tau C_1}{\alpha}\right)^2 +\frac{2\tau C_1^2}{\alpha}- \frac{1}{\tau} D_h(\nu, \nu^{n+1})\\
        &\leq \frac{2\tau C_1^2}{\alpha}- \frac{1}{\tau} D_h(\nu, \nu^{n+1}),
    \end{aligned}
    \end{equation}
    where the equality follows from the standard identity $-(a-b)^2 + b^2 = -a^2 +2ab.$

Similarly, using concavity of $\mu \mapsto F(\nu, \mu)$ in \eqref{eq: bregpp-mu}, with $\nu= \nu^n$ and $\mu = \mu^n,$
we have that
\begin{multline}
\label{eq:eq4}
    F(\nu^n, \mu) - F(\nu^n, \mu^n) - \frac{1}{\tau}D_h(\mu, \mu^n) \leq \int_\mathcal{X} \frac{\delta F}{\delta \mu}(\nu^n,\mu^n,y)(\mu^{n+1}-\mu^n)(\mathrm{d}y)\\ - \frac{1}{\tau}D_h(\mu^{n+1}, \mu^n) - \frac{1}{\tau}D_h(\mu, \mu^{n+1}).
\end{multline}
From Assumptions \ref{assumption:assump-h} and \ref{assumption:bddflatF}, it follows from \eqref{eq:eq4} that
\begin{equation}
\begin{aligned}
\label{eq:Lmu-smooth}
    F(\nu^n, \mu) &- F(\nu^n, \mu^n) - \frac{1}{\tau}D_h(\mu, \mu^n)\\ &\leq 2C_2\operatorname{TV}(\mu^{n+1},\mu^n) - \frac{\alpha}{2\tau}\operatorname{TV}^2(\mu^{n+1}, \mu^n) - \frac{1}{\tau}D_h(\mu, \mu^{n+1})\\
    &= - \frac{\alpha}{2\tau}\left(\operatorname{TV}(\mu^{n+1},\mu^n) - \frac{2\tau C_2}{\alpha}\right)^2 +\frac{2\tau C_2^2}{\alpha}- \frac{1}{\tau} D_h(\mu, \mu^{n+1})\\
    &\leq \frac{2\tau C_2^2}{\alpha}- \frac{1}{\tau} D_h(\mu, \mu^{n+1}).
\end{aligned}
\end{equation}
Adding inequalities \eqref{eq:Lnu-smooth} and \eqref{eq:Lmu-smooth} implies that for any $(\nu, \mu) \in \mathcal{C} \times \mathcal{D}$ we have
\begin{equation}
\begin{aligned}
\label{eq:NI-ineq}
     &F(\nu^n, \mu) - F(\nu, \mu^n)\\ 
     &\leq \frac{2\tau}{\alpha}\left(C_1^2 +  C_2^2\right) 
    + \frac{1}{\tau}D_h(\nu, \nu^n) + \frac{1}{\tau}D_h(\mu, \mu^n) - \frac{1}{\tau} D_h(\nu, \nu^{n+1}) - \frac{1}{\tau}D_h(\mu, \mu^{n+1}).
\end{aligned}
\end{equation}
Summing the previous inequality over $n=0,1,...,N-1,$ using $D_h(\nu, \nu^{N}) + D_h(\mu, \mu^{N}) \geq 0,$ for any $(\nu, \mu) \in \mathcal{C} \times \mathcal{D},$ bounding the right-hand from above by its supremum over $(\nu, \mu) \in \mathcal{C} \times \mathcal{D},$ and dividing by $N$ gives
\begin{equation}
\label{eq:3.9}
    \frac{1}{N}\sum_{n=0}^{N-1} \left(F(\nu^n, \mu) - F(\nu, \mu^n)\right) \leq \frac{2\tau}{\alpha}\left(C_1^2 +  C_2^2\right)  +\frac{1}{N\tau} \left(\sup_{\nu \in \mathcal{C}}D_h(\nu, \nu^0) + \sup_{\mu \in \mathcal{D}} D_h(\mu, \mu^0)\right).
\end{equation}
Since $\nu \mapsto F(\nu,\mu)$ and $\mu \mapsto -F(\nu, \mu)$ are convex, it follows by Jensen's inequality that
\begin{multline}
\label{eq:NI-Jensen}
    \frac{1}{N}\sum_{n=0}^{N-1} \left(F(\nu^n, \mu) - F(\nu, \mu^n)\right) = \frac{1}{N}\sum_{n=0}^{N-1} F(\nu^n, \mu) - \frac{1}{N}\sum_{n=0}^{N-1}F(\nu, \mu^n)\\ 
    \geq F\left(\frac{1}{N}\sum_{n=0}^{N-1}\nu^n, \mu\right) - F\left(\nu, \frac{1}{N}\sum_{n=0}^{N-1}\mu^n\right).
\end{multline}
Combining \eqref{eq:3.9} with \eqref{eq:NI-Jensen} and taking maximum over $(\nu, \mu)$ gives
\begin{equation*}
    \operatorname{NI}\left(\frac{1}{N}\sum_{n=0}^{N-1} \nu^n, \frac{1}{N}\sum_{n=0}^{N-1} \mu^n\right) \leq \frac{2\tau}{\alpha}\left(C_1^2 +  C_2^2\right)  +\frac{1}{N\tau}\left(\sup_{\nu \in \mathcal{C}} D_h(\nu, \nu^0) + \sup_{\mu \in \mathcal{D}} D_h(\mu, \mu^0)\right).
\end{equation*}
Minimizing the right-hand side over $\tau$ amounts to taking $$\tau = \sqrt{\frac{\alpha\left(\sup_{\nu \in \mathcal{C}} D_h(\nu, \nu^0) + \sup_{\mu \in \mathcal{D}} D_h(\mu, \mu^0)\right)}{2N\left(C_1^2 + C_2^2\right)}},$$ and hence we obtain
\begin{equation*}
    \operatorname{NI}\left(\frac{1}{N}\sum_{n=0}^{N-1} \nu^n, \frac{1}{N}\sum_{n=0}^{N-1} \mu^n\right) \leq 2\sqrt{\frac{2\left(C_1^2 +  C_2^2\right)\left(\sup_{\nu \in \mathcal{C}} D_h(\nu, \nu^0) + \sup_{\mu \in \mathcal{D}} D_h(\mu, \mu^0)\right)}{\alpha N}}.
\end{equation*}
\end{proof}
\subsection{Proof of Theorem \ref{thm:conv-alt-bregman}}
Before we proceed with the proof of Theorem \ref{thm:conv-alt-bregman}, we will need two auxiliary results, which will turn out to be essential. The proofs of Lemmas \ref{lemma:tau2-bound-primal} and \ref{lem:boundedness-from-flat-derivatives} are given in Appendix \ref{appendix:additional-results}.
\begin{lemma}
\label{lemma:tau2-bound-primal}
    Let Assumptions \ref{assumption:assump-h}, \ref{assumption:bddflatF} and  \ref{def:relative-smoothness} hold. Suppose that $\tau L \leq \frac{1}{2},$ with $L \coloneqq \max\{L_{\nu}, L_{\mu}\}.$ Then, for both Algorithms \ref{eq:mirror-sim-explicit} and \ref{eq:mirror-alt}, it holds, for all $n \geq 0,$ that
    \begin{equation*}
        D_h(\nu^{n+1}, \nu^n) \leq \frac{32 \tau^2C_1^2}{\alpha} \quad \text{ and } \quad D_h(\mu^{n+1}, \mu^n) \leq \frac{32 \tau^2C_2^2}{\alpha}.
    \end{equation*}
\end{lemma}
\begin{lemma}[Lipschitz continuity and uniform boundedness of $F$]
\label{lem:boundedness-from-flat-derivatives}
Let $\mathcal{X} \subseteq \mathbb R^d$ and let $\mathcal{C}, \mathcal{D}\subseteq \mathcal{P}(\mathcal{X})$ be nonempty and convex.
Assume that $F \in \mathfrak C^1(\mathcal{C}\times \mathcal{D})$ and that Assumption \ref{assumption:bddflatF} holds. Then for all $\nu,\nu'\in\mathcal{C}$ and $\mu,\mu'\in\mathcal{D}$, $F$ is $\operatorname{TV}$-Lipschitz, i.e.,
\[
\big|F(\nu',\mu')-F(\nu,\mu)\big| \le 2\max\{C_1,C_2\}\left(\operatorname{TV}(\nu',\nu)+\operatorname{TV}(\mu',\mu)\right).
\]
In particular, for $(\nu^0,\mu^0)\in\mathcal{C}\times\mathcal{D}$ such that $F(\nu^0,\mu^0) <  \infty$,
\[
\sup_{(\nu,\mu)\in\mathcal{C}\times\mathcal{D}} |F(\nu,\mu)| \le |F(\nu^0,\mu^0)| +2 (C_1 +C_2),
\]
so $F$ is uniformly bounded on $\mathcal{C} \times \mathcal{D}$.
\end{lemma}
\begin{proof}[Proof of Theorem \ref{thm:conv-alt-bregman}]
We start the proof by following the same calculations from Theorem \ref{thm: conv-sim-bregman}. For \eqref{eq:mirror-alt}, after applying Lemma \ref{lemma:Bregman-prox-ineq} and using convexity-concavity of $F,$ \eqref{eq:eq3} remains unchanged, i.e., 
\begin{align*}
    F(\nu^n, \mu^n) - F(\nu, \mu^n) - \frac{1}{\tau}D_h(\nu, \nu^n) &\leq \int_\mathcal{X} \frac{\delta F}{\delta \nu}(\nu^n,\mu^n,x)(\nu^{n}-\nu^{n+1})(\mathrm{d}x)\\ 
    &- \frac{1}{\tau}D_h(\nu^{n+1}, \nu^n) - \frac{1}{\tau} D_h(\nu, \nu^{n+1}),
\end{align*}
while \eqref{eq:eq4} becomes
\begin{align*}
    F(\nu^{n+1}, \mu) - F(\nu^{n+1}, \mu^n) - \frac{1}{\tau}D_h(\mu, \mu^n) &\leq \int_\mathcal{X} \frac{\delta F}{\delta \mu}(\nu^{n+1},\mu^n,y)(\mu^{n+1}-\mu^n)(\mathrm{d}y)\\ 
    &- \frac{1}{\tau}D_h(\mu^{n+1}, \mu^n) - \frac{1}{\tau}D_h(\mu, \mu^{n+1}).
\end{align*}

Adding the previous two inequalities, summing the resulting inequality over $n=0,1,...,N-1,$ dividing by $N,$ using \eqref{eq:NI-Jensen} and taking maximum over $(\nu, \mu)$ we arrive at
\begin{multline}
\label{eq:start-estimate}
     \operatorname{NI}\left(\frac{1}{N}\sum_{n=0}^{N-1} \nu^{n+1}, \frac{1}{N}\sum_{n=0}^{N-1} \mu^n\right) \leq \frac{1}{N} \sum_{n=0}^{N-1} \Bigg(\int_{\mathcal{X}} \frac{\delta F}{\delta \nu}(\nu^n, \mu^n, x)(\nu^n - \nu^{n+1})(\mathrm{d}x)\\ 
    + \int_{\mathcal{X}} \frac{\delta F}{\delta \mu}(\nu^{n+1}, \mu^n, y)(\mu^{n+1} - \mu^{n})(\mathrm{d}y)\Bigg) + \frac{1}{N\tau}\left(\sup_{\nu \in \mathcal{C}} D_h(\nu, \nu^0) + \sup_{\mu \in \mathcal{D}} D_h(\mu, \mu^0)\right)\\
    + \frac{1}{N}\sum_{n=0}^{N-1}\left(F(\nu^{n+1}, \mu^n) - F(\nu^{n}, \mu^n)\right) - \frac{1}{N\tau}\sum_{n=0}^{N-1}\left(D_h(\nu^{n+1}, \nu^n) + D_h(\mu^{n+1}, \mu^n)\right),
\end{multline}
where we used the fact that $D_h(\nu, \nu^{N}) + D_h(\mu, \mu^{N}) \geq 0,$ for any $(\nu, \mu) \in \mathcal{C} \times \mathcal{D}.$

Note that the key difference to the estimates from Theorem \ref{thm: conv-sim-bregman} is the appearance of the term $F(\nu^{n+1}, \mu^n) - F(\nu^{n}, \mu^n)$ due to the non-symmetry of the flat derivatives of $F$ in \eqref{eq:mirror-alt}. The idea is to combine $F(\nu^{n+1}, \mu^n) - F(\nu^{n}, \mu^n)$ with both $\int_{\mathcal{X}} \frac{\delta F}{\delta \nu}(\nu^n, \mu^n, x)(\nu^n - \nu^{n+1})(\mathrm{d}x)$ and $\int_{\mathcal{X}} \frac{\delta F}{\delta \mu}(\nu^{n+1}, \mu^n, y)(\mu^{n+1} - \mu^{n})(\mathrm{d}y)$ via relative smoothness in order to obtain $D_h(\nu^n, \nu^{n+1}) - D_h(\nu^{n+1}, \nu^{n})$ and $D_h(\mu^n, \mu^{n+1}) - D_h(\mu^{n+1}, \mu^{n}),$ which will prove to be of order $\mathcal{O}(\tau^3).$

By Proposition \ref{prop:foc}, the first-order conditions for \eqref{eq:mirror-alt} read
\begin{equation}
\label{eq:foc-alt}
\begin{cases}
    \frac{\delta h}{\delta \nu}(\nu^{n+1}, x) - \frac{\delta h}{\delta \nu}(\nu^{n}, x) = -\tau \frac{\delta F}{\delta \nu}(\nu^{n}, \mu^n, x) + C_{n,1},\\
    \frac{\delta h}{\delta \mu}(\mu^{n+1}, y) - \frac{\delta h}{\delta \mu}(\mu^{n}, y) = \tau \frac{\delta F}{\delta \mu}(\nu^{n+1}, \mu^n, y) + C_{n,2},
\end{cases}
\end{equation}
for all $x \in \mathcal{X}$ $\nu^{n+1}$-a.e. and $y \in \mathcal{X}$ $\mu^{n+1}$-a.e., where $C_{n,1}, C_{n,2} \in \mathbb R.$ It can be shown directly from Definition \ref{def:bregman-div} that
\begin{equation}
\label{eq:symm-bregman}
\int_{\mathcal{X}} \left(\frac{\delta h}{\delta \nu}(\nu', x) - \frac{\delta h}{\delta \nu}(\nu, x)\right)(\nu'-\nu)(\mathrm{d}x) = D_h(\nu', \nu) + D_h(\nu, \nu'),
\end{equation}
for all $\nu,\nu' \in \mathcal{C},$ and analogously for $D_h(\mu', \mu) + D_h(\mu, \mu').$ Then, using \eqref{eq:foc-alt} and \eqref{eq:symm-bregman} we obtain that
\begin{multline}
\label{eq:foc1}
    -\int_{\mathcal{X}} \frac{\delta F}{\delta \nu}(\nu^n, \mu^n, x)(\nu^{n+1} - \nu^{n})(\mathrm{d}x) = \frac{1}{\tau}\int_{\mathcal{X}} \left(\frac{\delta h}{\delta \nu}(\nu^{n+1}, x) - \frac{\delta h}{\delta \nu}(\nu^{n}, x)\right)(\nu^{n+1}-\nu^{n})(\mathrm{d}x)\\
    = \frac{1}{\tau}\left(D_h(\nu^{n+1}, \nu^n) + D_h(\nu^{n}, \nu^{n+1})\right),
\end{multline}
and similarly
\begin{multline}
\label{eq:foc2}
    \int_{\mathcal{X}} \frac{\delta F}{\delta \mu}(\nu^{n+1}, \mu^n, y)(\mu^{n+1} - \mu^{n})(\mathrm{d}y) = \frac{1}{\tau}\int_{\mathcal{X}} \left(\frac{\delta h}{\delta \mu}(\mu^{n+1}, y) - \frac{\delta h}{\delta \mu}(\mu^{n}, y)\right)(\mu^{n+1}-\mu^{n})(\mathrm{d}y)\\
    = \frac{1}{\tau}\left(D_h(\mu^{n+1}, \mu^n) + D_h(\mu^{n}, \mu^{n+1})\right).
\end{multline}
Therefore, using \eqref{eq:foc1} and \eqref{eq:foc2} in \eqref{eq:start-estimate}, we obtain that
\begin{multline}
\label{eq:second-estimate}
    \operatorname{NI}\left(\frac{1}{N}\sum_{n=0}^{N-1}\nu^{n+1}, \frac{1}{N}\sum_{n=0}^{N-1}\mu^n\right) \leq \frac{1}{N\tau}\left(\sup_{\nu \in \mathcal{C}} D_h(\nu, \nu^0) + \sup_{\mu \in \mathcal{D}} D_h(\mu, \mu^0)\right)\\ 
    + \frac{1}{N\tau} \sum_{n=0}^{N-1} \Bigg(D_h(\nu^{n+1}, \nu^n) + D_h(\nu^{n}, \nu^{n+1}) + D_h(\mu^{n+1}, \mu^n) + D_h(\mu^{n}, \mu^{n+1})\Bigg)\\ 
    + \frac{1}{N}\sum_{n=0}^{N-1}\left(F(\nu^{n+1}, \mu^n) - F(\nu^{n}, \mu^n)\right) - \frac{1}{N\tau}\sum_{n=0}^{N-1}\left(D_h(\nu^{n+1}, \nu^n) + D_h(\mu^{n+1}, \mu^n)\right).
\end{multline}
Then, we observe that
    \begin{equation}
    \label{eq:comut1}
    D_{h}(\nu^{n}, \nu^{n+1}) 
    = \frac{1}{2}\left(D_{h}(\nu^{n}, \nu^{n+1}) - D_{h}(\nu^{n+1}, \nu^{n})\right) + \frac{1}{2}\left(D_{h}(\nu^{n}, \nu^{n+1}) + D_{h}(\nu^{n+1}, \nu^{n})\right),
\end{equation}
and a similar representation holds for $D_{h}(\mu^{n}, \mu^{n+1}).$ Similarly, we can write
\begin{multline}
\label{eq:F-telescope}
    F(\nu^{n+1}, \mu^n) - F(\nu^n, \mu^n) = \frac{1}{2}\left(F(\nu^{n+1}, \mu^n) - F(\nu^n, \mu^n)\right)\\ 
    + \frac{1}{2}\left(F(\nu^{n+1}, \mu^n) - F(\nu^{n+1}, \mu^{n+1}) + F(\nu^{n+1}, \mu^{n+1}) - F(\nu^n, \mu^n)\right).
\end{multline}
Therefore, putting \eqref{eq:comut1} and \eqref{eq:F-telescope} into \eqref{eq:second-estimate} gives
\begin{multline}
\label{eq:third-estimate}
    \operatorname{NI}\left(\frac{1}{N}\sum_{n=0}^{N-1} \nu^{n+1}, \frac{1}{N}\sum_{n=0}^{N-1} \mu^n\right) \leq  \frac{1}{N\tau}\left(\sup_{\nu \in \mathcal{C}} D_h(\nu, \nu^0) + \sup_{\mu \in \mathcal{D}} D_h(\mu, \mu^0)\right)\\ 
    + \frac{1}{2N\tau}\sum_{n=0}^{N-1}\left(D_{h}(\nu^{n}, \nu^{n+1}) - D_{h}(\nu^{n+1}, \nu^{n}) + D_{h}(\mu^{n}, \mu^{n+1}) - D_{h}(\mu^{n+1}, \mu^{n})\right)\\ 
    + \frac{1}{2N}\sum_{n=0}^{N-1}\left(\frac{1}{\tau}(D_{h}(\nu^{n}, \nu^{n+1}) + D_{h}(\nu^{n+1}, \nu^{n})) + F(\nu^{n+1}, \mu^n) - F(\nu^n, \mu^n)\right)\\
    + \frac{1}{2N}\sum_{n=0}^{N-1} \Bigg(\frac{1}{\tau}(D_{h}(\mu^{n}, \mu^{n+1}) + D_{h}(\mu^{n+1}, \mu^{n})) + F(\nu^{n+1}, \mu^n) - F(\nu^{n+1}, \mu^{n+1})\\ + F(\nu^{n+1}, \mu^{n+1}) - F(\nu^n, \mu^n)\Bigg).
\end{multline}
Combining the fact that $\nu \mapsto F(\nu, \mu)$ is $L_{\nu}$-smooth relative to $h$ with the first-order condition \eqref{eq:foc-alt}, we have that
\begin{multline}
\label{eq:smooth1}
    F(\nu^{n+1}, \mu^n) - F(\nu^n, \mu^n) \leq \int_{\mathcal{X}} \frac{\delta F}{\delta \nu}(\nu^n, \mu^n,x)(\nu^{n+1}-\nu^n)(\mathrm{d}x) + L_{\nu}D_h(\nu^{n+1}, \nu^n)\\
    = -\frac{1}{\tau}\int_{\mathcal{X}} \left(\frac{\delta h}{\delta \nu}(\nu^{n+1}, x) - \frac{\delta h}{\delta \nu}(\nu^{n}, x)\right)(\nu^{n+1}-\nu^n)(\mathrm{d}x) + L_{\nu}D_h(\nu^{n+1}, \nu^n)\\
    = -\frac{1}{\tau}\left(D_h(\nu^{n+1}, \nu^n) + D_h(\nu^{n}, \nu^{n+1})\right) + L_{\nu}D_h(\nu^{n+1}, \nu^n),
\end{multline}
where the last equality follows from \eqref{eq:symm-bregman}.

Similarly, using $L_{\mu}$-smoothness of $\mu \mapsto F(\nu, \mu)$ relative to $h$ together with \eqref{eq:foc-alt}, we can show that
\begin{equation}
\label{eq:smooth2}
    F(\nu^{n+1}, \mu^n) - F(\nu^{n+1}, \mu^{n+1}) + \frac{1}{\tau}\left(D_{h}(\mu^{n}, \mu^{n+1}) + D_{h}(\mu^{n+1}, \mu^{n})\right) \leq L_{\mu}D_{h}(\mu^{n+1}, \mu^{n}).
\end{equation}
Therefore, using \eqref{eq:smooth1} and \eqref{eq:smooth2} in \eqref{eq:third-estimate}, and recalling that $L = \max\{L_{\nu}, L_{\mu}\}$ gives
\begin{multline}
\label{eq:fourth-estimate}
    \operatorname{NI}\left(\frac{1}{N}\sum_{n=0}^{N-1} \nu^{n+1}, \frac{1}{N}\sum_{n=0}^{N-1} \mu^n\right) 
    \leq \frac{1}{N\tau}\left(\sup_{\nu \in \mathcal{C}} D_h(\nu, \nu^0) + \sup_{\mu \in \mathcal{D}} D_h(\mu, \mu^0)\right)\\ 
    + \frac{1}{2N\tau}\sum_{n=0}^{N-1}\left(D_{h}(\nu^{n}, \nu^{n+1}) - D_{h}(\nu^{n+1}, \nu^{n}) + D_{h}(\mu^{n}, \mu^{n+1}) - D_{h}(\mu^{n+1}, \mu^{n})\right)\\ 
    + \frac{L}{2N}\sum_{n=0}^{N-1} \left(D_{h}(\nu^{n+1}, \nu^{n}) + D_{h}(\mu^{n+1}, \mu^{n})\right) + \frac{1}{2N}\left(F\left(\nu^{N}, \mu^{N}\right) - F(\nu^0, \mu^0)\right).
\end{multline}
Since, by Lemma \ref{lemma:tau2-bound-primal}, 
\begin{equation}
\label{eq:order-2-bound-bregman}
    D_h(\nu^{n+1}, \nu^n) + D_h(\mu^{n+1}, \mu^n) \leq \frac{32 \tau^2 \kappa_2}{\alpha},
\end{equation} 
where $\kappa_2 \coloneqq C_1^2 + C_2^2,$ it suffices to show that $D_{h}(\nu^{n}, \nu^{n+1}) - D_{h}(\nu^{n+1}, \nu^{n}) + D_{h}(\mu^{n}, \mu^{n+1}) - D_{h}(\mu^{n+1}, \mu^{n})$ is of order $\mathcal{O}(\tau^3).$ Indeed, we could then choose $\tau = \mathcal{O}\left(N^{-1/3}\right),$ and since by Lemma \ref{lem:boundedness-from-flat-derivatives}, $\left|F\left(\nu^{N}, \mu^{N}\right)\right| \leq |F(\nu^0,\mu^0)| +2 (C_1 +C_2)$, we would obtain that
\begin{equation*}
    \operatorname{NI}\left(\frac{1}{N}\sum_{n=0}^{N-1} \nu^{n+1}, \frac{1}{N}\sum_{n=0}^{N-1} \mu^n\right) \leq \mathcal{O}\left(\frac{1}{N^{2/3}}\right) + \mathcal{O}\left(\frac{1}{N}\right) =  \mathcal{O}\left(\frac{1}{N^{2/3}}\right),
\end{equation*}
because $\frac{1}{N} \leq \frac{1}{N^{2/3}},$ for all $N \geq 1.$

In order to show that $D_{h}(\nu^{n}, \nu^{n+1}) - D_{h}(\nu^{n+1}, \nu^{n}) + D_{h}(\mu^{n}, \mu^{n+1}) - D_{h}(\mu^{n+1}, \mu^{n})$ is $\mathcal{O}(\tau^3),$ we will leverage the connection between Bregman divergence and dual Bregman divergence given by Lemma \ref{lemma:primal-dual-iterates} together with Assumptions \ref{assumption:bddflatF} and \ref{assumption:lipschitz-h^*2}.

If we denote $f^n \coloneqq \frac{\delta h}{\delta \nu}(\nu^n, \cdot),$ for any $n \geq 0,$ then by Lemma \ref{lemma:primal-dual-iterates}, we have that $D_h(\nu^{n}, \nu^{n+1}) = D_{h^*}(f^{n+1}, f^{n}).$ For any $\varepsilon \in [0,1]$ denote $f^{\varepsilon,n} = \varepsilon f^{n+1} + (1-\varepsilon)f^n$ and $\phi(\varepsilon) = h^*(f^{\varepsilon,n}).$ Note that $\phi(0) = h^*(f^n)$ and $\phi(1) = h^*(f^{n+1}).$ We have
\begin{equation*}
    \phi'(\varepsilon) = \nabla_\mathcal{F}h^*(f^{\varepsilon,n})[f^{n+1}-f^n], \quad \phi''(\varepsilon) = \nabla^2_\mathcal{F}h^*(f^{\varepsilon,n})[f^{n+1}-f^n][f^{n+1}-f^n].
\end{equation*}
Note that $\phi'(0) = \nabla_\mathcal{F}h^*(f^n)[f^{n+1}-f^n].$ By the fundamental theorem of calculus and integration by parts, we have
\begin{align*}
    \phi(1)-\phi(0) = \int_0^1 \phi'(\varepsilon)\mathrm{d}\varepsilon&=[(\varepsilon-1)\phi'(\varepsilon)]|_{\varepsilon=0}^{\varepsilon=1} - \int_0^1 (\varepsilon-1)\phi''(\varepsilon)\mathrm{d}\varepsilon\\ &=\phi'(0) + \int_0^1 (1-\varepsilon)\phi''(\varepsilon)\mathrm{d}\varepsilon.
\end{align*}
Hence,
\begin{equation*}
    h^*(f^{n+1}) - h^*(f^n) - \nabla_\mathcal{F}h^*(f^n)[f^{n+1}-f^n] = \int_0^1 (1-\varepsilon)\nabla^2_\mathcal{F}h^*(f^{\varepsilon,n})[f^{n+1}-f^n][f^{n+1}-f^n]\mathrm{d}\varepsilon
\end{equation*}
By Definition \ref{def:dual-Bregman}, we have that
\begin{equation*}
    D_{h^*}(f^{n+1}, f^n) = \int_0^1 (1-\varepsilon)\nabla^2_\mathcal{F}h^*(f^{\varepsilon,n})[f^{n+1}-f^n][f^{n+1}-f^n]\mathrm{d}\varepsilon
\end{equation*}
Similarly, by Lemma \ref{lemma:primal-dual-iterates}, we have that $D_h(\nu^{n+1}, \nu^n) = D_{h^*}(f^n, f^{n+1}),$ and hence
\begin{align*}
    D_{h^*}(f^n, f^{n+1}) &= \int_0^1 (1-\varepsilon)\nabla^2_\mathcal{F}h^*(f^{1-\varepsilon,n})[f^{n+1}-f^n][f^{n+1}-f^n]\mathrm{d}\varepsilon.
\end{align*}
Therefore, we obtain that
\begin{align*}
    &D_{h^*}(f^{n+1}, f^n) - D_{h^*}(f^n, f^{n+1})\\
    &= \int_0^1 (1-\varepsilon)\left(\nabla^2_\mathcal{F}h^*(f^{\varepsilon,n}) -\nabla^2_\mathcal{F}h^*(f^{1-\varepsilon,n})\right)[f^{n+1}-f^n][f^{n+1}-f^n]\mathrm{d}\varepsilon.
\end{align*}
Note that $f^{\varepsilon,n} - f^{1-\varepsilon,n} = (2\varepsilon-1)(f^{n+1}-f^n)$. If we denote $g^{\varepsilon,\gamma,n} = f^{1-\varepsilon,n} + \gamma(f^{\varepsilon,n} - f^{1-\varepsilon,n}),$ then applying the fundamental theorem of calculus again gives
\begin{align*}
    &D_{h^*}(f^{n+1}, f^n) - D_{h^*}(f^n, f^{n+1})\\ 
    &= \int_0^1 (1-\varepsilon)(2\varepsilon-1)\int_0^1\nabla^3_\mathcal{F}h^*(g^{\varepsilon,\gamma,n})[f^{n+1}-f^n][f^{n+1}-f^n][f^{n+1}-f^n]\mathrm{d}\gamma\mathrm{d}\varepsilon.
\end{align*}
Using Assumption \ref{assumption:lipschitz-h^*2}, we further obtain
\begin{align*}
    D_{h^*}(f^{n+1}, f^n) - D_{h^*}(f^n, f^{n+1}) &\leq L_{h^*} \|f^{n+1}-f^n\|_{\infty}^3\int_0^1 |(1-\varepsilon) 
    (2\varepsilon-1)| \mathrm{d}\varepsilon\\
    &= \frac{L_{h^*}}{4}\|f^{n+1}-f^n\|^3_{\infty},
\end{align*}
The first-order condition for the minimizing player in \eqref{eq:foc-alt} can be rewritten as
\begin{equation}
    f^{n+1}(x) - f^n(x) = -\tau \frac{\delta F}{\delta \nu}(\nu^{n}, \mu^n, x),
\end{equation}
for all $x \in \mathcal{X}$ $\nu^{n+1}$-a.e. By Assumption \ref{assumption:bddflatF}, there exists $C_1 > 0$ such that $\left\|\frac{\delta F}{\delta \nu}(\nu^n, \mu^n, \cdot)\right\|_{\infty} \leq C_1,$ for any $n \geq 0.$ Hence, we obtain that
\begin{align*}
    D_{h^*}(f^{n+1}, f^n) - D_{h^*}(f^{n}, f^{n+1}) &\leq \frac{L_{h^*}}{4} \|f^{n+1}-f^n\|^3_{\infty}\\ 
    &= \frac{L_{h^*}}{4}\tau^3\left\|\frac{\delta F}{\delta \nu}(\nu^n, \mu^n, \cdot)\right\|_{\infty}^3 \leq \frac{L_{h^*}}{4} \tau^3 C_1^3,
\end{align*}
Similarly, denoting $g^n \coloneqq \frac{\delta h}{\delta \mu}(\mu^n, \cdot),$ for any $n \geq 0,$ and repeating the steps above, we can prove that
\begin{align*}
    D_{h^*}(g^{n+1}, g^n) - D_{h^*}(g^{n}, g^{n+1}) &\leq \frac{L_{h^*}}{4}\|g^{n+1}-g^n\|^3_{\infty}\\
    &=\frac{L_{h^*}}{4}\tau^3\left\|\frac{\delta F}{\delta \mu}(\nu^{n+1}, \mu^n, \cdot)\right\|_{\infty}^3 \leq \frac{L_{h^*}}{4}\tau^3 C_2^3,
\end{align*}
where $C_2 > 0$ exists due to Assumption \ref{assumption:bddflatF}.

Set $\kappa_1 \coloneqq \frac{1}{4}\left(C_1^3 + C_2^3\right) > 0.$ Then,
\begin{equation}
\label{eq:tau3bound}
D_{h}(\nu^{n}, \nu^{n+1}) - D_{h}(\nu^{n+1}, \nu^{n}) + D_{h}(\mu^{n}, \mu^{n+1}) - D_{h}(\mu^{n+1}, \mu^{n}) \leq \kappa_1 L_{h^*} \tau^3.
\end{equation}
Hence, using \eqref{eq:order-2-bound-bregman}, \eqref{eq:tau3bound} and Lemma \ref{lem:boundedness-from-flat-derivatives}, estimate \eqref{eq:fourth-estimate} becomes 
\begin{align*}
    &\operatorname{NI}\left(\frac{1}{N}\sum_{n=0}^{N-1} \nu^{n+1}, \frac{1}{N}\sum_{n=0}^{N-1} \mu^n\right) \leq \frac{1}{N\tau}\left(\sup_{\nu \in \mathcal{C}} D_h(\nu, \nu^0) + \sup_{\mu \in \mathcal{D}} D_h(\mu, \mu^0)\right)\\ 
    &+ \frac{1}{2N\tau}\sum_{n=0}^{N-1}\Bigg(\left(D_{h}(\nu^{n}, \nu^{n+1}) - D_{h}(\nu^{n+1}, \nu^{n})\right) + \left(D_{h}(\mu^{n}, \mu^{n+1}) - D_{h}(\mu^{n+1}, \mu^{n})\right)\Bigg)\\ 
    &+ \frac{L}{2N}\sum_{n=0}^{N-1} \left(D_{h}(\nu^{n+1}, \nu^{n}) + D_{h}(\mu^{n+1}, \mu^{n})\right) + \frac{1}{2N}\left(F\left(\nu^{N}, \mu^{N}\right) - F(\nu^0, \mu^0)\right)\\
    &= \frac{1}{N\tau}\left(\sup_{\nu \in \mathcal{C}} D_h(\nu, \nu^0) + \sup_{\mu \in \mathcal{D}} D_h(\mu, \mu^0)\right) + \left(\frac{\kappa_1 L_{h^*}}{2}
    + \frac{16\kappa_2 L}{\alpha}\right) \tau^2 + \frac{|F(\nu^0,\mu^0)| + C_1 +C_2}{N}.
\end{align*}
Minimizing the right-hand side over $\tau$ amounts to taking $$\tau = \left(\frac{\sup_{\nu \in \mathcal{C}} D_h(\nu, \nu^0) + \sup_{\mu \in \mathcal{D}} D_h(\mu, \mu^0)}{2N}\right)^{1/3}\left(\frac{\kappa_1 L_{h^*}}{2}
    + \frac{16\kappa_2 L}{\alpha}\right)^{-1/3},$$ and since $\frac{1}{N} \leq \frac{1}{N^{2/3}},$ for any $N \geq 1,$ it follows that
\begin{multline*}
    \operatorname{NI}\left(\frac{1}{N}\sum_{n=0}^{N-1} \nu^{n+1}, \frac{1}{N}\sum_{n=0}^{N-1} \mu^n\right) \leq \frac{1}{(2N)^{2/3}}\Bigg(3\left(\sup_{\nu \in \mathcal{C}} D_h(\nu, \nu^0) + \sup_{\mu \in \mathcal{D}} D_h(\mu, \mu^0)\right)^{2/3}\times\\
    \times\left(\frac{\kappa_1 L_{h^*}}{2}
    + \frac{16\kappa_2 L}{\alpha}\right)^{1/3} + 4^{1/3}\left(|F(\nu^0,\mu^0)| + C_1 +C_2\right)\Bigg).
\end{multline*}
\end{proof}

\section{Proofs of additional results}
\label{appendix:additional-results}
In this section, we present the proofs of the additional results of the paper. We start with examples of functions $h$ and the corresponding sets $\mathcal{E}$ such that Assumption \ref{assumption:assump-h} is satisfied. We continue with the proofs of Lemmas \ref{lemma:tau2-bound-primal}, \ref{lem:boundedness-from-flat-derivatives} and \ref{lemma:Bregman-prox-ineq}, which play a key role in proving the main results. Finally, we prove some auxiliary results.

\subsection{Examples satisfying Assumption \ref{assumption:assump-h}}
\label{subsection:examples satifying 1.1}
\begin{example}[Relative entropy]
\label{example:relative-entropy}
    Suppose that $h$ is the relative entropy, i.e., $h(m) \coloneqq \int_{\mathcal{X}} \frac{m(x)}{\pi(x)}\log\frac{m(x)}{\pi(x)} \pi(x)\mathrm{d}x,$ where $m, \pi \in \mathcal{P}_{\lambda}(\mathcal{X})$, i.e., they are absolutely continuous with respect to the Lebesgue measure on $\mathcal{X}$ and $\pi$ is a fixed reference probability measure on $\mathcal{P}_{\lambda}(\mathcal{X}).$ Fix $\bar{\beta} > 0$ and define $\mathcal{E}_{\bar\beta} := \left\{m \in \mathcal{P}_{\pi}(\mathcal{X}): \left\|\log\frac{m(\cdot)}{\pi(\cdot)}\right\|_{L^\infty(\mathcal{X})} < \bar\beta\right\}.$ Let $\mathcal{C} = \mathcal{D} := \mathcal{E}_{\bar\beta}.$ Then $\mathcal{E}_{\bar\beta}$ is convex. Moreover, it is proved in \cite[Proposition $2.17$]{kerimkulov2024mirror} that $h$ admits the flat derivative
    \begin{equation}
    \label{eq:flat-h}
    \frac{\delta h}{\delta m}(m,x) = \log \frac{m(x)}{\pi(x)} - h(m),
    \end{equation}
    on $\mathcal{E}_{\bar\beta},$ and for all $m, m' \in \mathcal{E_{\bar\beta}},$ the Bregman divergence $D_h(m', m)$ is in fact the Kullback-Leibler divergence (or relative entropy) $\operatorname{KL}(m', m).$ Therefore, by Pinsker's inequality, that is, $\operatorname{TV}^2(m',m) \leq \frac{1}{2}\operatorname{KL}(m',m),$ Assumption \ref{assumption:assump-h} holds with $\alpha=4$. A related observation appears in \cite[Remark 5]{korba}. A sufficient condition for the relative entropy (or the entropy) to admit a flat derivative in $L^{\infty}$ at $m$ is that there exists $\kappa_0, \kappa_1 > 0$ such that $\kappa_0 \leq \frac{m(\cdot)}{\pi(\cdot)} \leq \kappa_1$ a.e. on $\mathcal{X},$ which is precisely a similar condition to the one defining the class $\mathcal{E_{\bar\beta}}.$
\end{example}
\begin{example}[{\(\chi^2\)-divergence}]
\label{example:chi-squared}
    Suppose that $h$ is the \(\chi^2\)-divergence, i.e., $h(m) \coloneqq \frac{1}{2}\int_{\mathcal{X}} \left(\frac{m(x)}{\pi(x)}-1\right)^2\pi(x) \mathrm{d}x,$ where $m, \pi \in \mathcal{P}_{\lambda}(\mathcal{X})$. Let $L_\pi^2(\mathcal{X})$ be the set of square integrable functions on $\mathcal{X}$ with respect to $\pi.$ Fix $\bar\eta > 0$ and define $\mathcal{F}_{\bar\eta} := \left\{m \in \mathcal{P}_{\pi}(\mathcal{X}): \frac{m(\cdot)}{\pi(\cdot)} > 0 \ \pi\text{-a.e.},\left\|\frac{m(\cdot)}{\pi(\cdot)}\right\|_{L_\pi^2(\mathcal{X})} \leq \bar \eta \right\}.$ Let $\mathcal{C}=\mathcal{D} := \mathcal{F}_{\bar \eta} $. Note that $\mathcal{F}_{\bar\eta}$ is convex. Moreover, it is proved in \cite[Proposition $2.20$]{kerimkulov2024mirror} that $h$ admits the flat derivative $$\frac{\delta h}{\delta m}(m,x) = \frac{m(x)}{\pi(x)} - \int_{\mathbb R^d} \frac{m(z)}{\pi(z)}m(z)\mathrm{d}z,$$ on $\mathcal{F}_{\bar\eta},$ and for all $m, m' \in \mathcal{F}_{\bar\eta},$ the Bregman divergence $D_h(m', m)$ is in fact the \(L^2\)-distance $\frac{1}{2}\left\|\frac{m'(\cdot)}{\pi(\cdot)} - \frac{m(\cdot)}{\pi(\cdot)}\right\|^2_{L_\pi^2(\mathcal{X})}.$ Since $\pi \in \mathcal{P}_\lambda(\mathcal{X}),$ the Cauchy-Schwarz inequality implies $\frac{1}{2}\operatorname{TV}^2(m',m) \leq D_h(m', m)$. Thus, Assumption \ref{assumption:assump-h} holds with $\alpha=1$.
\end{example}

\subsection{Proof of Lemma \ref{lemma:tau2-bound-primal}}
\begin{proof}[Proof of Lemma \ref{lemma:tau2-bound-primal}]
We will only prove the lemma for Algorithm \ref{eq:mirror-sim-explicit} since the argument for Algorithm \ref{eq:mirror-alt} is almost identical. From $L_{\nu}$-relative smoothness and the definition of $\nu^{n+1}$ in \eqref{eq:mirror-sim-explicit}, for any $\nu \in \mathcal{C},$ it follows that
    \begin{align*}
        &F(\nu^{n+1}, \mu^n)\\ 
        &\leq F(\nu^n, \mu^n) + \int_{\mathcal{X}}\frac{\delta F}{\delta \nu}(\nu^n, \mu^n, x)(\nu^{n+1}-\nu^n)(\mathrm{d}x) +\left(\frac{1}{\tau} + L_{\nu} - \frac{1}{\tau}\right)D_h(\nu^{n+1}, \nu^n)\\
        &\leq F(\nu^n, \mu^n) + \int_{\mathcal{X}}\frac{\delta F}{\delta \nu}(\nu^n, \mu^n, x)(\nu-\nu^n)(\mathrm{d}x) + \frac{1}{\tau}D_h(\nu, \nu^n) + \left(L_{\nu} - \frac{1}{\tau}\right)D_h(\nu^{n+1}, \nu^n).
    \end{align*}
    Setting $\nu = \nu^n,$ we obtain that
    \begin{equation*}
    \label{eq:bd1}
         F(\nu^{n+1}, \mu^n) \leq F(\nu^n, \mu^n) + \left(L_{\nu} - \frac{1}{\tau}\right)D_h(\nu^{n+1}, \nu^n).
    \end{equation*}
    Recall $L\coloneqq \max\{L_{\nu}, L_{\mu}\} >0.$ By assumption, $\tau L \leq \frac{1}{2},$ and so we get
    \begin{align*}
        \frac{1}{2\tau}D_h(\nu^{n+1}, \nu^n) &\leq F(\nu^n, \mu^{n}) - F(\nu^{n+1}, \mu^n)\\
        &= \int_0^1 \int_\mathcal{X}\frac{\delta F}{\delta \nu}(\nu^{n+1}+\varepsilon(\nu^{n}-\nu^{n+1}),\mu^n,x)(\nu^n-\nu^{n+1})(\mathrm{d}x)\mathrm{d}\varepsilon\\
        &\leq 2C_1\operatorname{TV}(\nu^{n+1},\nu^n)\\
    &\leq 2C_1 \sqrt{\frac{2}{\alpha}D_h(\nu^{n+1}, \nu^n)},
\end{align*}
where the penultimate inequality follows from Assumption \ref{assumption:bddflatF} and the last inequality follows from Assumption \ref{assumption:assump-h}. Hence, since $D_h(\nu^{n+1}, \nu^n) \geq 0,$ for all $n \geq 0,$ we obtain that
\begin{equation*}
    D_h(\nu^{n+1}, \nu^n) \leq \frac{32 \tau^2C_1^2}{\alpha}.
\end{equation*}    
    From $L_{\mu}$-relative smoothness and the definition of $\mu^{n+1}$ in \eqref{eq:mirror-sim-explicit}, for any $\mu \in \mathcal{D},$ it follows that
    \begin{align*}
        &F(\nu^{n}, \mu^{n+1})\\ 
        &\geq F(\nu^n, \mu^n) + \int_{\mathcal{X}}\frac{\delta F}{\delta \mu}(\nu^n, \mu^n, y)(\mu^{n+1}-\mu^n)(\mathrm{d}y) - \left(\frac{1}{\tau} + L_{\mu} - \frac{1}{\tau}\right)D_h(\mu^{n+1}, \mu^n)\\
            &\geq F(\nu^n, \mu^n) + \int_{\mathcal{X}}\frac{\delta F}{\delta \mu}(\nu^n, \mu^n, y)(\mu-\mu^n)(\mathrm{d}y) - \frac{1}{\tau}D_h(\mu, \mu^n) - \left(L_{\mu} - \frac{1}{\tau}\right)D_h(\mu^{n+1}, \mu^n).
    \end{align*}
    Setting $\mu = \mu^n,$ we obtain that
    \begin{equation*}
    \label{eq:bd2}
         F(\nu^{n}, \mu^{n+1}) \geq F(\nu^n, \mu^n) - \left(L_{\mu} - \frac{1}{\tau}\right)D_h(\mu^{n+1}, \mu^n).
    \end{equation*}
    Using again the assumption $\tau L \leq \frac{1}{2},$ we get 
    \begin{equation*}
        \frac{1}{2\tau}D_h(\mu^{n+1}, \mu^n) \leq F(\nu^n, \mu^{n+1}) - F(\nu^{n}, \mu^n) 
    \leq 2C_2 \sqrt{\frac{2}{\alpha}D_h(\mu^{n+1}, \mu^n)},
\end{equation*}
where the last inequality follows from Assumptions \ref{assumption:bddflatF} and \ref{assumption:assump-h}. Hence, since $D_h(\mu^{n+1}, \mu^n) \geq 0,$ for all $n \geq 0,$ we obtain that
\begin{equation*}
    D_h(\mu^{n+1}, \mu^n) \leq \frac{32 \tau^2C_2^2}{\alpha}.
\end{equation*}    
\end{proof}
\subsection{Proof of Lemma \ref{lem:boundedness-from-flat-derivatives}}
\begin{proof}[Proof of Lemma \ref{lem:boundedness-from-flat-derivatives}]
Fix $\mu\in\mathcal{D}$ and $\nu,\nu'\in\mathcal{C}$, and define $\nu_\lambda := (1-\lambda)\nu+\lambda\nu'$, $\lambda\in[0,1]$,
which stays in $\mathcal{C}$ by convexity. By Remark G.2 applied to $\lambda\mapsto F(\nu_\lambda,\mu)$,
\[
F(\nu',\mu)-F(\nu,\mu)
=\int_0^1\int_{\mathcal{X}} \frac{\delta F}{\delta \nu}(\nu_\lambda,\mu,x)(\nu'-\nu)(\mathrm{d}x)\mathrm{d}\lambda.
\]
Using the dual characterization of total variation (i.e., $|\int \phi\,\mathrm{d}(\nu'-\nu)|\le 2\|\phi\|_\infty \operatorname{TV}(\nu',\nu)$ for bounded $\phi$), and Assumption \ref{assumption:bddflatF}, we obtain
\[
|F(\nu',\mu)-F(\nu,\mu)|
\le 2C_1 \operatorname{TV}(\nu',\nu).
\]
Similarly, fixing $\nu'\in\mathcal{C}$ and defining $\mu_\lambda:=(1-\lambda)\mu+\lambda\mu'\in\mathcal{D}$, we get $|F(\nu',\mu')-F(\nu',\mu)|\le 2C_2\operatorname{TV}(\mu',\mu)$.
Combining the two bounds with the triangle inequality yields
\begin{align*}
|F(\nu',\mu')-F(\nu,\mu)|
&\le |F(\nu',\mu')-F(\nu',\mu)| + |F(\nu',\mu)-F(\nu,\mu)|\\
&\le 2C_2\operatorname{TV}(\mu',\mu)+2C_1\operatorname{TV}(\nu',\nu),
\end{align*}
proving that $F$ is $\operatorname{TV}$-Lipschitz. The uniform bound on $F$ follows by applying this estimate with $(\nu',\mu')=(\nu,\mu)$ and
$(\nu,\mu)=(\nu^0,\mu^0)$, noting that $\operatorname{TV}(\cdot,\cdot)\le 1$ on probability measures and taking the supremum over $\mathcal{C}\times\mathcal{D}$.
\end{proof}

\subsection{Proof of Lemma \ref{lemma:Bregman-prox-ineq}}
\begin{proof}[Proof of Lemma \ref{lemma:Bregman-prox-ineq}]
Fix $\nu \in  \mathcal{E}$. Define $\mathcal G: \mathcal{E} \to \mathbb R$ by $\mathcal G(\nu) = G(\nu) + D_h(\nu,\mu)$ for all $\mu \in \mathcal{E}$. Since $G$ and $h$ are convex and flat differentiable, $\mathcal G$ is flat differentiable. As $\bar \nu$ minimises $\mathcal G$ over $\mathcal{E}$,
\[
\int_{\mathcal{X}} \frac{\delta \mathcal G}{\delta \nu}(\bar \nu,x)(\nu- \bar \nu)(\mathrm{d}x) 
=\lim_{\varepsilon \searrow  0}\frac{ \mathcal G (\bar{\nu}+\varepsilon (\nu-\bar{\nu} ))- \mathcal G(\bar{\nu})}{\varepsilon}\ge 0,
\]
which implies that 
\begin{align*}
D_{\mathcal G}(\nu, \bar \nu) = \mathcal G(\nu) -\mathcal G (\bar \nu) -  \int \frac{\delta \mathcal G}{ \delta m} (\bar \nu ,x)(\nu - \bar \nu)(\mathrm{d}x) \le   \mathcal G(\nu) - \mathcal G (\bar \nu ).
\end{align*}
Adding $\mathcal G (\bar \nu)$ to both sides of the inequality and 
using the definition of $\mathcal G$ gives
\begin{align*}
G(\nu)+D_h(\nu,\mu) 
&\ge 
D_{\mathcal G}(\nu, \bar \nu)  +G(\bar \nu) +D_h( \bar \nu ,\mu) 
\\
&=
D_{  G}(\nu, \bar \nu) +D_{ D_h(\cdot, \mu)}(\nu, \bar \nu)  +G(\bar \nu) +D_h( \bar \nu ,\mu) 
\\
&\ge 
D_{  h}(\nu, \bar \nu)  +G(\bar \nu) +D_h( \bar \nu, \mu),
\end{align*}
where the equality follows from the linearity of the Bregman divergence and the last inequality follows since $D_{G}(\nu, \Bar{\nu}) \geq 0$ by convexity of $G$ and the identity $D_{ D_h(\cdot, \mu)}(\nu, \bar \nu) = D_{  h}(\nu, \bar \nu)$.
\end{proof} 

\subsection{Proofs of auxiliary results}
In this subsection, we start by proving the convexity characterization $h$ via its flat derivative.
\begin{lemma}[Strong convexity of $h$]
\label{lemma:strong-convexity-h}
	Let Assumption \ref{assumption:assump-h} hold. Then there exists $\alpha > 0 $ such that for any $\nu, \nu' \in \mathcal{E},$
	\begin{equation*}
		 h(\nu') - h(\nu) \geq \int_{\mathcal{X}} \frac{\delta h}{\delta \nu}(\nu,x)(\nu'-\nu)(\mathrm{d}x) + \frac{\alpha}{2}\operatorname{TV}^2(\nu',\nu).
	\end{equation*} 
\end{lemma}
\begin{proof}
		Let $\varepsilon \in [0,1]$ and $\nu^\varepsilon = \nu + \varepsilon(\nu'-\nu).$ Since $h$ is strongly convex, and by Definition \ref{def:fderivative},
		\begin{align*}
			\varepsilon\left(h(\nu') - h(\nu)\right) - \frac{\alpha}{2}\varepsilon(1-\varepsilon)\operatorname{TV}^2(\nu',\nu)&\geq h(\nu^\varepsilon ) - h(\nu)\\ 
            &=\varepsilon\int_0^1 \int_{\mathcal{X}} \frac{\delta h}{\delta \nu}(\nu^{s\varepsilon},x)(\nu'-\nu)(\mathrm{d}x)\mathrm{d}s.
		\end{align*}
		Dividing by $\varepsilon$ and passing to the limit $\varepsilon \downarrow 0$ via dominated convergence theorem gives the conclusion since $\frac{\delta h}{\delta \nu}$ is bounded and continuous.
\end{proof}
Following \cite[Lemma 2]{korba}, we prove that the mirror updates in Algorithms \ref{eq:mirror-sim-explicit} and \ref{eq:mirror-alt} satisfy iterative schemes on the dual space. Before presenting the proof, we define the relative subdifferential of $h$ over $\mathcal K \subseteq\mathcal{P}(\mathcal{X})$ at $\rho \in \mathcal{P}(\mathcal{X})$ by
\begin{equation*}
    \partial_{\mathcal K}h(\rho) := \left\{q: \int_{\mathcal X}q(x)(\eta-\rho)(\mathrm dx) \leq \int_{\mathcal X}\frac{\delta h}{\delta\nu}(\rho,x)(\eta-\rho)(\mathrm dx), \forall\,\eta\in\mathcal K\right\}.
\end{equation*}

\begin{proposition}[MDA dual iteration]
\label{prop:foc}
    For each $n \geq 0,$ let $F(\cdot,\mu^n) \in \mathfrak{C}^1(\mathcal{C})$, $F(\nu^n,\cdot) \in \mathfrak{C}^1(\mathcal{D})$. Moreover, let Assumption \ref{assumption:assump-h} hold. Then the minimizer-maximizer pair $(\nu^{n+1},\mu^{n+1}) \in \mathcal{C} \times \mathcal{D}$ of each one of Algorithms \ref{eq:mirror-sim-explicit} and \ref{eq:mirror-alt} satisfies the corresponding dual iterative update
    \begin{equation}
    \begin{cases}
        \frac{\delta h}{\delta\nu}(\nu^n,\cdot) - \tau \frac{\delta F}{\delta\nu}(\nu^n,\mu^n,\cdot) \in \partial_{\mathcal C}h(\nu^{n+1}),\\
        \frac{\delta h}{\delta\mu}(\mu^n,\cdot) + \tau \frac{\delta F}{\delta\mu}(\nu^n,\mu^n,\cdot) \in \partial_{\mathcal D}h(\mu^{n+1}),
    \end{cases}
    \label{eq:relative-subdifferential-inclusion-sim}
\end{equation}
\begin{equation}
\label{eq:relative-subdifferential-inclusion-alt}
    \begin{cases}
        \frac{\delta h}{\delta\nu}(\nu^n,\cdot) - \tau \frac{\delta F}{\delta\nu}(\nu^n,\mu^n,\cdot) \in \partial_{\mathcal C}h(\nu^{n+1}),\\
        \frac{\delta h}{\delta\mu}(\mu^n,\cdot) + \tau \frac{\delta F}{\delta\mu}(\nu^{n+1},\mu^n,\cdot) \in \partial_{\mathcal D}h(\mu^{n+1}).
    \end{cases}
\end{equation}
Thus, if $\partial_{\mathcal C}h(\nu^{n+1}) = \left\{\frac{\delta h}{\delta \nu}(\nu^{n+1},\cdot)+c_{\nu,n}: c_{\nu,n}\in\mathbb R\right\}$ and $\partial_{\mathcal D}h(\mu^{n+1}) = \left\{\frac{\delta h}{\delta \mu}(\mu^{n+1},\cdot)+c_{\mu,n}: c_{\mu,n}\in\mathbb R\right\},$ \eqref{eq:relative-subdifferential-inclusion-sim} and \eqref{eq:relative-subdifferential-inclusion-alt}, respectively, correspond to
    \begin{equation*}
\begin{cases}
    \frac{\delta h}{\delta \nu}(\nu^{n+1}, \cdot) - \frac{\delta h}{\delta \nu}(\nu^{n}, \cdot) = -\tau \frac{\delta F}{\delta \nu}(\nu^{n}, \mu^n, \cdot) + c_{n,1},\quad \nu^{n+1}-\text{a.e.},\\
    \frac{\delta h}{\delta \mu}(\mu^{n+1}, \cdot) - \frac{\delta h}{\delta \mu}(\mu^{n}, \cdot) = \tau \frac{\delta F}{\delta \mu}(\nu^n, \mu^n, \cdot) + c_{n,2}, \quad \mu^{n+1}-\text{a.e.},
\end{cases}
    \end{equation*}
    \begin{equation*}
\begin{cases}
    \frac{\delta h}{\delta \nu}(\nu^{n+1}, \cdot) - \frac{\delta h}{\delta \nu}(\nu^{n}, \cdot) = -\tau \frac{\delta F}{\delta \nu}(\nu^{n}, \mu^n, \cdot) + c_{n,3},\quad \nu^{n+1}-\text{a.e.},\\
    \frac{\delta h}{\delta \mu}(\mu^{n+1}, \cdot) - \frac{\delta h}{\delta \mu}(\mu^{n}, \cdot) = \tau \frac{\delta F}{\delta \mu}(\nu^{n+1}, \mu^n, \cdot) + c_{n,4}, \quad \mu^{n+1}-\text{a.e.},
\end{cases}
    \end{equation*}
    where $c_{n,1}, c_{n,2}, c_{n,3}, c_{n,4} \in \mathbb R.$
\end{proposition}
\begin{proof}
    We present the proof only for Algorithm \ref{eq:mirror-sim-explicit}, as the proof for Algorithm \ref{eq:mirror-alt} is identical. For convenience, define
    \begin{equation*}
        G(\nu) \coloneqq \int_{\mathcal{X}} \frac{\delta F}{\delta \nu}(\nu^{n}, \mu^n, x)(\nu-\nu^n)(\mathrm{d}x), \quad \nu \in \mathcal{C}.
    \end{equation*}
    Since $\nu^{n+1}$ is the minimizer in Algorithm \ref{eq:mirror-sim-explicit}, we have
    \begin{equation*}
        G(\nu^{n+1}) + \frac{1}{\tau} D_h(\nu^{n+1},\nu^n) \leq G(\nu) + \frac{1}{\tau} D_h(\nu,\nu^n), \quad \forall \nu \in \mathcal{C}.
    \end{equation*}
    Because $\mathcal{C}$ is convex, fix any $\Tilde{\nu} \in \mathcal{C}$ and take $\nu = \nu^{n+1} +\varepsilon(\Tilde{\nu} - \nu^{n+1}) \in \mathcal{C}.$ Then
    \begin{align*}
        G(\nu^{n+1}) + \frac{1}{\tau} D_h(\nu^{n+1},\nu^n) \leq G(\nu^{n+1} +\varepsilon(\Tilde{\nu} - \nu^{n+1})) + \frac{1}{\tau} D_h(\nu^{n+1} +\varepsilon(\Tilde{\nu} - \nu^{n+1}),\nu^n),
    \end{align*}
    which, using linearity of $G$ and rearranging, becomes
    \begin{equation*}
        \varepsilon(G(\Tilde{\nu})-G(\nu^{n+1})) + \frac{1}{\tau}\left(D_h(\nu^{n+1} +\varepsilon(\Tilde{\nu} - \nu^{n+1}),\nu^n) - D_h(\nu^{n+1},\nu^n)\right) \geq 0.
    \end{equation*}
    Dividing by $\varepsilon$ and letting $\varepsilon \searrow 0,$ Definition \ref{def:fderivative} yields
    \begin{equation*}
        G(\Tilde{\nu})-G(\nu^{n+1}) + \frac{1}{\tau}\int_\mathcal{X} \frac{\delta D_h(\cdot,\nu^n)}{\delta \nu}(\nu^{n+1},x)(\Tilde{\nu} - \nu^{n+1})(\mathrm{d}x) \geq 0.
    \end{equation*}
   By the definition of Bregman divergence and flat derivative,
    \begin{equation*}
        \frac{\delta D_h(\cdot,\nu^n)}{\delta \nu}(\nu^{n+1},x) = \frac{\delta h}{\delta \nu}(\nu^{n+1}, x) - \frac{\delta h}{\delta \nu}(\nu^{n}, x).
    \end{equation*}
    Hence,
    \begin{equation*}
        \int_\mathcal{X}\left(\frac{\delta F}{\delta \nu}(\nu^{n}, \mu^n, x) + \frac{1}{\tau}\left(\frac{\delta h}{\delta \nu}(\nu^{n+1}, x) - \frac{\delta h}{\delta \nu}(\nu^{n}, x)\right)\right)(\Tilde{\nu}-\nu^{n+1})(\mathrm{d}x) \geq 0, \; \forall \Tilde{\nu} \in \mathcal{C}.
    \end{equation*}
    Rearranging this inequality gives
    \begin{equation*}
        \frac{\delta h}{\delta\nu}(\nu^n,\cdot) - \tau \frac{\delta F}{\delta\nu}(\nu^n,\mu^n,\cdot) \in \partial_{\mathcal C}h(\nu^{n+1})
    \end{equation*}
    An analogous argument gives the optimality condition for the maximizer $\mu^{n+1}.$
\end{proof}
We show that the relative subdifferential of $h$ is a singleton for both the relative entropy on $\mathcal{E}_{\bar \beta}$ (Example \ref{example:relative-entropy}) and the $\chi^2$-divergence on $\mathcal{F}_{\bar \eta}$ (Example \ref{example:chi-squared}).
\begin{lemma}[Singleton subdifferential for relative entropy]
\label{lemma:singleton-relative-entropy}
    Let $\pi\in\mathcal P(\mathcal X)$ be a reference probability measure and define $h(\nu) := \operatorname{KL}(\nu|\pi)$. For $\bar\beta>0$, consider the convex feasible set $\mathcal E_{\bar\beta} :=
    \left\{m\in\mathcal P_\pi(\mathcal X):\left\| \log\frac{\mathrm d m}{\mathrm d\pi}\right\|_{L^\infty(\pi)}< \bar\beta
    \right\}.$ Then 
    \begin{equation}
    \partial_{\mathcal E_{\bar\beta}}h(\nu) = \left\{\log \frac{\mathrm d\nu}{\mathrm d\pi}(\cdot)+c: c\in\mathbb R\right\}.
    \label{eq:entropy-relative-subdifferential}
\end{equation}
\end{lemma}
\begin{proof}
    By construction of the feasible set, $\nu\in\mathcal E_{\bar\beta}$ satisfies the strict slackness condition 
    \begin{equation}
    \label{eq:entropy-strict-slack}
        \left\|\log \frac{\mathrm d\nu}{\mathrm d\pi}\right\|_{L^\infty(\pi)} \leq \bar\beta-\delta
    \end{equation} 
    for some $\delta>0$. Let $p_\nu := \log \frac{\mathrm d\nu}{\mathrm d\pi}$ be the unnormalized flat derivative of $h$ at $\nu$. By definition,
    \begin{equation*}
    \partial_{\mathcal E_{\bar\beta}}h(\nu) :=
    \left\{q\in L^\infty(\pi): \int_{\mathcal X}q \ \mathrm d(\eta-\nu) \leq \int_{\mathcal X} p_\nu \ \mathrm d(\eta-\nu), \forall\eta\in\mathcal E_{\bar\beta}
    \right\}.
    \end{equation*}
Fix $q\in\partial_{\mathcal E_{\bar\beta}}h(\nu)$. Then
\begin{equation}
    \int_{\mathcal X} (q-p_\nu)\ \mathrm d(\eta-\nu) \leq 0, \quad \forall \eta\in\mathcal E_{\bar\beta}.
    \label{eq:normal-cone-entropy}
\end{equation}
Let $u\in L^\infty(\pi)$ be arbitrary with $\int_{\mathcal X}u\mathrm d\pi=0.$ Condition \eqref{eq:entropy-strict-slack} implies
\begin{equation*}
    e^{-(\bar\beta-\delta)} \leq \frac{\mathrm d\nu}{\mathrm d\pi} \leq e^{\bar\beta-\delta},
    \quad \pi\text{-a.e.}
\end{equation*}
Therefore, for sufficiently small $t>0$,
\begin{equation*}
    e^{-\bar\beta}\leq \frac{\mathrm d\nu}{\mathrm d\pi}\pm tu \leq e^{\bar\beta}, \quad \pi\text{-a.e.}
\end{equation*}
For example, one could take $t\|u\|_{L^\infty(\pi)} < \min\{e^{-(\bar\beta-\delta)}-e^{-\bar\beta}, e^{\bar\beta}-e^{\bar\beta-\delta}\}.$ Since $\int u \mathrm d\pi=0$, the measures $\frac{\mathrm d \eta_\pm}{\mathrm d\pi}:=\frac{\mathrm d\nu}{\mathrm d\pi}\pm tu$ belong to $\mathcal E_{\bar\beta}$. Applying \eqref{eq:normal-cone-entropy} first to $\eta_+$ and then to $\eta_-$ gives
\begin{equation*}
    t\int_{\mathcal X}(q-p_\nu) u \ \mathrm d\pi\leq0, \quad -t\int_{\mathcal X}(q-p_\nu) u \ \mathrm d\pi\leq0.
\end{equation*}
Hence
\begin{equation}
    \int_{\mathcal X}(q-p_\nu) u \ \mathrm d\pi=0,
    \label{eq:orthogonal-zero-mass}
\end{equation}
for every $u\in L^\infty(\pi)$ such that $\int_{\mathcal X}u \ \mathrm d\pi=0.$ Since $q-p_\nu\in L^\infty(\pi)$, we may take $ u = q-p_\nu -\int_{\mathcal X}(q-p_\nu)\ \mathrm d\pi$ in \eqref{eq:orthogonal-zero-mass}. This gives
\begin{equation*}
    0 = \int_{\mathcal X} \left((q-p_\nu)-\int_{\mathcal X}(q-p_\nu) \mathrm d\pi\right)^2\mathrm d\pi.
\end{equation*}
Hence $q-p_\nu$ is constant $\pi$-a.e., and therefore $q=p_\nu+c$ for some $c\in\mathbb R$. This proves one inclusion in \eqref{eq:entropy-relative-subdifferential}.

Conversely, if $q=p_\nu+c$ for some $c\in\mathbb R$, then for every $\eta\in\mathcal E_{\bar\beta}$,
\begin{equation*}
    \int_{\mathcal X}q \ \mathrm d(\eta-\nu) = \int_{\mathcal X}p_\nu \ \mathrm d(\eta-\nu).
\end{equation*}
Thus $q\in\partial_{\mathcal E_{\bar\beta}}h(\nu)$, proving
\eqref{eq:entropy-relative-subdifferential}.
\end{proof}

\begin{lemma}[Singleton subdifferential for $\chi^2$]
\label{lemma:singleton-chi^2}
Let $\pi\in\mathcal P(\mathcal X)$ be a reference probability measure and
define $ h(\nu) := \frac12 \int_{\mathcal X}\left(\frac{\mathrm d\nu}{\mathrm d\pi}-1\right)^2 \mathrm d\pi.$ For $\bar\eta>0$, consider the convex feasible set $ \mathcal F_{\bar\eta} := \left\{\nu\in\mathcal P_\pi(\mathcal X): \frac{\mathrm d\nu}{\mathrm d\pi} > 0 \ \pi\text{-a.e.}, \|\frac{\mathrm d\nu}{\mathrm d\pi}\|_{L^2(\pi)}\leq \bar\eta\right\}.$ Then 
\begin{equation}
\partial_{\mathcal F_{\bar\eta}}h(\nu) = \left\{\frac{\mathrm d\nu}{\mathrm d\pi}+c: c\in\mathbb R\right\},
\label{eq:chi-square-relative-subdifferential}
\end{equation}
\end{lemma}
\begin{proof}
    For $\nu\in\mathcal F_{\bar\eta}$, let $p_\nu:=\frac{\mathrm d\nu}{\mathrm d\pi},$ which is an unnormalized representative of the flat derivative of $h$ at $\nu$. By definition of relative subdifferential,
    \begin{equation*}
    \partial_{\mathcal F_{\bar\eta}}h(\nu) := \left\{q\in L^2(\pi):\int_{\mathcal X}q\, \mathrm d(\eta-\nu) \leq \int_{\mathcal X} p_\nu \ \mathrm d(\eta-\nu),\forall \eta\in\mathcal F_{\bar\eta}\right\}.
    \end{equation*}
Fix $q\in\partial_{\mathcal F_{\bar\eta}}h(\nu)$. The relative subdifferential inequality then becomes
\begin{equation}
    \int_{\mathcal X} (q-p_\nu)\mathrm d(\eta-\nu) \leq 0, \quad \forall \nu\in\mathcal F_{\bar\eta}. 
\label{eq:chi-square-normal-cone}
\end{equation}
Let $v\in L^\infty(\pi)$ satisfy $\int_{\mathcal X}v \ \mathrm d\nu = \int_{\mathcal X}\frac{\mathrm d\nu}{\mathrm d\pi}v\ \mathrm d\pi  = 0.$ For $t>0$, define $\frac{\mathrm d \eta_\pm}{\mathrm d \pi}:=\frac{\mathrm d\nu}{\mathrm d\pi}(1\pm tv).$ If $t> 0$ is chosen such that $t\|v\|_{L^\infty(\pi)}<1,$ then
\begin{equation*}
    \frac{\mathrm d \eta_\pm}{\mathrm d \pi}>0, \quad \pi\text{-a.e.}
\end{equation*}
Moreover,
\begin{equation*}
    \int_{\mathcal X}\frac{\mathrm d \eta_\pm}{\mathrm d \pi} \ \mathrm d\pi = 1 \pm t\int_{\mathcal X}\frac{\mathrm d\nu}{\mathrm d\pi}v \ \mathrm d\pi = 1.
\end{equation*}
Finally,
\begin{equation*}
    \left\|\frac{\mathrm d \eta_\pm}{\mathrm d \pi}\right\|_{L^2(\pi)} \leq \left\|\frac{\mathrm d\nu}{\mathrm d\pi}\right\|_{L^2(\pi)} + t\left\|\frac{\mathrm d\nu}{\mathrm d\pi}v\right\|_{L^2(\pi)}.
\end{equation*}
By the $L^2$-radius condition, the right-hand side is smaller than $\bar\eta$ for all sufficiently small $t>0$. For example, it suffices to take 
\begin{equation*}
    0 <t < \min \left\{\frac{1}{\|v\|_{L^\infty(\pi)}}, \frac{\bar\eta - \|\frac{\mathrm d\nu}{\mathrm d\pi}\|_{L^2(\pi)}}{\|\frac{\mathrm d\nu}{\mathrm d\pi}v\|_{L^2(\pi)}}\right\}.
\end{equation*}
Hence $ \frac{\mathrm d \eta_\pm}{\mathrm d \pi} \in \mathcal F_{\bar\eta}.$ Applying \eqref{eq:chi-square-normal-cone} first with $\eta_+$ and then
with $\eta_-$ gives
\begin{equation*}
    t\int_{\mathcal X}(q-p_\nu)v \ \mathrm d\nu\leq0, \quad -t\int_{\mathcal X}(q-p_\nu) v \ \mathrm d\nu\leq0.
\end{equation*}
Therefore,
\begin{equation*}
    \int_{\mathcal X}(q-p_\nu) v \mathrm d\nu=0
\end{equation*}
for every bounded $v$ satisfying $\int v \mathrm d\nu=0$. Since $q-p_\nu\in L^2(\pi)$ and $\frac{\mathrm d\nu}{\mathrm d\pi}\in L^2(\pi)$, the Cauchy-Schwarz inequality gives $q-p_\nu \in L^1(\nu)$. Set $c:=\int_{\mathcal X}(q-p_\nu) \mathrm d\nu.$ For any measurable set $A\subseteq\mathcal X$, take $v_A:=\mathbbm 1_A-\nu(A).$ Then $v_A\in L^\infty(\pi)$ and $\int v_A \mathrm d\nu=0$, and hence
\begin{equation*}
    0 = \int_{\mathcal X}(q-p_\nu) v_A \mathrm d\nu = \int_A (q-p_\nu) \mathrm d\nu - \nu(A)\int_{\mathcal X}(q-p_\nu) \mathrm d\nu = \int_A((q-p_\nu)-c)\mathrm d\nu.
\end{equation*}
Since this holds for every measurable $A$, it follows that $q-p_\nu=c$ $\nu$-a.e. By the strict positivity condition $\frac{\mathrm d\nu}{\mathrm d\pi} > 0 \ \pi\text{-a.e.}$, the measures $\nu$ and $\pi$ are equivalent, and therefore $q-p_\nu=c$ $\pi$-a.e., which proves one inclusion in \eqref{eq:chi-square-relative-subdifferential}.

Conversely, let $q=p_\nu+c$ for some $c\in\mathbb R$. Then it is immediate that $q\in\partial_{\mathcal F_{\bar\eta}}h(\nu)$, proving \eqref{eq:chi-square-relative-subdifferential}.
\end{proof}

We now check that the condition $\sup_{\nu \in \mathcal{C}}D_h(\nu, \nu^0) + \sup_{\mu \in \mathcal{D}}D_h(\mu, \mu^0) < \infty$ required in Theorems \ref{thm: conv-sim-bregman} and \ref{thm:conv-alt-bregman} is satisfied in the specific cases of Examples \ref{example:relative-entropy} and \ref{example:chi-squared}. In addition, we establish that single-valuedness of the relative subdifferential holds for all iterates along the trajectory.
\begin{lemma}
\label{lemma:verif-init-entropy}
Let Assumption \ref{assumption:bddflatF} hold and let $h$ denote the relative entropy from Example \ref{example:relative-entropy}. Assume $\nu^0,\mu^0\in\mathcal{E}_{\beta_0}$ for some $\beta_0>0$. Fix a horizon $N\ge 1$ and a stepsize $\tau>0$.
Define $\mathcal{C}=\mathcal{D}=\mathcal{E}_{\bar\beta}$, with $\bar\beta := \beta_0 + 2N\tau\max\{C_1,C_2\}+1$. Let $(\nu^n,\mu^n)_{n=0}^N$ be the iterates produced by Algorithms~\ref{eq:mirror-sim-explicit} and \ref{eq:mirror-alt} with feasible sets
$\mathcal{C},\mathcal{D}$. Then
\[
(\nu^n,\mu^n)_{n=0}^N \subseteq \mathcal{E}_{\bar\beta} \times \mathcal{E}_{\bar\beta}.
\]
Hence,
\begin{align*}
    \partial_{\mathcal C}h(\nu^n)
    &= \left\{\log \frac{\mathrm d\nu^n}{\mathrm d\pi}(\cdot)+c_{\nu,n}: c_{\nu,n}\in\mathbb R\right\},\\
    \partial_{\mathcal D}h(\mu^n) &= \left\{\log \frac{\mathrm d\mu^n}{\mathrm d\pi}(\cdot)+c_{\mu,n}: c_{\mu,n}\in\mathbb R\right\}, \quad n \in \{0,1,\dots, N\}.
\end{align*}
Moreover,
\[
\sup_{\nu\in\mathcal{C}}\mathrm{KL}(\nu,\nu^0)\le \bar\beta+\beta_0,
\quad
\sup_{\mu\in\mathcal{D}}\mathrm{KL}(\mu,\mu^0)\le \bar\beta+\beta_0,
\]
and in particular
\[
\sup_{\nu\in\mathcal{C}}\mathrm{KL}(\nu,\nu^0)\;+\;\sup_{\mu\in\mathcal{D}}\mathrm{KL}(\mu,\mu^0)\;<\infty.
\]
\end{lemma}

\begin{proof}
 We provide the proof only for Algorithm \ref{eq:mirror-sim-explicit}, as the argument for the other algorithm is essentially the same. Assume $\nu^0,\mu^0\in \mathcal{E}_{\beta_0}$ for some $\beta_0>0$. Fix a horizon $N\ge 1$ and let $\bar\beta := \beta_0 + 2N\tau \max\{C_1,C_2\}+1$. We take $\mathcal{C}=\mathcal{D}=\mathcal{E}_{\bar\beta}$.

 Using the flat derivative formula \eqref{eq:flat-h}, the first-order optimality condition in Proposition \ref{prop:foc} gives
\begin{equation*}
\begin{cases}
    \log \frac{\nu^{n+1}(x)}{\pi(x)} - \log \frac{\nu^{n}(x)}{\pi(x)} = -\tau \frac{\delta F}{\delta \nu}(\nu^n, \mu^n, x) - \log\int_{\mathcal{X}} e^{-\tau \frac{\delta F}{\delta \nu}(\nu^n, \mu^n, x)}\frac{\nu^n(x)}{\pi(x)}\pi(x)\mathrm{d}x,\\
    \log \frac{\mu^{n+1}(y)}{\pi(y)} - \log \frac{\mu^n(y)}{\pi(y)} = \tau \frac{\delta F}{\delta \mu}(\nu^n, \mu^n, y)  - \log\int_{\mathcal{X}} e^{\tau \frac{\delta F}{\delta \mu}(\nu^n, \mu^n, y)}\frac{\mu^n(y)}{\pi(y)}\pi(y)\mathrm{d}y,
\end{cases}
\end{equation*}
for all $x,y \in \mathcal{X}$ a.e. Taking the sup-norm on both sides over $x,y$ and using the assumptions gives
\begin{equation*}
    \left\|\log \frac{\nu^{n+1}(\cdot)}{\pi(\cdot)}\right\|_{L^\infty(\mathcal{X})} \leq \left\|\log \frac{\nu^n(\cdot)}{\pi(\cdot)}\right\|_{L^\infty(\mathcal{X})} + 2\tau C_1,
\end{equation*}
\begin{equation*}
    \left\|\log \frac{\mu^{n+1}(\cdot)}{\pi(\cdot)}\right\|_{L^\infty(\mathcal{X})} \leq \left\|\log \frac{\mu^n(\cdot)}{\pi(\cdot)}\right\|_{L^\infty(\mathcal{X})} + 2\tau C_2,
\end{equation*}
Iterating from $n=0$ and using $\nu^0,\mu^0\in \mathcal{E}_{\beta_0}$ gives, for all $n=0,\dots,N$,
\[
\left\|\log \frac{\nu^n(\cdot)}{\pi(\cdot)}\right\|_{L^\infty(\mathcal{X})}\le \beta_0 + 2n\tau C_1< \bar\beta,
\quad
\left\|\log \frac{\mu^n(\cdot)}{\pi(\cdot)}\right\|_{L^\infty(\mathcal{X})} \le \beta_0 + 2n\tau C_2< \bar\beta.
\]
Hence, $\nu^n,\mu^n\in \mathcal{E}_{\bar\beta}=\mathcal{C}=\mathcal{D}$ for all $n=0,\dots,N$, i.e.,
\[
(\nu^n,\mu^n)_{n=0}^N \subseteq \mathcal{E}_{\bar\beta} \times \mathcal{E}_{\bar\beta}.
\]
Hence, choosing, $\mathcal C=\mathcal D=\mathcal E_{\bar\beta},$ $\bar\beta := \beta_0 + 2N\tau\max\{C_1,C_2\} + 1,$ ensures strict slackness at every iterate $n\leq N$. Hence, by Lemma~\ref{lemma:singleton-relative-entropy},
\begin{align*}
    \partial_{\mathcal C}h(\nu^{n+1})
    &= \left\{\log \frac{\mathrm d\nu^{n+1}}{\mathrm d\pi}(\cdot)+c_{\nu,n}: c_{\nu,n}\in\mathbb R\right\},\\
    \partial_{\mathcal D}h(\mu^{n+1}) &= \left\{\log \frac{\mathrm d\mu^{n+1}}{\mathrm d\pi}(\cdot)+c_{\mu,n}: c_{\mu,n}\in\mathbb R\right\}.
\end{align*}
Combining this with Proposition~\ref{prop:foc} gives, for some constant $c_{\nu,n}\in\mathbb R$,
\begin{equation*}
    \log \frac{\mathrm d\nu^{n+1}}{\mathrm d\pi}-\log \frac{\mathrm d\nu^n}{\mathrm d\pi} = -\tau \frac{\delta F}{\delta\nu}(\nu^n,\mu^n,\cdot)
    +c_{\nu,n}.
\end{equation*}
The analogous ascent identity holds for the $\mu$ update.

Now fix any $\nu\in\mathcal{E}_{\bar\beta}$. Since $\nu^0\in\mathcal{E}_{\beta_0}$, we have a.e.
\[
-\bar\beta \le \log\frac{\nu(\cdot)}{\pi(\cdot)}\le \bar\beta,
\quad
-\beta_0 \le \log\frac{\nu^0(\cdot)}{\pi(\cdot)}\le \beta_0,
\]
and therefore a.e.
\[
\log\frac{\nu(\cdot)}{\nu^0(\cdot)} = \log\frac{\nu(\cdot)}{\pi(\cdot)}-\log\frac{\nu^0(\cdot)}{\pi(\cdot)} \le \bar\beta+\beta_0.
\]
Hence, since $\nu$ is a probability measure,
\[
\mathrm{KL}(\nu,\nu^0) \le \int_{\mathcal{X}} (\bar\beta+\beta_0)\,\nu(\mathrm{d}x) = \bar\beta+\beta_0.
\]
Thus, $\sup_{\nu\in\mathcal{C}}\mathrm{KL}(\nu,\nu^0)\le \bar\beta+\beta_0<\infty$.
The same argument with $\mu\in\mathcal{E}_{\bar\beta}$ and $\mu^0\in\mathcal{E}_{\beta_0}$ gives $\sup_{\mu\in\mathcal{D}}\mathrm{KL}(\mu,\mu^0)\le \bar\beta+\beta_0<\infty$. Combining the two bounds proves
\[
\sup_{\nu\in\mathcal{C}}\mathrm{KL}(\nu,\nu^0)+\sup_{\mu\in\mathcal{D}}\mathrm{KL}(\mu,\mu^0)<\infty.
\]
\end{proof}

\begin{lemma}
\label{lemma:verif-init-chi}
    Let Assumption \ref{assumption:bddflatF} hold and let $h$ denote the \(\chi^2\)-divergence from Example \ref{example:chi-squared}. Assume $\nu^0,\mu^0\in\mathcal{F}_{\eta_0}$ for some $\eta_0>0$. Fix a horizon $N\ge 1$ and a stepsize $\tau>0$. Define $\mathcal{C}=\mathcal{D}=\mathcal{F}_{\bar\eta}$, with $\bar\eta \coloneqq \eta_0 + N\tau\max\{C_1,C_2\}$. Suppose, moreover, that
\begin{equation}
\label{eq:chi-square-initial-lower-bound}
    \underline\kappa_\nu:= \operatorname*{ess\,inf}_{\pi} \frac{\mathrm d\nu^0}{\mathrm{d}\pi} > 2N\tau C_1,
    \quad
    \underline\kappa_\mu :=\operatorname*{ess\,inf}_{\pi} \frac{\mathrm{d}\mu^0}{\mathrm{d}\pi} > 2N\tau C_2.
\end{equation}
Let $(\nu^n,\mu^n)_{n=0}^N$ be the iterates produced by Algorithms~\ref{eq:mirror-sim-explicit} and \ref{eq:mirror-alt} with feasible sets $\mathcal{C},\mathcal{D}$. Then, for every $n=0,...,N$,
\begin{equation}
\label{eq:chi-square-radius-bounds}
    \left\|\frac{\mathrm d\nu^n}{\mathrm d\pi}\right\|_{L^2_\pi} \leq \bar \eta,
    \quad
    \left\|\frac{\mathrm d\mu^n}{\mathrm d\pi}\right\|_{L^2_\pi}
    \leq \bar \eta,
\end{equation}
and
\begin{equation}
\label{eq:chi-square-positivity-bounds}
    \frac{\mathrm d\nu^n}{\mathrm d\pi} \geq \underline\kappa_\nu-2n\tau C_1 >0,
    \quad
    \frac{\mathrm d\mu^n}{\mathrm d\pi} \geq \underline\kappa_\mu-2n\tau C_2
    >0,
    \quad \pi\text{-a.e.}
\end{equation}
In particular,
\[
(\nu^n,\mu^n)_{n=0}^N \subseteq \mathcal{F}_{\bar\eta} \times \mathcal{F}_{\bar\eta},
\]
and all iterates are equivalent to $\pi$. Hence,
\begin{align*}
    \partial_{\mathcal C}h(\nu^n)
    &= \left\{\frac{\mathrm d\nu^n}{\mathrm d\pi}(\cdot)+c_{\nu,n}: c_{\nu,n}\in\mathbb R\right\},\\
    \partial_{\mathcal D}h(\mu^n) &= \left\{\frac{\mathrm d\mu^n}{\mathrm d\pi}(\cdot)+c_{\mu,n}: c_{\mu,n}\in\mathbb R\right\}, \quad n \in \{0,1,\dots, N\}.
\end{align*}
Moreover,
\[
\frac{1}{2}\sup_{\nu \in \mathcal{C}}\left\|\frac{\nu(\cdot)}{\pi(\cdot)} - \frac{\nu^0(\cdot)}{\pi(\cdot)}\right\|^2_{L_\pi^2(\mathcal{X})}\le \bar\eta^2+\eta_0^2,
\quad
\frac{1}{2}\sup_{\mu \in \mathcal{D}}\left\|\frac{\mu(\cdot)}{\pi(\cdot)} - \frac{\mu^0(\cdot)}{\pi(\cdot)}\right\|^2_{L_\pi^2(\mathcal{X})}\le \bar\eta^2+\eta_0^2,
\]
and in particular
\begin{equation*}
        \frac{1}{2}\sup_{\nu \in \mathcal{C}}\left\|\frac{\nu(\cdot)}{\pi(\cdot)} - \frac{\nu^0(\cdot)}{\pi(\cdot)}\right\|^2_{L_\pi^2(\mathcal{X})} + \frac{1}{2}\sup_{\mu \in \mathcal{D}}\left\|\frac{\mu(\cdot)}{\pi(\cdot)} - \frac{\mu^0(\cdot)}{\pi(\cdot)}\right\|^2_{L_\pi^2(\mathcal{X})} < \infty.
\end{equation*}
\end{lemma}
\begin{proof}
     We provide the proof only for Algorithm \ref{eq:mirror-sim-explicit}, as the argument for the other algorithm is essentially the same. Assume $\nu^0,\mu^0\in\mathcal{F}_{\eta_0}$ for some $\eta_0>0$. Fix a horizon $N\ge 1$ and let $\bar\eta := \eta_0 + N\tau\max\{C_1,C_2\}$. We take $\mathcal{C}=\mathcal{D}=\mathcal{F}_{\bar\eta}$. The first-order condition (see e.g., \cite[Section 5.1.1]{Bonnans2000PerturbationAO}) shows that
\begin{equation*} 
\label{eq:first-order_L2}
\left\langle \frac{\delta F}{\delta \nu}(\nu^n, \mu^{n}, \cdot) + \frac{1}{\tau}
\left(\frac{\mathrm d \nu^{n+1}}{\mathrm d\pi }  -\frac{\mathrm d \nu^n }{\mathrm d\pi}  \right),
\phi-  \frac{\mathrm d \nu^{n+1}}{\mathrm d\pi} 
\right\rangle_{L^2_\pi}\ge 0,  
\quad \forall \phi\in \mathfrak F\,, 
\end{equation*}
\begin{equation*} 
\label{eq:first-order_L2_2}
\left\langle \frac{\delta F}{\delta \mu}(\nu^n, \mu^n, \cdot) - \frac{1}{\tau}
\left(\frac{\mathrm d \mu^{n+1}}{\mathrm d\pi }  -\frac{\mathrm d \mu^n }{\mathrm d\pi}  \right),
\phi-  \frac{\mathrm d \mu^{n+1}}{\mathrm d\pi} 
\right\rangle_{L^2_\pi}\ge 0,  
\quad \forall \phi\in \mathfrak F\,, 
\end{equation*}
where $\left\langle \cdot, \cdot \right\rangle_{L^2_\pi} $ is the inner product on $L^2_\pi(\mathcal{X})$,
and $\mathfrak F$ is the nonempty closed convex set defined by 
$$
\mathfrak F=\left\{\phi\in L^2_\pi(\mathcal{X})\bigg\vert \phi\geq 0\ \pi\text{-a.e. on } \mathcal{X} \text{ and } \int \phi(x)\pi(\mathrm{d}x)=1\right\}.
$$
Define the projection map 
$\Pi_{\mathfrak F}: L^2_\pi(\mathcal{X})\mapsto \mathfrak F$ such that 
$\Pi_{\mathfrak F}(\varphi)=\argmin_{\phi\in \mathfrak F}\|\phi-\varphi\|_{L^2_\pi(\mathcal{X})}$ for all $\varphi \in L^2_\pi(\mathcal{X})$, 
which satisfies 
$$
\left\langle \Pi(\varphi)-\varphi, \phi-  \Pi(\varphi)\right\rangle_{L^2_\pi}\ge 0,  \quad \forall \phi\in \mathfrak F. 
$$
Then 
$$
\frac{\mathrm d \nu^{n+1}}{\mathrm d \pi} = \Pi_{\mathfrak F}\left( \frac{\mathrm d \nu^n }{\mathrm d\pi }  -\tau\frac{\delta F}{\delta \nu}(\nu^n, \mu^n, \cdot)  \right),
$$
$$\frac{\mathrm d \mu^{n+1}}{\mathrm d \pi} = \Pi_{\mathfrak F}\left( \frac{\mathrm d \mu^n }{\mathrm d\pi }  +\tau\frac{\delta F}{\delta \mu}(\nu^n, \mu^n, \cdot)  \right).$$ Note $\|\Pi_{\mathfrak F}(\varphi_1)-\Pi_{\mathfrak F}(\varphi_2)\|_{L^2_\pi(\mathcal{X})}
\le \|\varphi_1-\varphi_2\|_{L^2_\pi(\mathcal{X})}$ for all $\varphi_1,\varphi_2\in L^2_\pi(\mathcal{X})$ (see e.g., \cite[Theorem 4.3-1]{ciarlet2013linear}). Moreover, since $\frac{\mathrm d \nu^n}{\mathrm d \pi} = \Pi_{\mathfrak F}\left(\frac{\mathrm d \nu^n}{\mathrm d \pi}\right),$ $\frac{\mathrm d \mu^n}{\mathrm d \pi} = \Pi_{\mathfrak F}\left(\frac{\mathrm d \mu^n}{\mathrm d \pi}\right),$
\begin{align*}
\left\|\frac{\mathrm d \nu^{n+1}}{\mathrm d \pi} \right\|_{L^2_\pi(\mathcal{X})} 
& \le  
\left \| \frac{\mathrm d \nu^n }{\mathrm d\pi }\right\|_{L^2_\pi(\mathcal{X})} + \tau
\left \|\frac{\delta F}{\delta \nu}(\nu^n, \mu^n, \cdot)\right\|_{L^2_\pi(\mathcal{X})} \leq \left \| \frac{\mathrm d \nu^n }{\mathrm d\pi }\right\|_{L^2_\pi(\mathcal{X})} +\tau C_1,
\end{align*}
\begin{align*}
\left\|\frac{\mathrm d \mu^{n+1}}{\mathrm d \pi} \right\|_{L^2_\pi(\mathcal{X})} 
& \le  
\left \| \frac{\mathrm d \mu^n }{\mathrm d\pi }\right\|_{L^2_\pi(\mathcal{X})} + \tau
\left \|\frac{\delta F}{\delta \mu}(\nu^n, \mu^n, \cdot)\right\|_{L^2_\pi(\mathcal{X})}
 \leq \left \| \frac{\mathrm d \mu^n }{\mathrm d\pi }\right\|_{L^2_\pi(\mathcal{X})} + \tau C_2.
\end{align*}
 Iterating from $n=0$ and using $\nu^0,\mu^0\in \mathcal{F}_{\eta_0}$ gives, for all $n=0,\dots,N$,
\[
\left \| \frac{\mathrm d \nu^n }{\mathrm d\pi }\right\|_{L^2_\pi(\mathcal{X})}\le \eta_0 + n\tau C_1\le \bar\eta,
\quad
\left \| \frac{\mathrm d \mu^n }{\mathrm d\pi }\right\|_{L^2_\pi(\mathcal{X})} \le \eta_0 + n\tau C_2\le \bar\eta.
\]
Consider the normalized affine candidate
\begin{equation*}
\label{eq:chi-square-nu-candidate}
    \frac{\mathrm d\widehat \nu^{n+1}}{\mathrm{d} \pi} :=
    \frac{\mathrm d \nu^n}{\mathrm d \pi}-\tau\left(\frac{\delta F}{\delta \nu}(\nu^n, \mu^n, \cdot)- \int_{\mathcal{X}}\frac{\delta F}{\delta \nu}(\nu^n, \mu^n, x)\pi(\mathrm{d}x)\right).
\end{equation*}
It belongs to $L^2_\pi(\mathcal{X})$ and has unit $\pi$-mass. Moreover, if \eqref{eq:chi-square-positivity-bounds} holds at iteration $n$,
\begin{equation*}
\begin{aligned}
    \frac{\mathrm d\widehat \nu^{n+1}}{\mathrm{d} \pi} &\geq
    \frac{\mathrm d \nu^n}{\rm d \pi}-\tau\left|\frac{\delta F}{\delta \nu}(\nu^n, \mu^n, \cdot)- \int_{\mathcal{X}}\frac{\delta F}{\delta \nu}(\nu^n, \mu^n, x)\pi(\mathrm{d}x)\right|\\
    &\geq
    \underline\kappa_\nu-2n\tau C_1-2\tau C_1\\
    &=
    \underline\kappa_\nu-2(n+1)\tau C_1
    >0,
\label{eq:chi-square-nu-positive}
\end{aligned}
\end{equation*}
where the last inequality follows from Assumption \ref{assumption:bddflatF} and $n+1\leq N$. Thus, $\frac{\mathrm d\widehat \nu^{n+1}}{\mathrm{d} \pi} \in\mathfrak F$. 

For every $\phi \in \mathfrak F$, we have
\begin{align*}
    &\left\langle
        \frac{\mathrm d\widehat \nu^{n+1}}{\mathrm{d} \pi}-\left(\frac{\mathrm d \nu^n}{\rm d \pi}-\tau \frac{\delta F}{\delta \nu}(\nu^n, \mu^n, \cdot)\right),
        \phi-\frac{\mathrm d\widehat \nu^{n+1}}{\mathrm{d} \pi}
    \right\rangle_{L^2_\pi}\\ 
    &=
    \tau\int_{\mathcal{X}}\frac{\delta F}{\delta \nu}(\nu^n, \mu^n, x)\pi(\mathrm{d}x)
    \int_{\mathcal X}
        \left(\phi(x)- \frac{\mathrm d\widehat \nu^{n+1}}{\mathrm{d} \pi}(x)\right)\pi(\mathrm{d}x) =0.
\end{align*}
Consequently,
\[
   \frac{\mathrm d \widehat\nu^{n+1}}{\mathrm d \pi} = \Pi_{\mathfrak F}\left( \frac{\mathrm d \nu^n }{\mathrm d\pi }  -\tau\frac{\delta F}{\delta \nu}(\nu^n, \mu^n, \cdot)  \right).
\]
In addition,
\begin{equation*}
\begin{aligned}
    \left\|\frac{\mathrm d \widehat\nu^{n+1}}{\mathrm d \pi}\right\|_{L^2_\pi}
    &\leq
    \left\|\frac{\mathrm d \nu^n}{\mathrm d \pi}\right\|_{L^2_\pi}
    +
    \tau\left\|\frac{\delta F}{\delta \nu}(\nu^n, \mu^n, \cdot)- \int_{\mathcal{X}}\frac{\delta F}{\delta \nu}(\nu^n, \mu^n, x)\pi(\mathrm{d}x)\right\|_{L^2_\pi}\\
    &\leq
    \eta_0+(n+1)\tau C_1
    \leq \bar\eta,
\end{aligned}
\end{equation*}
for every $0 \leq n \leq N-1$. Therefore, $\frac{\mathrm d \widehat\nu^{n+1}}{\mathrm d \pi}$ belongs to $\mathcal F_{\bar\eta}$. Since it minimizes the update objective over the larger set $\mathfrak F$, it also minimizes it over $\mathcal F_{\bar\eta}$ because $\frac{\mathrm d \widehat\nu^{n+1}}{\mathrm d \pi} \in \mathcal F_{\bar\eta}$. Hence, by Lemma~\ref{lemma:singleton-chi^2},
\begin{align*}
    \partial_{\mathcal C}h(\nu^{n+1})
    &= \left\{\frac{\mathrm d\nu^{n+1}}{\mathrm d\pi}(\cdot)+c_{\nu,n}: c_{\nu,n}\in\mathbb R\right\},\\
    \partial_{\mathcal D}h(\mu^{n+1}) &= \left\{\frac{\mathrm d\mu^{n+1}}{\mathrm d\pi}(\cdot)+c_{\mu,n}: c_{\mu,n}\in\mathbb R\right\}.
\end{align*}
Combining this with Proposition~\ref{prop:foc} gives
\begin{equation*}
    \frac{\mathrm d \nu^{n+1}}{\mathrm d \pi} = \frac{\mathrm d \nu^n}{\mathrm d \pi} - \tau\left(\frac{\delta F}{\delta \nu}(\nu^n, \mu^n, \cdot)- \int_{\mathcal{X}}\frac{\delta F}{\delta \nu}(\nu^n, \mu^n, x)\pi(\mathrm{d}x)\right).
\end{equation*}
The same argument for the maximizing player gives
\begin{equation*}
\frac{\mathrm d \mu^{n+1}}{\mathrm d \pi} = \frac{\mathrm d \mu^n}{\mathrm d \pi} + \tau\left(\frac{\delta F}{\delta \mu}(\nu^n, \mu^n, \cdot)- \int_{\mathcal{X}}\frac{\delta F}{\delta \mu}(\nu^n, \mu^n, x)\pi(\mathrm{d}x)\right).
\end{equation*}
because this update is the projection of $\frac{\mathrm d \mu^n}{\mathrm d \pi} + \tau\frac{\delta F}{\delta \mu}(\nu^n, \mu^n, \cdot)$ onto
$\mathfrak F$. In particular,
\[
    \frac{\mathrm d \mu^{n+1}}{\mathrm d \pi} \geq \underline\kappa_\mu-2(n+1)\tau C_2 >0
\]
and
\[
    \left\|\frac{\mathrm d \mu^{n+1}}{\mathrm d \pi}\right\|_{L^2_\pi} \leq
    \eta_0+(n+1)\tau C_2 \leq\bar\eta.
\]
Starting from $n=0$, induction proves \eqref{eq:chi-square-radius-bounds} and
\eqref{eq:chi-square-positivity-bounds}.

Hence, $\nu^n,\mu^n\in \mathcal{F}_{\bar\eta}=\mathcal{C}=\mathcal{D}$ for all $n=0,\dots,N$, i.e.,
\[
(\nu^n,\mu^n)_{n=0}^N \subseteq \mathcal{F}_{\bar\eta} \times \mathcal{F}_{\bar\eta}.
\]
Fix any $\nu\in\mathcal{F}_{\bar\eta}$. Since $\nu^0\in\mathcal{F}_{\eta_0}$, we have
\[
\left \| \frac{\mathrm d \nu }{\mathrm d\pi }\right\|_{L^2_\pi(\mathcal{X})}\le \bar\eta,
\quad
\left \| \frac{\mathrm d \nu^0 }{\mathrm d\pi }\right\|_{L^2_\pi(\mathcal{X})}\le \eta_0,
\]
and therefore 
\[
\frac{1}{2}\left\|\frac{\nu(\cdot)}{\pi(\cdot)} - \frac{\nu^0(\cdot)}{\pi(\cdot)}\right\|^2_{L_\pi^2(\mathcal{X})} \le \bar\eta^2+\eta_0^2.
\]
Thus, $\sup_{\nu\in\mathcal{C}}\frac{1}{2}\left\|\frac{\nu(\cdot)}{\pi(\cdot)} - \frac{\nu^0(\cdot)}{\pi(\cdot)}\right\|^2_{L_\pi^2(\mathcal{X})} \le \bar\eta^2+\eta_0^2<\infty$.
The same argument with $\mu\in\mathcal{F}_{\bar\eta}$ and $\mu^0\in\mathcal{F}_{\eta_0}$ gives $\sup_{\mu\in\mathcal{D}}\frac{1}{2}\left\|\frac{\mu(\cdot)}{\pi(\cdot)} - \frac{\mu^0(\cdot)}{\pi(\cdot)}\right\|^2_{L_\pi^2(\mathcal{X})} \le \bar\eta^2+\eta_0^2<\infty$. Combining the two bounds proves
\[
\frac{1}{2}\sup_{\nu \in \mathcal{C}}\left\|\frac{\nu(\cdot)}{\pi(\cdot)} - \frac{\nu^0(\cdot)}{\pi(\cdot)}\right\|^2_{L_\pi^2(\mathcal{X})} + \frac{1}{2}\sup_{\mu \in \mathcal{D}}\left\|\frac{\mu(\cdot)}{\pi(\cdot)} - \frac{\mu^0(\cdot)}{\pi(\cdot)}\right\|^2_{L_\pi^2(\mathcal{X})} < \infty.
\]
\end{proof}

When $F$ is strongly-convex-strongly-concave relative to $h$ and Assumption \ref{assumption:assump-h} holds, it can be shown that $(\nu^*, \mu^*)$ is the unique MNE of \eqref{eq:game} (see the proof of \cite[Lemma 6]{lascu2023entropic}). Moreover, based on relative convexity-concavity of $F,$ we prove in Lemma \ref{lemma:quad-growth-NI} that the NI error satisfies a type of ``quadratic growth'' inequality relative to $h.$
\begin{assumption}[Relative convexity-concavity]
\label{def: def-F-strong-conv-conc}
    Assume $F$ is $(\ell_{\nu}, \ell_{\mu})$-strongly convex-concave relative to $h,$ i.e., there exist $\ell_{\nu}, \ell_{\mu} > 0$ such that for any $\nu, \nu' \in \mathcal{C}$ and $\mu, \mu' \in \mathcal{D},$ we have
		\begin{equation}
        \label{eq:strong-convexF}
		D_{F(\cdot,\mu)}(\nu',\nu) = F(\nu', \mu) - F(\nu, \mu) - \int_{\mathcal{X}} \frac{\delta F}{\delta \nu}(\nu,\mu,x) (\nu'-\nu)(\mathrm{d}x) \geq \ell_{\nu}D_h(\nu',\nu), 
		\end{equation}
        \begin{equation}
        \label{eq:strong-concaveF}
        D_{F(\nu, \cdot)}(\mu',\mu) = F(\nu, \mu') - F(\nu, \mu) - \int_{\mathcal{X}} \frac{\delta F}{\delta \mu}(\nu,\mu,y) (\mu'-\mu)(\mathrm{d}y) \leq - \ell_{\mu}D_h(\mu',\mu).
        \end{equation}
\end{assumption}
\begin{lemma}[``Quadratic growth'' of NI error relative to $h$]
\label{lemma:quad-growth-NI}
    Assume that \(\mathcal C,\mathcal D\) are convex and that \((\nu^*,\mu^*)\in\mathcal C\times\mathcal D\) is an MNE of $F$. Suppose that Assumption \ref{assumption:assump-h} and \ref{def: def-F-strong-conv-conc} hold. Then, for all $(\nu, \mu) \in \mathcal{C} \times \mathcal{D},$ it holds that
    \begin{equation*}
        \operatorname{NI}\left(\nu, \mu\right) \geq \ell\left(D_h(\nu, \nu^*) + D_h(\mu, \mu^*)\right),
    \end{equation*}
    where $\ell \coloneqq \min\{\ell_{\nu}, \ell_{\mu}\}.$
\end{lemma}
\begin{remark}
    We refer to Lemma \ref{lemma:quad-growth-NI} as ``quadratic growth'' of NI error relative to $h$ due to the similar notion of quadratic growth of a convex function relative to the squared Euclidean norm on $\mathbb R^d$ (see e.g. \cite{anitescu}).
\end{remark}
\begin{proof}
Fix $(\nu,\mu)\in\mathcal C\times\mathcal D$. Since $(\nu^*,\mu^*)$ is a saddle point, we have
\begin{equation*}
    F(\nu^*,\tilde\mu)\leq F(\nu^*,\mu^*)\leq F(\tilde\nu,\mu^*), \quad
    (\tilde\nu,\tilde\mu)\in\mathcal C\times\mathcal D.
\end{equation*}
By convexity of $\mathcal C$ and $\mathcal D$, the measures $\nu_t=(1-t)\nu^*+t\nu,$ $\mu_t=(1-t)\mu^*+t\mu$ are feasible for every $t\in[0,1]$. Applying the saddle point inequalities to these measures, dividing by $t>0$, and letting $t\downarrow0$ gives
\begin{align}
    \int_{\mathcal X}\frac{\delta F}{\delta\nu}(\nu^*,\mu^*,x)(\nu-\nu^*)(\mathrm dx) &\geq 0,
    \label{eq:NI-growth-optimality-nu}
    \\
    \int_{\mathcal X} \frac{\delta F}{\delta\mu}(\nu^*,\mu^*,y)(\mu-\mu^*)(\mathrm dy)&\leq 0.
    \label{eq:NI-growth-optimality-mu}
\end{align}
Applying \eqref{eq:strong-convexF} with base point $\nu^*$, comparison point $\nu$, and fixed second argument $\mu^*$, and using \eqref{eq:NI-growth-optimality-nu}, we obtain
\begin{align*}
    F(\nu,\mu^*)-F(\nu^*,\mu^*) &\geq \int_{\mathcal X}\frac{\delta F}{\delta\nu}(\nu^*,\mu^*,x)(\nu-\nu^*)(\mathrm dx) +\ell_\nu D_h(\nu,\nu^*) \geq \ell_\nu D_h(\nu,\nu^*).
\end{align*}
Similarly, applying \eqref{eq:strong-concaveF} with base point $\mu^*$, comparison point $\mu$, and fixed first argument $\nu^*$, and using \eqref{eq:NI-growth-optimality-mu}, gives
\begin{align*}
    F(\nu^*,\mu^*)-F(\nu^*,\mu) &\geq -\int_{\mathcal X}\frac{\delta F}{\delta\mu}(\nu^*,\mu^*,y)(\mu-\mu^*)(\mathrm dy)+\ell_\mu D_h(\mu,\mu^*) \geq \ell_\mu D_h(\mu,\mu^*).
\end{align*}
Hence, by the definition of the Nikaid\^{o}--Isoda error,
\begin{align*}
    \operatorname{NI}(\nu,\mu) &= \sup_{\tilde\mu\in\mathcal D}F(\nu,\tilde\mu) - \inf_{\tilde\nu\in\mathcal C}F(\tilde\nu,\mu) \geq F(\nu,\mu^*)-F(\nu^*,\mu)
    \\ &= \left[F(\nu,\mu^*)-F(\nu^*,\mu^*)\right] + \left[F(\nu^*,\mu^*)-F(\nu^*,\mu)\right]
    \\ &\geq \ell_\nu D_h(\nu,\nu^*) + \ell_\mu D_h(\mu,\mu^*) \geq \ell\left(D_h(\nu,\nu^*)+D_h(\mu,\mu^*)\right),
\end{align*}
where the last inequality follows from the nonnegativity of the Bregman divergences and $\ell=\min\{\ell_\nu,\ell_\mu\}$.
\end{proof}
By Lemma \ref{lemma:quad-growth-NI}, the time-averaged iterates $\left(\frac{1}{N}\sum_{n=0}^{N-1}\nu^{n}, \frac{1}{N}\sum_{n=0}^{N-1}\mu^{n}\right)$ converge in Bregman divergence to the unique MNE $(\nu^*, \mu^*)$ of \eqref{eq:game} with the rates proved in Theorem \ref{thm: conv-sim-bregman} and Theorem \ref{thm:conv-alt-bregman}, respectively.

\section{Verification of Assumption \ref{def: def-F-conv-conc}, \ref{assumption:bddflatF} and \ref{def:relative-smoothness} for the examples in Section \ref{example:GAN-example}}
\label{appendix:verification-GAN}

In this section we verify that Assumptions \ref{def: def-F-conv-conc}, \ref{assumption:bddflatF} and \ref{def:relative-smoothness} are satisfied by the objective function $F$ in Section \ref{example:GAN-example}.

\begin{proposition}[Verification of assumptions for the Wasserstein GAN]
\label{prop:verification-GANs}
Let $\mathcal{Y}, \mathcal{Z} \subseteq \mathbb R^d$, with $\hat{\xi} \in \mathcal{P}(\mathcal{Y})$ and $\xi \in \mathcal{P}(\mathcal{Z})$. Suppose $T_{\theta}: \mathcal{Z} \to \mathcal{Y}$ is measurable with $\theta \in \Theta \subseteq \mathbb R^d$, and for each $w \in \mathcal{W} \subseteq \mathbb R^d$, the critic $D_{w}: \mathcal{Y} \to \mathbb R$ is measurable and $1$-Lipschitz, i.e., $|D_w(y)-D_w(y')| \le |y-y'|$, for all $y,y'\in\mathcal Y$. Assume moreover that $\mathcal Y$ has finite diameter, i.e., $\mathrm{diam}(\mathcal Y) \coloneqq \sup_{y,y'\in\mathcal Y}\|y-y'\| < \infty$, (e.g., if $\mathcal Y$ is compact), and fix $y_0\in\mathcal Y$. Then Assumptions \ref{def: def-F-conv-conc}, \ref{assumption:bddflatF} and \ref{def:relative-smoothness} are satisfied by the objective
\[
F(\nu, \mu) \coloneqq \int_{\mathcal{W}} \int_{\Theta} f(\theta, w)\, \nu(\mathrm{d}\theta)\, \mu(\mathrm{d}w)
\]
from Example \ref{example:GAN-example}.
\end{proposition}
\begin{proof}
By Definition \ref{def:fderivative},
\[
\frac{\delta F}{\delta \nu}(\nu,\mu,\theta) = \int_{\mathcal{W}} f(\theta, w)\, \mu(\mathrm{d}w),
\qquad
\frac{\delta F}{\delta \mu}(\nu,\mu,w) = \int_{\Theta} f(\theta, w)\, \nu(\mathrm{d}\theta).
\]
Therefore, Assumption \ref{def: def-F-conv-conc} holds with equality and Assumption \ref{def:relative-smoothness} holds with equality with $L_\nu = L_\mu = 0$.

It remains to verify uniform boundedness of $f$ and hence of $F$ and its flat derivatives. Recall
\[
f(\theta,w)=\int_{\mathcal Y} D_w(y)\,\bigl(T_\theta\#\xi-\hat\xi\bigr)(\mathrm dy)
= \int_{\mathcal Y} D_w(y)\,(T_\theta\#\xi)(\mathrm dy)-\int_{\mathcal Y} D_w(y)\,\hat\xi(\mathrm dy).
\]
Since adding a constant to $D_w$ does not change $f(\theta,w)$ because $(T_\theta\#\xi-\hat\xi)$ has total mass $0$, we may replace $D_w$ by $
\widetilde D_w(y)\coloneqq D_w(y)-D_w(y_0)$, so that $\widetilde D_w(y_0)=0$ and $\widetilde D_w$ is still $1$-Lipschitz. Then for all $y\in\mathcal Y$,
\[
|\widetilde D_w(y)|
=|D_w(y)-D_w(y_0)|
\le |y-y_0|
\le \mathrm{diam}(\mathcal Y).
\]
Hence, using $\hat\xi\in\mathcal P(\mathcal Y)$ and $T_\theta\#\xi\in\mathcal P(\mathcal Y)$,
\begin{align*}
|f(\theta,w)|
&= \left|\int_{\mathcal Y} \widetilde D_w(y)\,\bigl(T_\theta\#\xi-\hat\xi\bigr)(\mathrm dy)\right|\\
&\le \int_{\mathcal Y} |\widetilde D_w(y)|\,(T_\theta\#\xi)(\mathrm dy)
   +\int_{\mathcal Y} |\widetilde D_w(y)|\,\hat\xi(\mathrm dy)
\le 2\mathrm{diam}(\mathcal Y).
\end{align*}
Consequently, for all $\nu,\mu\in\mathcal P(\mathcal X)$ and all $\theta\in\Theta$, $w\in\mathcal W$,
\[
\left|\frac{\delta F}{\delta \nu}(\nu,\mu,\theta)\right| \le 2\,\mathrm{diam}(\mathcal Y), \quad
\left|\frac{\delta F}{\delta \mu}(\nu,\mu,w)\right| \le 2\,\mathrm{diam}(\mathcal Y).
\]
Thus Assumption \ref{assumption:bddflatF} holds with
\[
C_1 = C_2 = 2\,\mathrm{diam}(\mathcal Y).
\]
\end{proof}

\begin{proposition}[Verification of assumptions for RLHF]
\label{prop:nlhf-assumptions}
Assume $\mathcal{S}, \mathcal{A}$ are measurable and $P(\cdot\succ\cdot|x)\in[0,1]$ for all $x \sim \rho \in \mathcal{P}(\mathcal{S})$. Then the game with $F= P(\mu \succ \nu)$ satisfies Assumptions \ref{def: def-F-conv-conc}, \ref{assumption:bddflatF} and \ref{def:relative-smoothness}.
\end{proposition}
\begin{proof}
The map $\mu \mapsto P(\mu\succ \nu)$ is linear, hence concave, and $\nu \mapsto P(\mu\succ \nu)$ is also linear, hence convex. Therefore $F$ is convex in $\nu$ and concave in $\mu$, as required by Assumption~ \ref{def: def-F-conv-conc}. Since $F$ is bilinear, it admits the flat derivatives
\[
\frac{\delta F}{\delta \nu}(\mu \succ \nu,y')
= \mathbb E_{x\sim \rho, y\sim \mu(\cdot|x)}\left[P(y\succ y'| x)\right],
\]
\[
\frac{\delta F}{\delta \mu}(\mu\succ \nu ,y)
= \mathbb E_{x\sim \rho, y'\sim \nu(\cdot|x)}\left[P(y\succ y'| x)\right].
\]
Because $P(\cdot\succ\cdot|x)\in[0,1]$, both derivatives are bounded by $1$, yielding Assumption \ref{assumption:bddflatF}. Since $F$ is bilinear it satisfies Assumption \ref{def:relative-smoothness} with equality and $L_\nu=L_\mu=0$.
\end{proof} 

For the adversarial NN example, first, we prove that $F_\varepsilon(\cdot,\mu)$ is $L$-smooth in $\nu$, so it satisfies the first condition of Assumption  \ref{def:relative-smoothness}.
\begin{proposition}[$L$-smoothness of $F_\varepsilon(\cdot,\mu)$]\label{prop:Dh_bound_nu}
Let Assumption \ref{assumption:assump-h} hold. Let $\mu$ be a probability measure on $\mathcal Y\times\mathcal Z$, let $\nu,\nu'$ be probability measures on $\mathbb R^d$ and $\varepsilon \geq 0$. Assume $\hat \varphi$ is bounded by $M_\varphi > 0$. Then
\[
F_\varepsilon(\nu',\mu)-F_\varepsilon(\nu,\mu)-\int_{\mathbb R^d}\frac{\delta F_\varepsilon}{\delta\nu}(\nu,\mu,x)\,(\nu'-\nu)(\mathrm{d}x)\le \frac{4M_\varphi}{\alpha}D_h(\nu',\nu).
\]
In particular, one may take $L_\nu=\frac{4M_\varphi}{\alpha}$ in Assumption \ref{def:relative-smoothness}.
\end{proposition}
\begin{proof}
For $z\in\mathcal Z$, define $a_\nu(z):=\int_{\mathbb R^d}\hat\varphi(x,z)\,\nu(\mathrm{d}x)$ and $\Delta a(z):=a_{\nu'}(z)-a_\nu(z)$.
Note that
\[
\frac{\delta F_\varepsilon}{\delta\nu}(\nu,\mu,x) =\int_{\mathcal Y\times\mathcal Z}\Big(a_\nu(z)-y\Big)\hat\varphi(x,z)\,\mu(\mathrm{d}y,\mathrm{d}z).
\]
Fix $(y,z) \in \mathcal{Y} \times \mathcal{Z}$. Then
$y-a_{\nu'}(z)=y-a_\nu(z)-\Delta a(z)$ and thus
\begin{align*}
\frac12\big(y-a_{\nu'}(z)\big)^2-\frac12\big(y-a_\nu(z)\big)^2
&=\frac12\Big((y-a_\nu(z)-\Delta a(z))^2-(y-a_\nu(z))^2\Big)\\
&=(a_\nu(z)-y)\Delta a(z)+\frac12(\Delta a(z))^2.
\end{align*}
Integrating against $\mu(\mathrm{d}y,\mathrm{d}z)$ gives
\begin{align}\label{eq:exact_remainder}
&\frac12\int_{\mathcal Y\times\mathcal Z}\big(y-a_{\nu'}(z)\big)^2\,\mu(\mathrm{d}y,\mathrm{d}z)
-\frac12\int_{\mathcal Y\times\mathcal Z}\big(y-a_\nu(z)\big)^2\,\mu(\mathrm{d}y,\mathrm{d}z)\nonumber\\
&\qquad=
\int_{\mathcal Y\times\mathcal Z}(a_\nu(z)-y)\Delta a(z)\,\mu(\mathrm{d}y,\mathrm{d}z)
+\frac12\int_{\mathcal Y\times\mathcal Z}(\Delta a(z))^2\,\mu(\mathrm{d}y,\mathrm{d}z).
\end{align}
Next, by Fubini,
\begin{align*}
\int_{\mathcal Y\times\mathcal Z}(a_\nu(z)-y)\Delta a(z)\,\mu(\mathrm{d}y,\mathrm{d}z)
&=\int_{\mathcal Y\times\mathcal Z}(a_\nu(z)-y)
\left(\int_{\mathbb R^d}\hat\varphi(x,z)\,(\nu'-\nu)(\mathrm{d}x)\right)\mu(\mathrm{d}y,\mathrm{d}z)\\
&=\int_{\mathbb R^d}
\left[\int_{\mathcal Y\times\mathcal Z}(a_\nu(z)-y)\hat\varphi(x,z)\,\mu(\mathrm{d}y,\mathrm{d}z)\right]
(\nu'-\nu)(\mathrm{d}x)\\
&=\int_{\mathbb R^d}\frac{\delta F_\varepsilon}{\delta\nu}(\nu,\mu,x)\,(\nu'-\nu)(\mathrm{d}x).
\end{align*}
Since the term $-\operatorname{TV}_\varepsilon^2(\mu,\hat\mu)$ does not depend on $\nu$, it cancels in the difference
$F_\varepsilon(\nu',\mu)-F_\varepsilon(\nu,\mu)$, and \eqref{eq:exact_remainder} implies the identity
\begin{equation}\label{eq:remainder_identity_final}
F_\varepsilon(\nu',\mu)-F_\varepsilon(\nu,\mu)-\int_{\mathbb R^d}\frac{\delta F_\varepsilon}{\delta\nu}(\nu,\mu,x)\,(\nu'-\nu)(\mathrm dx)
=\frac12\int_{\mathcal Y\times\mathcal Z}(\Delta a(z))^2\,\mu(\mathrm{d}y,\mathrm{d}z).
\end{equation}
In particular, the left-hand side is nonnegative, so the absolute value equals the same quantity.

For each $z\in\mathcal Z$,
\[
|\Delta a(z)|
=\left|\int_{\mathbb R^d}\hat\varphi(x,z)\,(\nu'-\nu)(dx)\right|
\le 2M_\varphi\operatorname{TV}(\nu',\nu),
\]
where we used the dual characterization of total variation. Squaring and inserting into \eqref{eq:remainder_identity_final} gives
\begin{align*}
&F_\varepsilon(\nu',\mu)-F_\varepsilon(\nu,\mu)-\int_{\mathbb R^d}\frac{\delta F_\varepsilon}{\delta\nu}(\nu,\mu,x)\,(\nu'-\nu)(\mathrm dx)\\
&=\frac12\int_{\mathcal Y\times\mathcal Z}(\Delta a(z))^2\,\mu(\mathrm dy,\mathrm dz)\le 2M_\varphi\operatorname{TV}^2(\nu',\nu).
\end{align*}
By Assumption \ref{assumption:assump-h},
\[
F_\varepsilon(\nu',\mu)-F_\varepsilon(\nu,\mu)-\int_{\mathbb R^d}\frac{\delta F_\varepsilon}{\delta\nu}(\nu,\mu,x)\,(\nu'-\nu)(\mathrm dx)
\le \frac{4M_\varphi}{\alpha}\,D_h(\nu',\nu),
\]
which proves the claim.
\end{proof}

Now, we prove that $\operatorname{TV}_\varepsilon^2$ is flat convex.
\begin{lemma}[Convexity of the squared smoothed TV]\label{lem:convex-tv-eps-square}
Let $\hat\mu$ be a fixed reference probability measure on $\mathcal{Y} \times \mathcal{Z}$ with Lebesgue density $\hat m$. Fix $\varepsilon>0$ and define $\rho_\varepsilon(u):=\sqrt{u^2+\varepsilon^2}$, so that 
\begin{equation*}
    \operatorname{TV}_\varepsilon(\mu,\hat\mu):=\frac12\int_{\mathcal{Y} \times \mathcal{Z}} \rho_\varepsilon\big(m(r)-\hat m(r)\big)\mathrm{d}r,
\end{equation*}
for any $\mu$ absolutely continuous with respect to Lebesgue with density $m$. Then the function
\[
\mu \mapsto \operatorname{TV}_\varepsilon^2(\mu,\hat\mu)
\]
is flat convex on the set of absolutely continuous probability measures.
\end{lemma}
\begin{proof}
We first show that $\mu\mapsto\operatorname{TV}_\varepsilon(\mu,\hat\mu)$ is convex.
Since $\rho_\varepsilon$ is twice continuously differentiable with
\[
\rho_\varepsilon''(u)=\frac{\varepsilon^2}{(u^2+\varepsilon^2)^{3/2}}\ge 0,
\]
it is convex on $\mathbb{R}$. Let $\mu_0,\mu_1$ be absolutely continuous measures with densities $m_0,m_1$ and let $\lambda\in[0,1]$. The convex combination $\mu_\lambda:=\lambda\mu_0+(1-\lambda)\mu_1$ is absolutely continuous with density $m_\lambda=\lambda m_0+(1-\lambda)m_1$. Using convexity of $\rho_\varepsilon$ pointwise, for a.e. $r\in \mathcal{Y} \times \mathcal{Z}$ we have
\begin{align*}
\rho_\varepsilon\!\big(m_\lambda(r)-\hat m(r)\big)
&=\rho_\varepsilon\!\big(\lambda(m_0(r)-\hat m(r))+(1-\lambda)(m_1(r)-\hat m(r))\big)\\
&\le \lambda \rho_\varepsilon\!\big(m_0(r)-\hat m(r)\big)
+(1-\lambda)\rho_\varepsilon\!\big(m_1(r)-\hat m(r)\big).
\end{align*}
Integrating and multiplying by $\frac{1}{2}$ yields
\[
\operatorname{TV}_\varepsilon(\mu_\lambda,\hat\mu)
\le \lambda\,\operatorname{TV}_\varepsilon(\mu_0,\hat\mu)+(1-\lambda)\,\operatorname{TV}_\varepsilon(\mu_1,\hat\mu),
\]
so $\operatorname{TV}_\varepsilon(\cdot,\hat\mu)$ is convex. Moreover $\operatorname{TV}_\varepsilon(\mu,\hat\mu)\ge 0$
for all $\mu$.

Now define $f(\mu):=\operatorname{TV}_\varepsilon(\mu,\hat\mu)$. Since $f\ge 0$ and the map
$\phi(t)=t^2$ is nondecreasing on $[0,\infty)$, from convexity of $f$ we obtain
\[
f(\mu_\lambda)^2
\le \big(\lambda f(\mu_0)+(1-\lambda)f(\mu_1)\big)^2.
\]
Finally, using convexity of $\phi(t)=t^2$ on $\mathbb{R}$ gives
\[
f(\mu_\lambda)^2
\le \lambda f(\mu_0)^2+(1-\lambda)f(\mu_1)^2,
\]
i.e.,
\[
\operatorname{TV}_\varepsilon(\mu_\lambda,\hat\mu)^2
\le \lambda \operatorname{TV}_\varepsilon(\mu_0,\hat\mu)^2+(1-\lambda)\operatorname{TV}_\varepsilon(\mu_1,\hat\mu)^2.
\]
\end{proof}
Next, we prove an upper bound on $\operatorname{TV}_\varepsilon$.
\begin{lemma}[Upper bound on $\operatorname{TV}_\varepsilon$]\label{lem:TvEps_upper_bound}
Let ${\mathcal{Y} \times \mathcal{Z}}\subseteq\mathbb{R}^d$ be compact and convex and let $\mu, \hat\mu$ be absolutely continuous with
densities $m, \hat m$, respectively. Then
\begin{equation}\label{eq:TvEps_bound_finite_volume}
\operatorname{TV}_\varepsilon(\mu,\hat\mu)\ \le\ 1+\frac{\varepsilon}{2}\,|{\mathcal{Y} \times \mathcal{Z}}|.
\end{equation}
\end{lemma}
\begin{proof}
We use the elementary inequality, valid for all $u\in\mathbb{R}$,
\begin{equation}\label{eq:sqrt_ineq}
\sqrt{u^2+\varepsilon^2}\le |u|+\varepsilon.
\end{equation}
Indeed, squaring both sides gives $u^2+\varepsilon^2\le u^2+2\varepsilon|u|+\varepsilon^2$.

Applying \eqref{eq:sqrt_ineq} with $u=m-\hat m$ and integrating,
\begin{align*}
\operatorname{TV}_\varepsilon(\mu,\hat\mu)
&=\frac12\int_{\mathcal{Y} \times \mathcal{Z}} \sqrt{(m-\hat m)^2+\varepsilon^2}\mathrm dr\\
&\le \frac12\int_{\mathcal{Y} \times \mathcal{Z}} |m-\hat m|\mathrm dr+\frac{\varepsilon}{2}|{\mathcal{Y} \times \mathcal{Z}}|
=\operatorname{TV}(\mu,\hat\mu)+\frac{\varepsilon}{2}\,|{\mathcal{Y} \times \mathcal{Z}}|.
\end{align*}
Since $\mu,\hat\mu$ are probability measures, then $\operatorname{TV}(\mu,\hat\mu)\le 1$, yielding
\eqref{eq:TvEps_bound_finite_volume}.
\end{proof}

\begin{proposition}[Verification of assumptions for adversarial NN]
\label{prop:verification-adversarial}
    Let $\mathcal{Y} \subseteq \mathbb R$ and $\mathcal{Z} \subseteq \mathbb R^{d-1}$ be compact and convex with $\hat{\mu} \in \mathcal{P}(\mathcal{Y} \times \mathcal{Z}).$ Let $h$ be the \(\chi^2\)-divergence in Example \ref{example:chi-squared}. For $x \coloneqq (w,b) \in \mathbb R^d$ and $z \in \mathbb R^{d-1},$ let $\hat{\varphi}(x,z) \coloneqq \ell(b)\varphi(w \cdot z),$ where $\ell: \mathbb R \to [-K,K]$ is a clipping function with clipping threshold $K > 0$ and $\varphi:\mathbb R \to \mathbb R$ is a bounded, continuous, non-constant function. Let $\varepsilon > 0 $. Then Assumptions \ref{def: def-F-conv-conc}, \ref{assumption:bddflatF} and \ref{def:relative-smoothness} are satisfied by the objective 
    \begin{equation*}
        F_\varepsilon(\nu, \mu) = \frac{1}{2}\int_{\mathcal{Y} \times \mathcal{Z}} \left|y-\mathbb E^{X \sim \nu}[\hat{\varphi}(X, z)]\right|^2 \mu(\mathrm{d}y, \mathrm{d}z) - \operatorname{TV}_\varepsilon^2(\mu,\hat{\mu}).
    \end{equation*}
\end{proposition}
\begin{proof}
    Observe that by linearity of the expectation in $\nu$ and convexity of $|\cdot|^2,$ the function $$F^0(\nu,\mu) = \frac{1}{2}\int_{\mathcal{Y} \times \mathcal{Z}} \left|y-\mathbb E^{X \sim \nu}[\hat{\varphi}(X, z)]\right|^2 \mu(\mathrm{d}y, \mathrm{d}z)$$ satisfies the flat-convexity condition $$F^0((1-\lambda)\nu + \lambda \nu',\mu) \leq (1-\lambda)F^0(\nu,\mu) + \lambda F^0(\nu',\mu),$$ for any $\nu, \nu' \in \mathcal{P}(\mathbb R^d),$ $\mu \in \mathcal{P}(\mathcal{Y} \times \mathcal{Z})$ and any $\lambda \in [0,1].$ Since $F^0(\cdot,\mu) \in \mathfrak{C}^1(\mathbb R^d),$ by \cite[Lemma 4.1]{10.1214/20-AIHP1140}, $\nu \mapsto F_\varepsilon(\nu,\mu)$ satisfies $D_{F_\varepsilon(\cdot,\mu)}(\nu',\nu) \geq 0$. 
    
    By Lemma \ref{lem:convex-tv-eps-square}, it holds that $\operatorname{TV}_\varepsilon^2$ is convex, that is, $$\operatorname{TV}_\varepsilon^2((1-\lambda)\mu + \lambda \mu',\hat{\mu}) \leq (1-\lambda)\operatorname{TV}^2(\mu,\hat{\mu}) + \lambda \operatorname{TV}^2(\mu',\hat{\mu}),$$ for any $\mu, \mu' \in \mathcal{P}(\mathcal{Y} \times \mathcal{Z})$ and any $\lambda \in [0,1].$ We claim that $\operatorname{TV}_\varepsilon^2(\cdot,\hat \mu) \in \mathfrak{C}^1(\mathcal{Y} \times \mathcal{Z})$. 

    Assume $\mu,\mu',\hat \mu$ are absolutely continuous with respect to Lebesgue on $\mathcal Y\times \mathcal Z$ with densities $m,m',\hat m$, respectively. Fix $\varepsilon>0$ and define $\rho_\varepsilon(u):=\sqrt{u^2+\varepsilon^2}$, $\psi_\varepsilon(u):=\rho_\varepsilon'(u)=\frac{u}{\sqrt{u^2+\varepsilon^2}}$. For a perturbed measure $\mu+\delta (\mu'-\mu)$, with $\delta > 0$,
\begin{equation*}
    \lim_{\delta \to 0}\frac{1}{\delta}\left(\operatorname{TV}_\varepsilon(\mu + \delta(\mu'-\mu), \hat{\mu}) - \operatorname{TV}_\varepsilon(\mu, \hat{\mu})\right)= \frac12\int_{\mathcal{Y} \times \mathcal{Z}} \psi_\varepsilon\big(m(r)-\hat m(r)\big)(\mu'-\mu)(\mathrm{d}r),
\end{equation*}
so the flat derivative is
\[
\frac{\delta \operatorname{TV}_\varepsilon(\cdot,\hat\mu)}{\delta\mu}(\mu,r)
=\frac12\,\psi_\varepsilon\big(m(r)-\hat m(r)\big).
\]
By the chain rule for $-\operatorname{TV}_\varepsilon^2$,
\[
-\frac{\delta \operatorname{TV}_\varepsilon^2(\cdot,\hat\mu)}{\delta\mu}(\mu,r)
=-2\operatorname{TV}_\varepsilon(\mu,\hat\mu) \frac12\,\psi_\varepsilon\big(m(r)-\hat m(r)\big)
=-\operatorname{TV}_\varepsilon(\mu,\hat\mu)\,\psi_\varepsilon\big(m(r)-\hat m(r)\big).
\]
    Also, by linearity of $F^0$ in $\mu,$ it follows from \cite[Lemma 4.1]{10.1214/20-AIHP1140} that $\mu \mapsto F_\varepsilon(\nu,\mu)$ satisfies $D_{F_\varepsilon(\nu, \cdot)}(\mu',\mu) \leq 0$. Therefore, $F_\varepsilon$ satisfies Assumption \ref{def: def-F-conv-conc}.

    To verify Assumption \ref{def:relative-smoothness}, we have by Proposition \ref{prop:Dh_bound_nu} that $F_\varepsilon(\cdot,\mu)$ satisfies the first condition of Assumption \ref{def:relative-smoothness}. For the second condition of Assumption \ref{def:relative-smoothness}, recall $\rho_\varepsilon(u):=\sqrt{u^2+\varepsilon^2}$, $\psi_\varepsilon(u):=\rho_\varepsilon'(u)=\frac{u}{\sqrt{u^2+\varepsilon^2}}$. Then $|\psi_\varepsilon(u)|\le 1$ for all $u \in \mathbb R$ and moreover
\[
\psi_\varepsilon'(u)=\frac{\varepsilon^2}{(u^2+\varepsilon^2)^{3/2}},
\]
implies
$0\le \psi_\varepsilon'(u)\le \frac1\varepsilon$, hence
$|\psi_\varepsilon(a)-\psi_\varepsilon(b)|\le \frac1\varepsilon |a-b|$, for all $a,b \in \mathbb R$.

Define the loss function
\[
\ell_\nu(y,z):=\left|y-\mathbb E^{X\sim\nu}[\hat\varphi(X,z)]\right|^2 .
\]
Since the loss term $\frac12\int \ell_\nu\,d\mu$ is linear in $\mu$, its flat derivative of $F^0$ is $\frac12\,\ell_\nu(r)$. Therefore
\begin{equation}\label{eq:flat-deriv-F-eps}
\frac{\delta F_\varepsilon}{\delta\mu}(\nu,\mu,r)
=\frac12\,\ell_\nu(r)\;-\;\operatorname{TV}_\varepsilon(\mu,\hat\mu)\,\psi_\varepsilon\big(m(r)-\hat m(r)\big).
\end{equation}
Subtracting \eqref{eq:flat-deriv-F-eps} for $\mu'$ and $\mu$ cancels the loss term, yielding
\[
\left|\frac{\delta F_\varepsilon}{\delta\mu}(\nu,\mu',r)-\frac{\delta F_\varepsilon}{\delta\mu}(\nu,\mu,r)\right|
=\big|\operatorname{TV}_\varepsilon(\mu',\hat\mu)\psi_\varepsilon(m'(r)-\hat m(r))-\operatorname{TV}_\varepsilon(\mu,\hat\mu)\psi_\varepsilon(m(r)-\hat m(r))\big|.
\]
Using the triangle inequality,
\begin{align*}
&\big|\operatorname{TV}_\varepsilon(\mu',\hat\mu)\psi_\varepsilon(m'(r)-\hat m(r))-\operatorname{TV}_\varepsilon(\mu,\hat\mu)\psi_\varepsilon(m(r)-\hat m(r))\big|\\
&\le |\operatorname{TV}_\varepsilon(\mu',\hat\mu)-\operatorname{TV}_\varepsilon(\mu,\hat\mu)||\psi_\varepsilon(m'(r)-\hat m(r))|\\
&+ \operatorname{TV}_\varepsilon(\mu,\hat\mu)|\psi_\varepsilon(m'(r)-\hat m(r))-\psi_\varepsilon(m(r)-\hat m(r))|.
\end{align*}
We bound each term. Since $\rho_\varepsilon$ is $1$-Lipschitz, $|\rho_\varepsilon(a)-\rho_\varepsilon(b)|\le |a-b|$, hence
\begin{align*}
|\operatorname{TV}_\varepsilon(\mu',\hat\mu)-\operatorname{TV}_\varepsilon(\mu,\hat\mu)|
&=\frac12\left|\int_{\mathcal{Y} \times \mathcal{Z}} \rho_\varepsilon(m'-\hat m)-\rho_\varepsilon(m-\hat m)\mathrm{d}r\right|\\
&\le \frac12\int_{\mathcal{Y} \times \mathcal{Z}} |m'-m|\mathrm{d}r
=\operatorname{TV}(\mu',\mu).
\end{align*}
Also, $|\psi_\varepsilon(m'(r)-\hat m(r))|\le 1$. By the Lipschitz bound on $\psi_\varepsilon$,
\[
\left|\psi_\varepsilon\big(m'(r)-\hat m(r)\big)-\psi_\varepsilon\big(m(r)-\hat m(r)\big)\right|
\le \frac1\varepsilon |m'(r)-m(r)|.
\]
Combining,
\begin{equation*}
\left|\frac{\delta F_\varepsilon}{\delta\mu}(\nu,\mu',r)-\frac{\delta F_\varepsilon}{\delta\mu}(\nu,\mu,r)\right|
\le \operatorname{TV}(\mu',\mu)+\frac{\operatorname{TV}_\varepsilon(\mu,\hat\mu)}{\varepsilon}\,|m'(r)-m(r)|,
\end{equation*}
for a.e. $r\in{\mathcal{Y} \times \mathcal{Z}}$.

Letting $\mu_\theta = \mu + \theta(\mu'-\mu)$, for $\theta \in [0,1]$, gives
\begin{align*}
    \left|\frac{\delta F_\varepsilon}{\delta\mu}(\nu,\mu_\theta,r) - \frac{\delta F_\varepsilon}{\delta\mu}(\nu,\mu,r)\right| &\leq \operatorname{TV}(\mu_\theta,\mu)+\frac{\operatorname{TV}_\varepsilon(\mu_\theta,\hat\mu)}{\varepsilon}\,|m_\theta(r)-m(r)|\\
    &\leq \theta\operatorname{TV}(\mu',\mu) + \theta\left(\frac{1}{\varepsilon}+\frac{1}{2}|\mathcal{Y} \times \mathcal{Z}|\right)|m'(r)-m(r)|,
\end{align*}
where the last inequality follows from the convexity of $\operatorname{TV}$ and Lemma \ref{lem:TvEps_upper_bound}. Lastly, by Definition \ref{def:fderivative}, 
\begin{align*}
    &\left|F_\varepsilon(\nu,\mu') - F_\varepsilon(\nu,\mu) - \int_{\mathcal{Y} \times \mathcal{Z}}\frac{\delta F_\varepsilon}{\delta\mu}(\nu,\mu,r)(\mu'-\mu)(\mathrm{d}r)\right|\\ &\leq \int_0^1 \theta\int_{\mathcal{Y} \times \mathcal{Z}} \left(\operatorname{TV}(\mu',\mu) + \left(\frac{1}{\varepsilon}+\frac{1}{2}|\mathcal{Y} \times \mathcal{Z}|\right)|m'(r)-m(r)|\right)\left|m'(r)-m(r)\right|\mathrm{d}r\mathrm{d}\theta\\
    &=\frac{1}{2}\operatorname{TV}^2(\mu',\mu) + \frac{1}{2}\left(\frac{1}{\varepsilon}+\frac{1}{2}|\mathcal{Y} \times \mathcal{Z}|\right)\|m'-m\|_{L^2_{\mathcal{Y} \times \mathcal{Z}}}^2\\
    &\leq D_h(\mu',\mu)+ \left(\frac{1}{\varepsilon}+\frac{1}{2}|\mathcal{Y} \times \mathcal{Z}|\right)D_h(\mu',\mu)\\
    &= \left(1+\frac{1}{\varepsilon}+\frac{1}{2}|\mathcal{Y} \times \mathcal{Z}|\right)D_h(\mu',\mu),
\end{align*}
where the last inequality follows since $h$ is the $\chi^2$-divergence in Example \ref{example:chi-squared} and since $\mathcal{Y} \times \mathcal{Z}$ is compact, we may take $\pi$ to be the Lebesgue measure in that example. In particular, we take $L_\mu = 1+\frac{1}{\varepsilon}+\frac{1}{2}|\mathcal{Y} \times \mathcal{Z}|$ in Assumption \ref{def:relative-smoothness}.

    To verify Assumptions \ref{assumption:bddflatF}, note that
    \begin{align*}
        \left|\frac{\delta F_\varepsilon}{\delta \nu}(\nu, \mu, x)\right| &\leq \int_{\mathcal{Y} \times \mathcal{Z}} \left|y - \mathbb E^{X \sim \nu}[\hat{\varphi}(X, z)]\right|\left|\hat{\varphi}(x, z)\right|\mu(\mathrm{d}y, \mathrm{d}z)\\
        &\leq M_\varphi\left(\mu_{\mathcal{Y}} + M_\varphi\right) \coloneqq C_1,
    \end{align*}
    where $$\mu_\mathcal{Y} \coloneqq \int_{\mathcal{Y} \times \mathcal{Z}} |y| \mu(\mathrm{d}y, \mathrm{d}z) < \infty$$ since $\mathcal{Y} \times \mathcal{Z}$ is compact. 
    
    Similarly,
    \begin{align*}
        \left|\frac{\delta F_\varepsilon}{\delta \mu}(\nu, \mu, r)\right| &= \left|\frac{1}{2}\left|y-\mathbb E^{X \sim \nu}[\hat{\varphi}(X, z)]\right|^2 - \frac12\,\psi_\varepsilon\big(m(r)-\hat m(r)\big)\right|\\ 
        &\leq \frac{1}{2}\left(1+|\mathcal{Y}|+M_\varphi\right)^2 \coloneqq C_2,
    \end{align*}
    since $\mathcal{Y}$ is compact and $|\psi_\varepsilon|\leq 1.$
\end{proof}

\section{Numerical experiments}
\label{section:numerical-example}
In this section, we outline how to implement the infinite-dimensional algorithms \ref{eq:mirror-sim-explicit} and \ref{eq:mirror-alt} in the case where $h$ is the relative entropy. For brevity, we present the derivations only for Algorithm \ref{eq:mirror-sim-explicit}, as the arguments for Algorithm \ref{eq:mirror-alt} are entirely analogous. The complete algorithms for both the simultaneous and alternating MDA schemes can be found in Algorithm \ref{eq:implement-mirror-sim-explicit} and Algorithm \ref{eq:implement-mirror-alt} in Section \ref{sec:detailed-numerics}.

\subsection{Scope and interpretation of the GAN experiments}
\label{app:gan-experiment-interpretation}
The GAN experiments in this appendix provide a qualitative comparison of the practical stability and optimization behavior of simultaneous and alternating MDA. They are not designed to estimate or directly verify the theoretical convergence rates $\mathcal{O}(N^{-1/2})$ and $\mathcal{O}(N^{-2/3})$,
respectively.

\paragraph{Theoretical guarantees versus empirical diagnostics.}
The main theorems control the NI error of averaged iterates generated by exact infinite-dimensional mirror steps, with stepsizes chosen as a function of the iteration horizon $N$. In contrast, the experiments report the $W_1$ distance at the current generated distribution using a finite-particle, minibatch-based Langevin approximation with neural-network parametrizations. Thus the theoretical and empirical quantities differ in the error criterion, the use of averaged versus current iterates, and the implementation of the updates. Therefore, the reported $W_1$ trajectories should not be identified with the averaged-iterate NI error appearing in the convergence bounds.

\paragraph{Averaging and last-iterate behavior.}
In particular, the simultaneous convergence theorem does not assert last-iterate convergence of the error. Averaging is essential to the analysis because simultaneous MDA can exhibit oscillations even in elementary bilinear games. Thus an NI error bound for averaged iterates is compatible with oscillatory behavior of the current iterates. Accordingly, the noisy behavior of the current generator in Figure~\ref{fig:wass} does not, by itself, contradict the simultaneous convergence theorem. Nor does that theorem directly control the last-iterate $W_1$ diagnostic of the approximate implementation.

\paragraph{Connection with the error analysis.}
Figure~\ref{fig:wass} primarily illustrates a qualitative difference in practical stability and optimization behavior, and should not be viewed as an asymptotic rate comparison. This distinction is related to the structure of the proof. The simultaneous update uses the opponent's previous strategy and leaves a first-order contribution in the stepsize to the relevant error estimate, whereas the alternating update uses the newly updated first-player strategy in the second-player step and reduces the corresponding contribution to higher order under its additional assumptions. This improved error control provides a theoretical distinction between the schemes, but it does not prove that the observed difference in $W_1$ is entirely attributable to their asymptotic NI convergence rates. Therefore, the experiments should be interpreted as a qualitative comparison of the implemented methods, not as a rate estimation experiment.

\subsection{Simulation of infinite-dimensional MDA}
As shown in Example \ref{example:relative-entropy}, by taking $h$ to be the entropy, the corresponding $h$-Bregman divergence is exactly the KL divergence. Moreover, using the flat derivative formula \eqref{eq:flat-h}, the first-order optimality condition in Proposition \ref{prop:foc} gives
\begin{equation*}
\begin{cases}
    \log \nu^{n+1}(x) - \log \nu^n(x) = -\tau \frac{\delta F}{\delta \nu}(\nu^n, \mu^{n}, x) + C,\\
    \log \mu^{n+1}(y) - \log \mu^n(y) = \tau \frac{\delta F}{\delta \mu}(\nu^n, \mu^{n}, y) + C',
\end{cases}
\end{equation*}
for every $n \geq 0$ and, for all $x,y \in \mathcal{X}$ a.e., where $C,C' \in \mathbb R.$ Exponentiating both sides, we obtain 
\begin{equation*}
\begin{cases}
    \nu^{n+1}(x) \propto \nu^n(x)e^{-\tau \frac{\delta F}{\delta \nu}(\nu^n, \mu^{n}, x)},\\
    \mu^{n+1}(y) \propto \mu^n(y)e^{\tau \frac{\delta F}{\delta \mu}(\nu^n, \mu^{n}, y)},
\end{cases}
\end{equation*}
where the constants $C, C'$ are absorbed into the normalizations. Iterating these multiplicative updates gives
    \begin{equation*}
    \begin{cases}
        \nu^n(x)\propto \nu^0(x)\exp\left(-\tau \sum_{k=0}^{n-1}\frac{\delta F}{\delta \nu}(\nu^k,\mu^k,x)\right),\\
        \mu^n(y)\propto \mu^0(y)\exp\left(\tau \sum_{k=0}^{n-1}\frac{\delta F}{\delta \mu}(\nu^k,\mu^k,y)\right).
    \end{cases}
    \end{equation*}
    Hence,
    \begin{equation*}
    \begin{cases}
        \nabla \log \nu^n(x) = \nabla \log \nu^0(x) - \tau \sum_{k=0}^{n-1}\nabla \frac{\delta F}{\delta \nu}(\nu^k,\mu^k,x),\\
        \nabla \log \mu^n(y) = \nabla \log \mu^0(y) - \tau \sum_{k=0}^{n-1}\nabla \frac{\delta F}{\delta \mu}(\nu^k,\mu^k,y),
    \end{cases}
    \end{equation*}
    since the normalization constant disappears after taking the gradient. Thus, in the Langevin implementation one does not need to compute $\nabla \log \nu^1$ and $\nabla \log \mu^1$ explicitly. One only needs $\nabla \log \nu^0$ and $\nabla \log \mu^0$, which is easy to evaluate for standard initializations such as Gaussian or uniform measures, together with the gradients of the flat derivatives along the iterates, which are computed using empirical averages on minibatches.

    Suppose the initial samples ${(X_j, Y_j)}_{j=1}^J$ are drawn from $(\nu^0, \mu^0)$. We set $(X_{j,0}, Y_{j,0})_{j=1}^J = (X_j, Y_j)_{j=1}^J$ and sample from $(\nu^1, \mu^1)$ via Langevin dynamics:
\begin{equation*}
    X_{j,t+1} = X_{j,t} + \gamma \left(\nabla\log \nu^0(X_{j,t}) - \tau\nabla \frac{\delta F}{\delta \nu}(\nu^0, \mu^{0}, X_{j,t})\right) + \sqrt{\frac{2\gamma}{\tau}}\mathcal{N}_{j,t},
\end{equation*}
\begin{equation*}
    Y_{j,t+1} = Y_{j,t} + \gamma \left( \nabla\log \mu^0(Y_{j,t}) + \tau\nabla \frac{\delta F}{\delta \mu}(\nu^0, \mu^{0}, Y_{j,t})\right) + \sqrt{\frac{2\gamma}{\tau}}\mathcal{N}_{j,t},
\end{equation*}
for $1 \leq j \leq J$ and $0 \leq t \leq T-1$, where $\gamma>0$ is the step size and $\mathcal{N}_{j,t}$ are i.i.d standard Gaussian variables. For sufficiently large $J$ and $T$, the terminal particles $(X_{j,T}, Y_{j,T})_{j=1}^J$ approximate samples from $(\nu^1, \mu^1)$. Repeating this procedure recursively then yields samples from $(\nu^2, \mu^2), \dots, (\nu^n, \mu^n).$ For simplicity, we may take $\nu_0, \mu_0$ be uniform distributions, so that $\nabla \log \nu^0 = \nabla \log \mu^0 = 0$. 

\subsection{Training GANs by MDA}
We train the mean-field GAN from Example \ref{example:GAN-example} using simultaneous and alternating MDA-GAN (Algorithms \ref{alg:sim_MD} and \ref{alg:alt_MD}) on the 8-Gaussian mixture and Swiss Roll datasets \cite{gulrajani2017improvedtrainingwassersteingans}. Full algorithmic details, including hyperparameters and network architectures, are in Section \ref{sec:detailed-numerics}. Both methods are run for $2000$ iterations, with performance assessed by visualizing generated samples at $400$, $1000$, and $2000$ iterations.
\begin{figure}[htbp]
  \centering
  \begin{subfigure}{0.245\textwidth}
    \includegraphics[width=\linewidth]{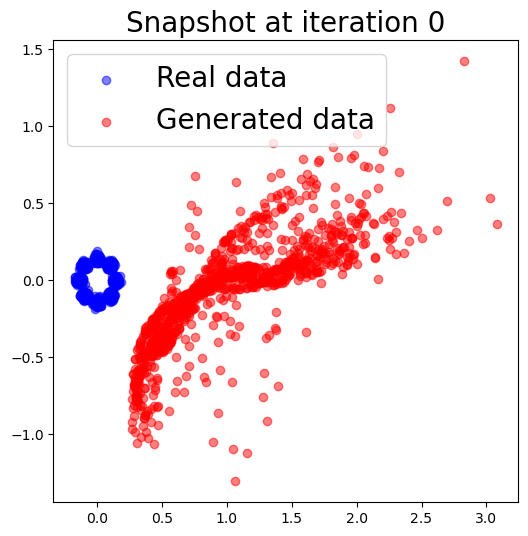}
    \caption{Iteration $n=0$}\label{fig:gauss_sim0}
  \end{subfigure}
  \begin{subfigure}{0.245\textwidth}
    \includegraphics[width=\linewidth]{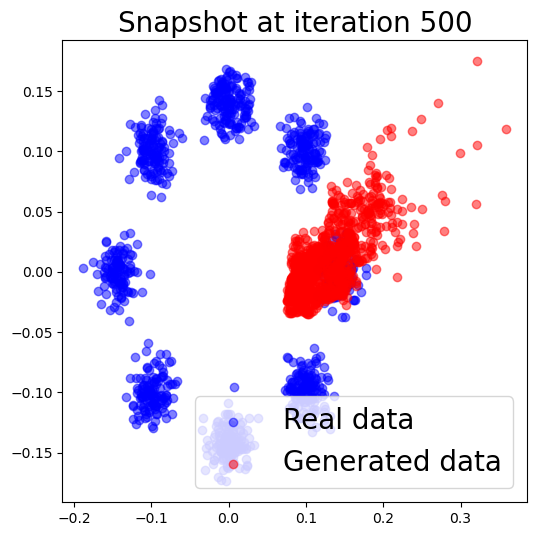}
    \caption{Iteration $n=500$}\label{fig:gauss_sim500}
  \end{subfigure}
  \begin{subfigure}{0.245\textwidth}
    \includegraphics[width=\linewidth]{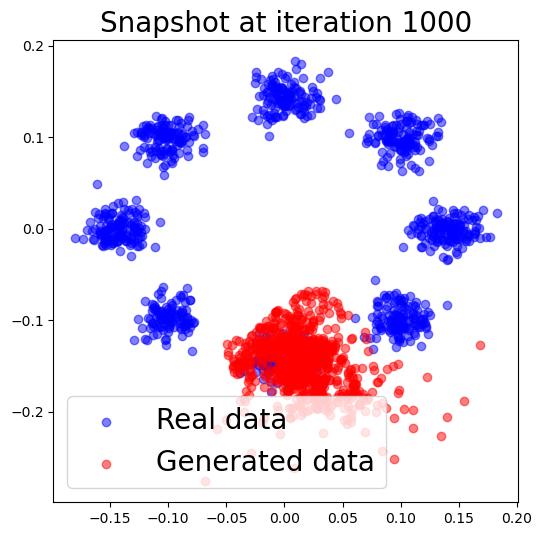}
    \caption{Iteration $n=1000$}\label{fig:gauss_sim1000}
  \end{subfigure}
  \begin{subfigure}{0.245\textwidth}
    \includegraphics[width=\linewidth]{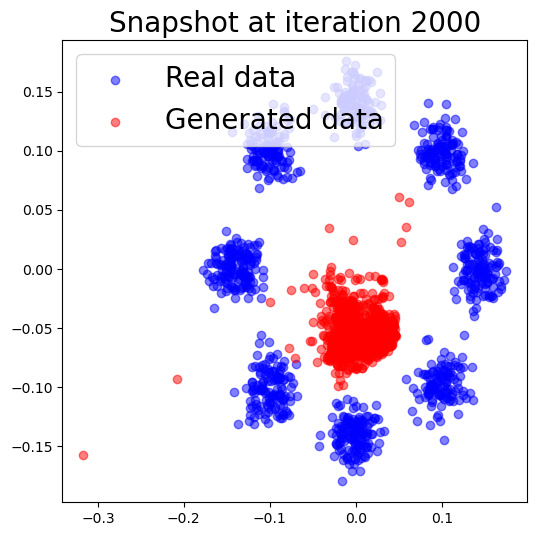}
    \caption{Iteration $n=2000$}\label{fig:gauss_sim2000}
  \end{subfigure}
  \caption{Simultaneous MDA-GAN (Algorithm \ref{alg:sim_MD}) learning an 8-Gaussian mixture}
  \label{fig:gauss-simultaneous}
\end{figure}

\begin{figure}[htbp]
  \centering
  \begin{subfigure}{0.245\textwidth}
    \includegraphics[width=\linewidth]{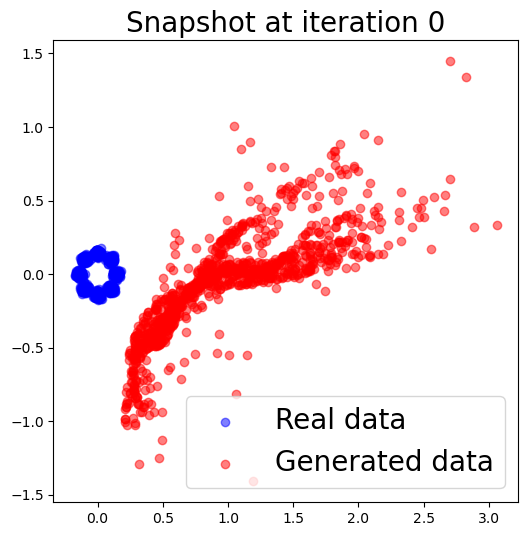}
    \caption{Iteration $n=0$}\label{fig:gauss_seq0}
  \end{subfigure}
  \begin{subfigure}{0.245\textwidth}
    \includegraphics[width=\linewidth]{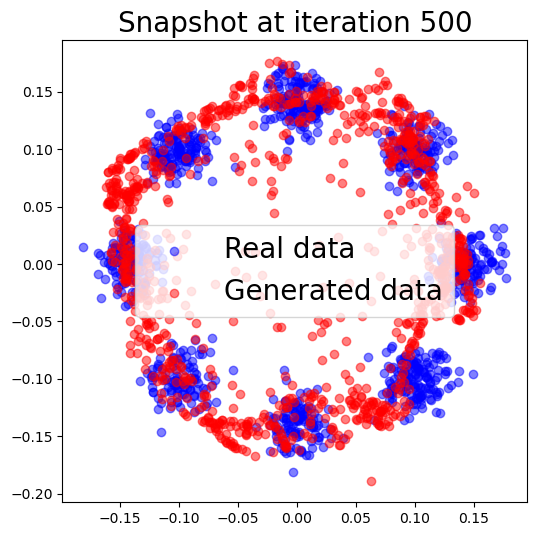}
    \caption{Iteration $n=500$}\label{fig:gauss_seq500}
  \end{subfigure}
  \begin{subfigure}{0.245\textwidth}
    \includegraphics[width=\linewidth]{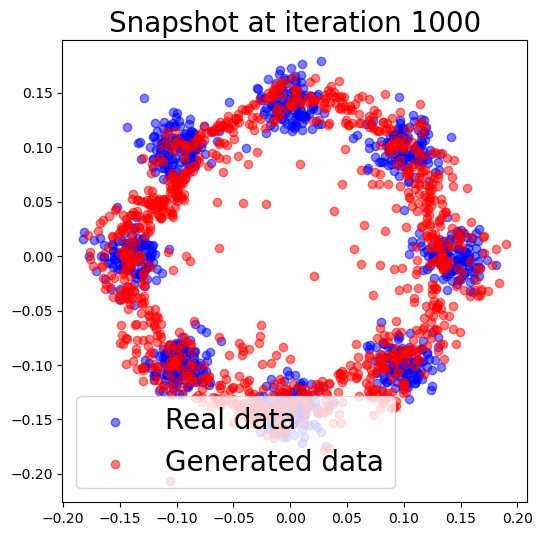}
    \caption{Iteration $n=1000$}\label{fig:gauss_seq1000}
  \end{subfigure}
  \begin{subfigure}{0.245\textwidth}
    \includegraphics[width=\linewidth]{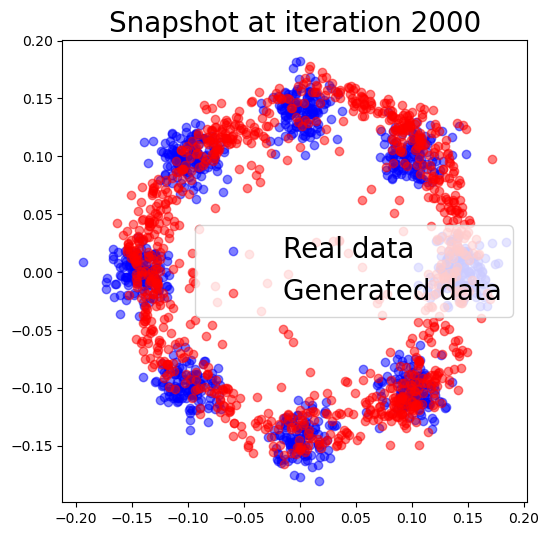}
    \caption{Iteration $n=2000$}\label{fig:gauss_seq2000}
  \end{subfigure}
  
  \caption{Alternating MDA-GAN (Algorithm \ref{alg:alt_MD}) learning an 8-Gaussian mixture}
  \label{fig:gauss-alternating}
\end{figure}
\begin{figure}[htbp]
	\centering
	\begin{subfigure}{0.245\textwidth}
		\includegraphics[width=\linewidth]{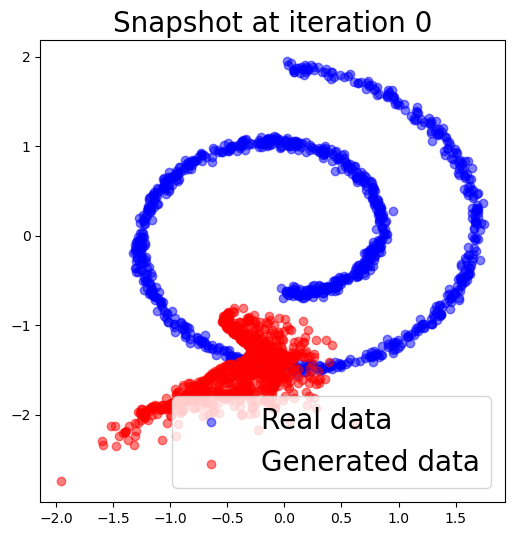}
		\caption{Iteration $n=0$}\label{fig:roll_sim0}
	\end{subfigure}
	\begin{subfigure}{0.245\textwidth}
		\includegraphics[width=\linewidth]{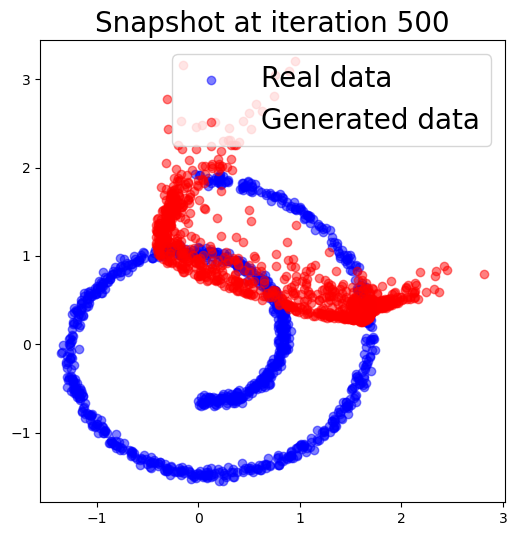}
		\caption{Iteration $n=500$}\label{fig:roll_sim500}
	\end{subfigure}
	\begin{subfigure}{0.245\textwidth}
		\includegraphics[width=\linewidth]{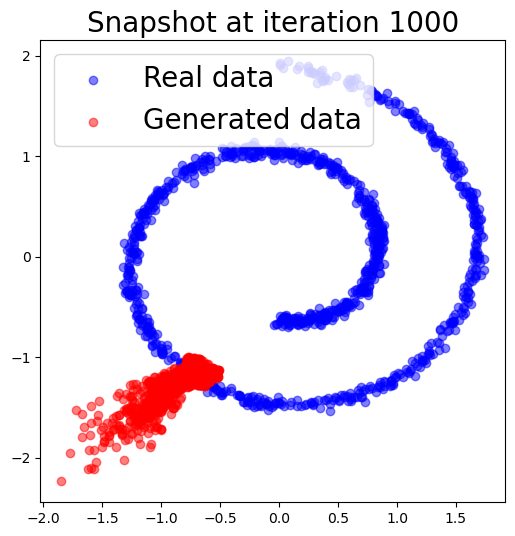}
		\caption{Iteration $n=1000$}\label{fig:roll_sim1000}
	\end{subfigure}
	\begin{subfigure}{0.245\textwidth}
		\includegraphics[width=\linewidth]{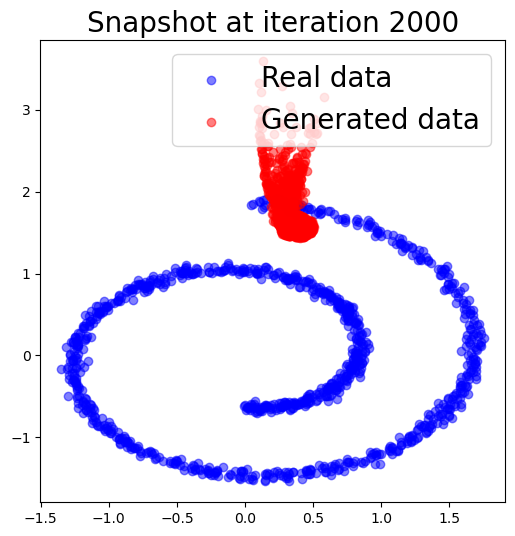}
		\caption{Iteration $n=2000$}\label{fig:roll_sim2000}
	\end{subfigure}
	\caption{Simultaneous MDA-GAN (Algorithm \ref{alg:sim_MD}) learning the Swiss Roll}
	\label{fig:roll-simultaneous}
\end{figure}
\begin{figure}[htbp]
	\centering
	\begin{subfigure}{0.245\textwidth}
		\includegraphics[width=\linewidth]{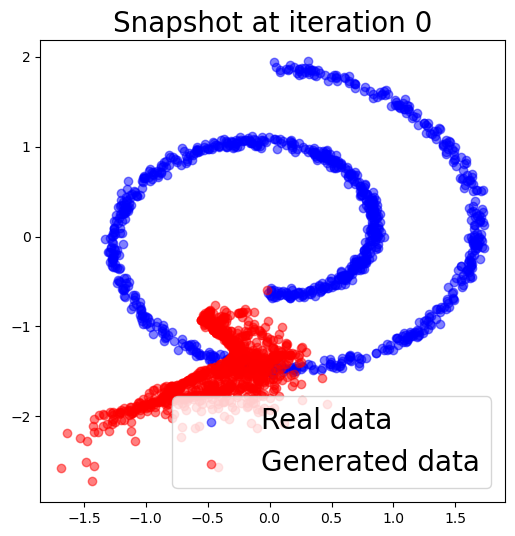}
		\caption{Iteration $n=0$}\label{fig:roll_seq0}
	\end{subfigure}
	\begin{subfigure}{0.245\textwidth}
		\includegraphics[width=\linewidth]{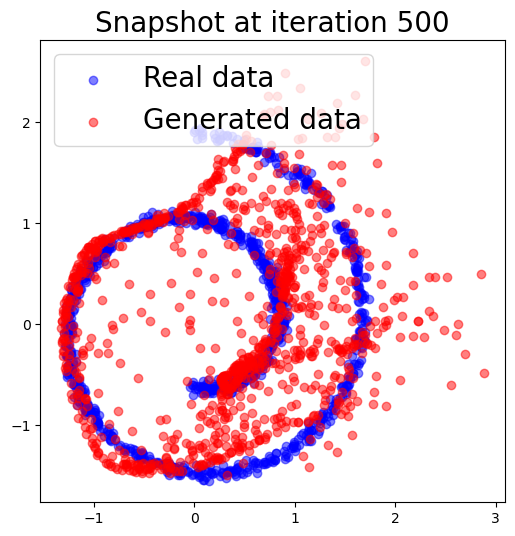}
		\caption{Iteration $n=500$}\label{fig:roll_seq500}
	\end{subfigure}
	\begin{subfigure}{0.245\textwidth}
		\includegraphics[width=\linewidth]{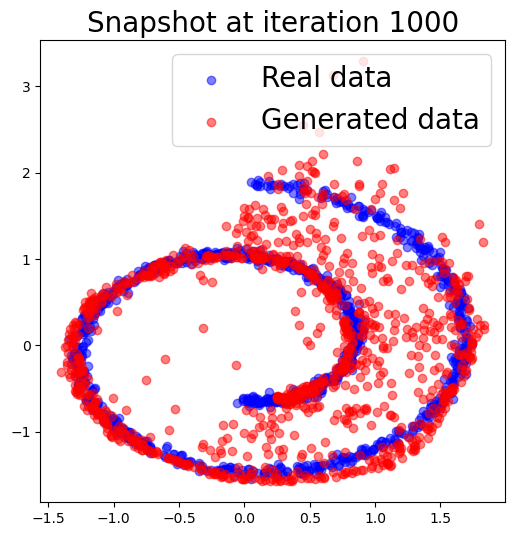}
		\caption{Iteration $n=1000$}\label{fig:roll_seq1000}
	\end{subfigure}
	\begin{subfigure}{0.245\textwidth}
		\includegraphics[width=\linewidth]{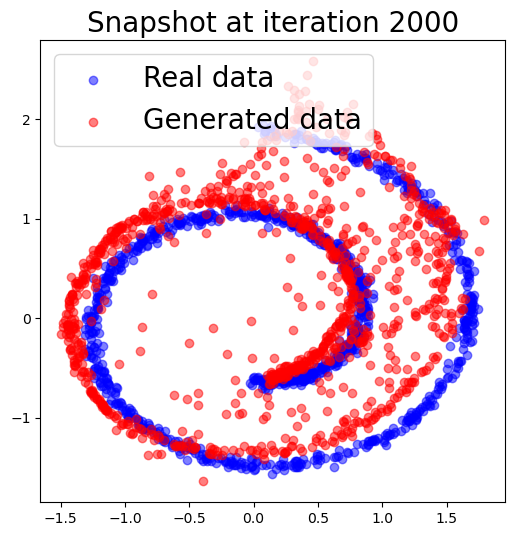}
		\caption{Iteration $n=2000$}\label{fig:roll_seq2000}
	\end{subfigure}
	\caption{Alternating MDA-GAN (Algorithm \ref{alg:alt_MD}) learning the Swiss Roll}
	\label{fig:roll-alternating}
\end{figure}
\begin{figure}[htbp]
	\centering
	\begin{subfigure}{0.495\textwidth}
		\includegraphics[width=\linewidth]{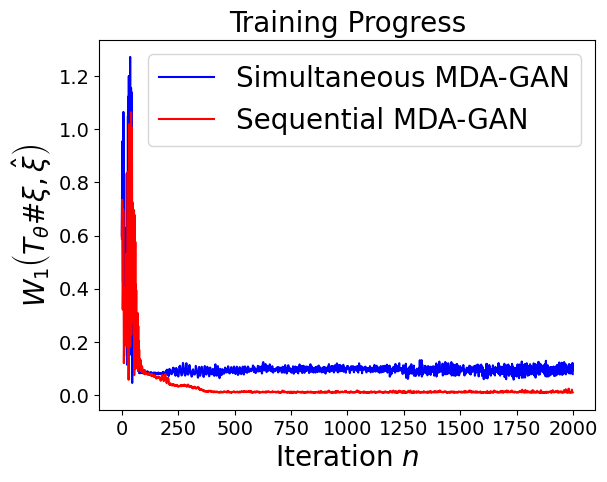}
		\caption{8-Gaussian mixture}\label{fig:gauss_wass}
	\end{subfigure}
	\begin{subfigure}{0.495\textwidth}
		\includegraphics[width=\linewidth]{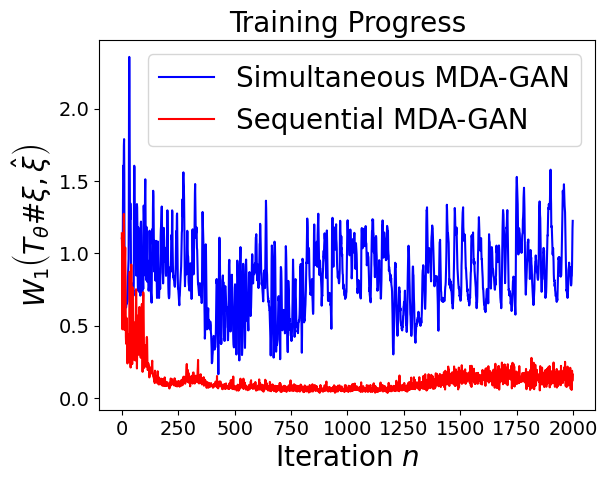}
		\caption{Swiss Roll}\label{fig:roll_wass}
	\end{subfigure}
	\caption{$L^1$-Wasserstein distance between generated and real data for the 8-Gaussian mixture and Swiss Roll}
	\label{fig:wass}
\end{figure}
Figures \ref{fig:gauss-simultaneous} and \ref{fig:gauss-alternating} show the training dynamics of simultaneous and alternating MDA-GANs on the 8-Gaussian mixture, with analogous results on the Swiss Roll in Figures \ref{fig:roll-simultaneous} and \ref{fig:roll-alternating}. In both settings, generated samples start far from the data but the alternating variant captures the multi-modal structure and the spiral geometry of the Swiss Roll more clearly and at earlier iterations. In Figure \ref{fig:wass}, we plot the $L^1$-Wasserstein distance $W_1\big(T_{\theta^n}\#\xi, \hat{\xi}\big)$ for both tasks over iterations $n$, illustrating the more favorable empirical behavior of alternating MDA-GAN.
\subsection{Details on numerical experiments}
\label{sec:detailed-numerics}
In this section, we present the additional details of the numerical experiments. We begin by summarizing the implementable versions of the simultaneous and alternating MDA algorithms introduced in Section \ref{section:numerical-example}.
\begin{algorithm}[!ht] 
	\caption{\textsc{Implementable simultaneous MDA}}
	\label{eq:implement-mirror-sim-explicit}
	\KwIn{objective function $F,$ initial measures $(\nu^0, \mu^0)$, stepsize $\tau, \gamma > 0,$ time horizons $K, N$ and number of particles $J$}
	Generate i.i.d $\left(X_j^0, Y_j^0\right)_{j=1}^J \sim (\nu^0, \mu^0)$\\ 
	Set $\left(X_{j,0}^0, Y_{j,0}^0\right)_{j=1}^J = \left(X_j^0, Y_j^0\right)_{j=1}^J$\\
	\For{$n =0, 1, \dots ,N-1$}
	{
		\For{$k =0, 1, \dots ,K-1$}
		{
			Generate independent Gaussian random variables $\mathcal{N}_{j,k}^n$\\
			\For{$j =1, 2, \dots ,J$}
			{
				$X_{j,k+1}^n = X_{j,k}^n + \gamma \left(\nabla \log \nu^n(X_{j,k}^n) - \tau\nabla \frac{\delta F}{\delta \nu}(\nu^n, \mu^{n}, X_{j,k}^n)\right) + \sqrt{\frac{2\gamma}{\tau}}\mathcal{N}_{j,k}^n$\\
                
				$Y_{j,k+1}^n = Y_{j,k}^n + \gamma \left(\nabla \log \mu^n(Y_{j,k}^n) + \tau\nabla \frac{\delta F}{\delta \mu}(\nu^n, \mu^{n}, Y_{j,k}^n)\right) + \sqrt{\frac{2\gamma}{\tau}}\mathcal{N}_{j,k}^n$
			}
		}
		\For{$j =1, 2, \dots ,J$}{
			$X_{j,0}^{n+1} = X_{j,K}^n, \quad $
			$Y_{j,0}^{n+1} = Y_{j,K}^n$
		}
		$\nu^n = \frac{1}{J}\sum_{j=1}^J \delta_{X_{j,0}^n}, \quad  \mu^n = \frac{1}{J}\sum_{j=1}^J \delta_{Y_{j,0}^n}$
	}
    Generate $\hat n \sim \mathrm{Unif}\{0,\dots,N-1\}$\\
    \KwOut{$(\nu^{\hat n},\mu^{\hat n})$}
\end{algorithm}
\begin{algorithm}[!ht] 
	\caption{\textsc{Implementable alternating MDA}}
	\label{eq:implement-mirror-alt}
	\KwIn{objective function $F,$ initial measures $(\nu^0, \mu^0)$, stepsize $\tau, \gamma > 0,$ time horizons $K, N$ and number of particles $J$}
	Generate i.i.d $\left(X_j^0, Y_j^0\right)_{j=1}^J \sim (\nu^0, \mu^0)$\\ 
	Set $\left(X_{j,0}^0, Y_{j,0}^0\right)_{j=1}^J = \left(X_j^0, Y_j^0\right)_{j=1}^J$\\
	\For{$n =0, 1, \dots ,N-1$}
	{
		\For{$k =0, 1, \dots ,K-1$}
		{
			Generate independent Gaussian random variables $\mathcal{N}_{j,t}^n$\\
			\For{$j =1, 2, \dots ,J$}
			{
				$X_{j,k+1}^n = X_{j,k}^n +\left(\nabla \log \nu^n(X_{j,k}^n)- \gamma \nabla \frac{\delta F}{\delta \nu}(\nu^n, \mu^{n}, X_{j,k}^n)\right) + \sqrt{\frac{2\gamma}{\tau}}\mathcal{N}_{j,k}^n$
			}
		}
		\For{$j =1, 2, \dots ,J$}{
			$X_{j,0}^{n+1} = X_{j,K}^n$
		}
		$\nu^{n+1} = \frac{1}{J}\sum_{j=1}^J \delta_{X_{j,0}^{n+1}}$\\
		\For{$k =0, 1, \dots ,K-1$}
		{
			Generate independent Gaussian random variables $\mathcal{N}_{j,k}^n$\\
			\For{$j =1, 2, \dots ,J$}
			{
				$Y_{j,k+1}^n = Y_{j,k}^n + \gamma \left(\nabla \log \mu^n(Y_{j,k}^n) + \tau\nabla \frac{\delta F}{\delta \mu}(\nu^{n+1}, \mu^{n}, Y_{j,k}^n)\right) + \sqrt{\frac{2\gamma}{\tau}}\mathcal{N}_{j,k}^n$
			}
		}
		\For{$j =1, 2, \dots ,J$}{
			$Y_{j,0}^{n+1} = Y_{j,K}^n$
		}
		$\mu^n = \frac{1}{J}\sum_{j=1}^J \delta_{Y_{j,0}^n}$
	}
    Generate $\hat n \sim \mathrm{Unif}\{0,\dots,N-1\}$\\
    \KwOut{$(\nu^{\hat n},\mu^{\hat n})$}
\end{algorithm}
We now turn to Algorithms \ref{eq:implement-mirror-sim-explicit} and \ref{eq:implement-mirror-alt} in the setting where $F$ corresponds to the GAN objective introduced in Example \ref{example:GAN-example}. Recall that $F$ takes the form
\begin{align*}
	F(\nu, \mu) &= \int_{\mathcal{W}} \int_{\Theta} \int_\mathcal{Y} D_{w}(y) \left(T_{\theta} \# \xi - \hat{\xi}\right)(\mathrm{d}y) \nu(\mathrm{d}\theta) \mu(\mathrm{d}w)\\
	&= \int_{\mathcal{W}} \int_{\Theta} \int_\mathcal{Y} D_{w}(y) \left(T_{\theta} \# \xi\right)(\mathrm{d}y) \nu(\mathrm{d}\theta) \mu(\mathrm{d}w) -\int_{\mathcal{W}} \int_\mathcal{Y} D_{w}(y) \hat{\xi}(\mathrm{d}y) \mu(\mathrm{d}w)
\end{align*} 
By Definition \ref{def:fderivative}, we have
\begin{equation*}
	\frac{\delta F}{\delta \nu}(\nu,\mu,\theta) = \int_{\mathcal{W}} \int_\mathcal{Y} D_{w}(y) \left(T_{\theta} \# \xi\right)(\mathrm{d}y) \mu(\mathrm{d}w), 
\end{equation*}
\begin{equation*}
	\frac{\delta F}{\delta \mu}(\nu,\mu,w) = \int_{\Theta} \int_\mathcal{Y} D_{w}(y) \left(T_{\theta} \# \xi\right)(\mathrm{d}y) \nu(\mathrm{d}\theta) - \int_\mathcal{Y} D_{w}(y) \hat{\xi}(\mathrm{d}y).
\end{equation*}
The flat derivatives can be approximated using empirical averages. For a batch of real data $\{\xi_1^\textup{real}, \dots, \xi_M^\textup{real}\} \sim \hat{\xi},$ we have 
\begin{equation*}
	\int_\mathcal{Y} D_{w}(y) \hat{\xi}(\mathrm{d}y) \approx \frac{1}{M}\sum_{i=1}^M D_w(\xi^{\textup{real}}_i).
\end{equation*} 
For the term in $\frac{\delta F}{\delta \mu}(\nu,\mu,w)$ that involves integration with respect to both $\nu$ and the generated data $T_{\theta}\#\xi$, we approximate via sampling as follows. We sample 
\begin{equation*}
	\{\theta^{(1)}, \theta^{(2)}, ... ,\theta^{(J)}\} \sim \nu, \quad  \left\{ Z_i^{(j)} \right\}_{i=1}^{M} \sim  T_{\theta^{(j)}} \# \xi,
\end{equation*}
leading to the estimator
\begin{align*}
	\int_{\Theta} \int_\mathcal{Y} D_{w}(y) \left(T_{\theta} \# \xi\right)(\mathrm{d}y) \nu(\mathrm{d}\theta) \approx \frac{1}{JM} \sum_{i=1}^M \sum_{j=1}^{J} D_{w}\left( X_i^{(j)}\right).
\end{align*}
Analogously, for $\frac{\delta F}{\delta \nu}(\nu,\mu,\theta)$ we sample
\begin{equation*}
	\{w^{(1)}, w^{(2)}, ... ,w^{(J)}\} \sim \mu, \quad \{Z_i\}_{i=1}^M\sim T_{\theta} \# \xi,
\end{equation*}
and approximate
\begin{align*}
	\int_{\mathcal{W}} \int_\mathcal{Y} D_{w}(y) \left(T_{\theta} \# \xi\right)(\mathrm{d}y) \mu(\mathrm{d}w) \approx \frac{1}{JM} \sum_{i=1}^M \sum_{j=1}^{J} D_{w^{(j)}}\left(Z_i  \right).  
\end{align*}
To mitigate the computational cost of Algorithms \ref{eq:implement-mirror-sim-explicit} and \ref{eq:implement-mirror-alt}, we follow the approach of \cite{pmlr-v97-hsieh19b} and employ Langevin dynamics with exponential damping (see also their Algorithm 3). Algorithm \ref{eq:implement-mirror-sim-explicit} is essentially a realization of the Mirror-GAN algorithm of \cite{pmlr-v97-hsieh19b}, with different architectures and hyperparameter choices. The alternating variant, Algorithm \ref{eq:implement-mirror-alt}, which is not studied in \cite{pmlr-v97-hsieh19b}, is obtained by adapting that simultaneous algorithm to the alternating update setting. Below, we present the simultaneous and alternating algorithms used in our experiments.
\begin{algorithm}[!ht]
	\caption{\textsc{Simultaneous MDA-GAN}}
	\label{alg:sim_MD}
	\KwIn{Initial parameters $w^{0}, \theta^{0}$, step sizes $\{\gamma^n\}_{n=0}^{N-1},$ $\{\tau^n\}_{n=0}^{N-1}$, time horizon $\{K^n\}_{n=0}^{N-1}$, averaging parameter $\beta \in [0,1]$, source probability measure $\xi$} 
	\For{$n = 0, 1, \dots, N-1$}{
		Set $\bar{w}^{n}, w^n_{0} = w^n$ and $\bar{\theta}^{n}, \theta^n_{0} = \theta^n;$\\
		\For{$k = 0, 1, \dots, K_n-1$}{
			$A=\{Z_1, \dots, Z_M\} \sim T_{\theta^n_{k}} \# \xi$\;
			\begin{equation*}
				\theta^n_{k+1} = \theta^n_{k} - \frac{\gamma^n}{M}\nabla_\theta \sum_{Z_i \in A} D_{w^n}(Z_i) + \sqrt{\frac{2\gamma^n}{\tau^n}}\mathcal{N}^n_{k};
			\end{equation*}
			$B=\{\xi_1^\textup{real}, \dots, \xi_M^\textup{real}\} \sim \hat{\xi}$\;
			$B'=\{Z'_1, \dots, Z'_M\} \sim T_{\theta^n} \# \xi$\;
			\begin{equation*}
				w^n_{k+1} = w^n_{k} + \frac{\gamma^n}{M}\nabla_w\sum_{Z'_i \in B'} D_{w^n_{k}}(Z'_i) - \frac{\gamma^n}{M}\nabla_w\sum_{\xi^\textup{real}_i \in B} D_{w^n_{k}}(\xi^\textup{real}_i) + \sqrt{\frac{2\gamma^n}{\tau^n}}\mathcal{N}^n_{k};
			\end{equation*}
			$\bar{w}^{n} = (1-\beta)\bar{w}^{n} + \beta w^n_{k+1}, \quad
			\bar{\theta}^{n} = (1-\beta)\bar{\theta}^{n} + \beta \theta^n_{k+1};$
		}
		$w^{n+1} = (1-\beta)w^n + \beta \bar{w}^{n}, \quad
		\theta^{n+1} = (1-\beta)\theta^n + \beta \bar{\theta}^{n};$
	}
	\KwOut{$w^N, \theta^N$}
\end{algorithm}
\begin{algorithm}[!ht]
	\caption{\textsc{alternating MDA-GAN}}
	\label{alg:alt_MD}
	\KwIn{Initial parameters $w^{0}, \theta^{0}$, step sizes $\{\gamma^n\}_{n=0}^{N-1},$ $\{\tau^n\}_{n=0}^{N-1}$, time horizon $\{K^n\}_{n=0}^{N-1}$, averaging parameter $\beta \in [0,1]$, source probability measure $\xi$} 
	\For{$n = 0, 1, \dots, N-1$}{
		Set $\bar{w}^{n}, w^n_{0} = w^n$ and $\bar{\theta}^{n}, \theta^n_{0} = \theta^n;$\\
		\For{$k = 0, 1, \dots, K_n-1$}{
			$A=\{Z_1, \dots, Z_M\} \sim T_{\theta^n_{k}} \# \xi$\;
			\begin{equation*}
				\theta^n_{k+1} = \theta^n_{k} - \frac{\gamma^n}{M}\nabla_\theta \sum_{Z_i \in A} D_{w^n}(Z_i) + \sqrt{\frac{2\gamma^n}{\tau^n}}\mathcal{N}^n_{k};
			\end{equation*}
			$\bar{\theta}^{n} = (1-\beta)\bar{\theta}^{n} + \beta \theta^n_{k+1};$
		}
		$\theta^{n+1} = (1-\beta)\theta^n + \beta \bar{\theta}^{n};$\\
		\For{$k = 0, 1, \dots, K_n-1$}{
			$B=\{\xi_1^\textup{real}, \dots, \xi_M^\textup{real}\} \sim \hat{\xi}$\;
			$B'=\{Z'_1, \dots, Z'_M\} \sim T_{\theta^{n+1}} \# \xi$\;
			\begin{equation*}
				w^n_{k+1} = w^n_{k} + \frac{\gamma^n}{M}\nabla_w\sum_{Z'_i \in B'} D_{w^n_{k}}(Z'_i) - \frac{\gamma^n}{M}\nabla_w\sum_{\xi^\textup{real}_i \in B} D_{w^n_{k}}(\xi^\textup{real}_i) + \sqrt{\frac{2\gamma^n}{\tau^n}}\mathcal{N}^n_{k};
			\end{equation*}
			$\bar{w}^{n} = (1-\beta)\bar{w}^{n} + \beta w^n_{k+1};$
		}
		$w^{n+1} = (1-\beta)w^n + \beta \bar{w}^{n};$
	}
	\KwOut{$w^N, \theta^N$}
\end{algorithm}

In all experiments, we closely follow the specifications from \cite{pmlr-v97-hsieh19b}. We adopt the gradient-penalized discriminator of \cite{gulrajani2017improvedtrainingwassersteingans} as a soft-constraint alternative to the original Wasserstein GAN formulation to increase stability. The gradient penalty parameter is set to $\lambda = 0.1.$ For our Simultaneous and alternating MDA-GANs, we fix the damping factor to $\beta = 0.8$. The scheduling of the parameters $K^n, \gamma^n,$ and $\tau^n$ is $K^n = \floor{(1+10^{-5})^n},$ $\gamma^n = \gamma(1-10^{-5})^n,$ with $\gamma = 0.01,$ and $\tau^n = \tau(1- 5\times 10^{-5})^{-n},$ with $\tau = 100.$ The number of samples per batch is $M=1024.$ For both the 8-Gaussian mixture and Swiss Roll datasets, we use fully connected networks for the generator and discriminator, each consisting of two-hidden-layers with $J=512$ neurons on each layer. The generator and discriminator networks use ReLU activations, except for the output layer of the discriminator, which employs a tanh activation. All network parameters are initialized from a normal distribution $\mathcal{N}(0,0.01).$

\begin{remark}[Implementation of the measure space mirror steps]
Algorithms~\ref{eq:mirror-sim-explicit} and~\ref{eq:mirror-alt} are formulated as idealized infinite-dimensional mirror algorithms, and our convergence results concern their exact measure space updates. For a general mirror potential $h$, the corresponding proximal step is itself an optimization problem over probability measures and may require an expensive inner solver. Moreover, even when the proximal problem has a tractable variational characterization, as we saw in the case of relative entropy, the flat derivatives entering the update generally involve expectations or integrals that must be approximated from samples.

Thus, relative entropy gives a closed-form density update at the variational level, but its numerical realization can still be nontrivial since the normalizing constant may be unavailable in closed form, the flat derivatives may require Monte Carlo estimation, and sampling from the exponentially tilted measures may require an additional sampling procedure. Therefore, in line with particle algorithms that are standard in the mean-field literature, we implement an inexact particle approximation using finitely many particles, minibatch estimates, and Langevin dynamics. The experiments should accordingly be interpreted as an illustrative approximate realization of the ideal measure space algorithms, rather than as claiming generic efficient implementability. For other mirror geometries, such as the $\chi^2$ divergence, the proximal update may additionally require an infinite-dimensional projection and may activate feasibility constraints. A joint analysis of finite-particle approximation, stochastic estimation, and Langevin sampling errors is beyond the scope of the present convergence theory.
\end{remark}

\section{Differentiability on the primal space}
\label{appendix:B}
\begin{remark}
\label{rmk:FTC}
One can show that if $F:\mathcal{P}(\mathbb R^d) \to \mathbb R$ admits a flat derivative $\frac{\delta F}{\delta \mu},$ then for all $\mu,\mu'\in \mathcal{P}(\mathbb R^d),$ the function $[0,1]\ni \varepsilon \mapsto F(\mu^\varepsilon) $ is
continuous on $[0,1]$ and differentiable on $(0,1)$ with derivative 
$\frac{\mathrm{d}}{\mathrm{d} \varepsilon}F(\mu^\varepsilon) = \int_{\mathbb R^d} \frac{\delta F}{\delta \mu} (\mu^\varepsilon, x) (\mu'-\mu) (\mathrm{d}x)$ (see \cite[Theorem 2.3]{Jourdain2020CentralLT}).
Hence, by the fundamental theorem of calculus, $F(\mu')-F(\mu)=\int_0^1 \int_{\mathbb R^d} \frac{\delta F}{\delta \mu} (\mu^\varepsilon, x) (\mu'-\mu) (\mathrm{d}x)\mathrm{d}\varepsilon,$
provided that $\varepsilon \mapsto \int  \frac{\delta F}{\delta \mu} (\mu^\varepsilon, x) (\mu'-\mu)(\mathrm{d}x)$ is integrable.
\end{remark}

\section{Differentiability on the dual space}
\label{appendix:differentiability-dual-space}
In this section, we start by recalling the notions of Fréchet and Gâteaux derivative for functions $H:C_b\left(\mathcal{X}\right) \to \mathbb R,$ where $\left(C_b\left(\mathcal{X}\right), \|\cdot\|_{\infty}\right)$ is the Banach space of real-valued bounded continuous functions on $\mathcal{X} \subseteq \mathbb R^d$; see e.g. Chapters $7,1,3$ in \cite{aliprantis2007infinite,ambrosetti1995primer,ortega1970iterative}, respectively. Based on these notions of differentiablity, we will introduce the notion of first variation for functions $H.$
\subsection{Preliminaries on Fréchet and Gâteaux derivatives}
For $\mathcal{X} \subseteq \mathbb R^d,$ let $\mathcal{L}\left(C_b\left(\mathcal{X}\right),\mathbb R\right)$ and $\mathcal{L}\left(C_b\left(\mathcal{X}\right)\right)$ denote the space of continuous linear maps from $C_b\left(\mathcal{X}\right)$ to $\mathbb R,$ and from $C_b\left(\mathcal{X}\right)$ to itself, respectively.
\begin{definition}[Fréchet differentiability]
\label{definition:frechet-diff}
Let $\mathcal{U} \subseteq C_b\left(\mathcal{X}\right)$ be open. Given $f \in \mathcal{U},$ the function $H:\mathcal{U} \to \mathbb R$ is \textit{Fréchet differentiable} at $f$ if there exists $T \in \mathcal{L}(C_b(\mathcal{X}),\mathbb R)$ such that, for all $g \in C_b(\mathcal{X}),$
\begin{equation*}
    \lim_{\left\|g\right\|_{\infty} \to 0} \frac{\left|H\left(f+ g\right) - H(f) - T\left[g\right]\right|}{\left\|g\right\|_{\infty}} = 0.
\end{equation*}
If it exists, the map $T$ is unique, we write $T = \nabla_{\mathcal{F}} H(f),$ and call $\nabla_{\mathcal{F}} H(f)$ the \textit{Fréchet derivative} of $H$ at $f.$ If $H$ is Fréchet differentiable at every $f \in \mathcal{U},$ then we say that $H$ is Fréchet differentiable on $\mathcal{U}.$ 
\end{definition}
\begin{example}[Convex conjugate of the relative entropy]
\label{example:entropy1}
    If $h$ is the relative entropy in Example \ref{example:relative-entropy}, then a straightforward calculation directly from Definition \ref{def: convex-conjugate} shows that its dual $h^*$ is given by
\begin{equation*}
    h^*(f) = \log \left(\int_{\mathcal{X}} e^{f(z)}\pi(\mathrm{d}z)\right).
\end{equation*} 
\end{example}
\begin{example}[Fréchet derivative of the relative entropy]
\label{example:first-Frechet-entropy}
A straightforward calculation directly from Definition \ref{definition:frechet-diff} shows that $h^*$ is Fréchet differentiable on $C_b(\mathcal{X})$ with Fréchet derivative given by 
\begin{equation*}
    \nabla_{\mathcal{F}} h^*(f)[g] = \int_{\mathcal{X}}g(z)\frac{e^{f(z)}}{\int_\mathcal{X}e^{f(y)}\pi(\mathrm{d}y)}\pi(\mathrm{d}z),
\end{equation*}
for all $g \in C_b(\mathcal{X}).$
\end{example}
\begin{example}[Convex conjugate of the $\chi^2$-divergence]
\label{example:chi1}
    If $h$ is the $\chi^2$-divergence in Example \ref{example:chi-squared}, then \cite[Example 7.4]{Polyanskiy_Wu_2025} shows that its dual $h^*$ is given by
\begin{equation*}
    h^*(f) = \frac{1}{2}\int_{\mathcal{X}} f(z) \pi(\mathrm{d}z) + \frac{1}{8}\int_{\mathcal{X}} f^2(z) \pi(\mathrm{d}z).
\end{equation*} 
\end{example}
\begin{example}[Fréchet derivative of the relative $\chi^2$-divergence]
\label{example:first-Frechet-chi}
A straightforward calculation directly from Definition \ref{definition:frechet-diff} shows that $h^*$ is Fréchet differentiable on $C_b(\mathcal{X})$ with Fréchet derivative given by 
\begin{equation*}
    \nabla_{\mathcal{F}} h^*(f)[g] = \frac{1}{2}\int_{\mathcal{X}}g(z)\pi(\mathrm{d}z) + \frac{1}{4}\int_{\mathcal{X}}g(z)f(z)\pi(\mathrm{d}z),
\end{equation*}
for all $g \in C_b(\mathcal{X}).$
\end{example}
\begin{definition}[Gâteaux differentiability]
\label{definition:gateaux-diff}
Let $\mathcal{U} \subseteq C_b\left(\mathcal{X}\right)$ be open. Given $f \in \mathcal{U},$ the function $H:\mathcal{U} \to \mathbb R$ is \textit{Gâteaux differentiable} at $f$ if there exists $T \in \mathcal{L}(C_b(\mathcal{X}),\mathbb R)$ such that for any direction $f' \in C_b(\mathcal{X}),$
\begin{equation*}
    \lim_{\varepsilon \downarrow 0} \frac{H\left(f+\varepsilon f'\right) - H\left(f\right)}{\varepsilon} = T\left[f'\right].
\end{equation*}
If it exists, the map $T$ is unique, we write $T = \nabla_{\mathcal{G}} H(f),$ and call $\nabla_{\mathcal{G}} H(f)$ the \textit{Gâteaux derivative} of $H$ at $f.$ If $H$ is Gâteaux differentiable at every $f \in \mathcal{U},$ then we say that $H$ is Gâteaux differentiable on $\mathcal{U}.$ 
\end{definition}

As observed in Chapter $1,3$ in \cite{ambrosetti1995primer, ortega1970iterative}, if $H$ is Fréchet differentiable, then it is automatically Gâteaux differentiable and the two derivatives coincide, i.e., $ \nabla_{\mathcal{F}} H =  \nabla_{\mathcal{G}} H.$ Moreover, \cite[Proposition $3.1.6$]{ortega1970iterative} proves that Fréchet differentiability of $H$ at $f \in \mathcal{U}$ implies that $H$ is continuous at $f,$ whereas in the case of Gâteaux differentiability, this does not necessarily hold; see \cite[Proposition $3.1.4$]{ortega1970iterative}.

Following the discussions in \cite{aliprantis2007infinite,ambrosetti1995primer,ortega1970iterative}, it is possible to extend Definition \ref{definition:frechet-diff} to higher-order Fréchet derivatives.  
\begin{definition}[Second-order Fréchet differentiability]
\label{def:second-order-frechet}
Let $\mathcal{U} \subseteq C_b\left(\mathcal{X}\right)$ be open and let $f \in \mathcal{U}.$ Suppose that $H:\mathcal{U} \to \mathbb R$ is \textit{Fréchet differentiable} (cf. Definition \ref{definition:frechet-diff}) at $f,$ and admits Fréchet derivative $\nabla_{\mathcal{F}}H(f).$ Then $\nabla_{\mathcal{F}}H(f)$ is \textit{Fréchet differentiable} at $f,$ if there exists $T \in \mathcal{L}\left(C_b(\mathcal{X}), \mathcal{L}\left(C_b(\mathcal{X}),\mathbb R\right)\right)$ such that for all $f',f'' \in C_b(\mathcal{X}),$
\begin{equation*}
    \lim_{\left\|f''\right\|_{\infty} \to 0} \frac{\left|\nabla_{\mathcal{F}}H\left(f+f''\right)[f'] - \nabla_{\mathcal{F}}H(f)[f'] - T\left[f''\right]\left[f'\right]\right|}{\left\|f''\right\|_{\infty}} = 0.
\end{equation*}
If it exists, the map $T$ is unique, we write $T = \nabla^2_{\mathcal{F}} H(f),$ and call $\nabla^2_{\mathcal{F}} H(f)$ the \textit{second Fréchet derivative} of $H$ at $f.$
\end{definition}
\begin{example}[Second-order Fréchet derivative of the relative entropy]
\label{example:entropy2}
    If $h$ is the relative entropy in Example \ref{example:relative-entropy}, using Example \ref{example:first-Frechet-entropy}, we can show that $\nabla_{\mathcal{F}} h^*(f)$ is Fréchet differentiable on $C_b\left(\mathcal{X}\right)$ with Fréchet derivative given by
    \begin{align*}
        \nabla^2_{\mathcal{F}} h^*(f)[g'][g] &= \int_{\mathcal{X}}g(x)g'(x)\varphi(f)(\mathrm{d}x) - \left( \int_{\mathcal{X}} g'(z)\varphi(f)(\mathrm{d}z)\right)\left(\int_{\mathcal{X}}g(z)\varphi(f)(\mathrm{d}z)\right)\\
        &=\operatorname{Cov}_{\varphi(f)}\left(g',g\right),
    \end{align*}
    for all $g,g' \in C_b(\mathcal{X}),$ where
    \begin{equation*}
        \varphi(f)(\mathrm{d}x) \coloneqq\frac{e^{f(x)}}{\int_\mathcal{X}e^{f(y)}\pi(\mathrm{d}y)}\pi(\mathrm{d}x)
    \end{equation*}
    If $g'=g,$ then
    \begin{equation*}
        \nabla^2_{\mathcal{F}} h^*(f)[g,g] = \operatorname{Var}_{\varphi(f)}\left(g\right).
    \end{equation*}
\end{example}
\begin{example}[Second-order Fréchet derivative of the $\chi^2$-divergence]
\label{example:chi2}
    If $h$ is the $\chi^2$-divergence in Example \ref{example:chi-squared}, using Example \ref{example:first-Frechet-chi}, we can show that $\nabla_{\mathcal{F}} h^*(f)$ is Fréchet differentiable on $C_b\left(\mathcal{X}\right)$ with Fréchet derivative given by
    \begin{equation*}
        \nabla^2_{\mathcal{F}} h^*(f)[g'][g] = \frac{1}{4}\int_{\mathcal{X}}g(z)g'(z)\pi(\mathrm{d}z),
    \end{equation*}
    for all $g,g' \in C_b(\mathcal{X}).$
\end{example}

\begin{definition}[Third-order Fréchet differentiability]
\label{def:third-Frechet}
Let $\mathcal{U} \subseteq C_b\left(\mathcal{X}\right)$ be open and let $f \in \mathcal{U}.$ Suppose that $H:\mathcal{U} \to \mathbb R$ is twice \textit{Fréchet differentiable} (cf. Definition \ref{def:second-order-frechet}) at $f,$ and admits second-order Fréchet derivative $\nabla^2_{\mathcal{F}}H(f).$ Then $\nabla^2_{\mathcal{F}}H(f)$ is \textit{Fréchet differentiable} at $f,$ if there exists $T \in \mathcal{L}\left(C_b(\mathcal{X}), \mathcal{L}\left(C_b(\mathcal{X}),\mathcal{L}\left(C_b(\mathcal{X}),\mathbb R\right)\right)\right)$ such that for all $g, g',g'' \in C_b(\mathcal{X}),$
\begin{equation*}
    \lim_{\left\|f''\right\|_{\infty} \to 0} \frac{\left|\nabla_{\mathcal{F}}H\left(f+g''\right)[g'][g] - \nabla_{\mathcal{F}}H(f)[g'][g] - T\left[g''\right]\left[g'\right]\left[g\right]\right|}{\left\|g''\right\|_{\infty}} = 0.
\end{equation*}
If it exists, the map $T$ is unique, we write $T = \nabla^3_{\mathcal{F}} H(f),$ and call $\nabla^3_{\mathcal{F}} H(f)$ the \textit{third Fréchet derivative} of $H$ at $f.$
\end{definition}
\begin{example}[Third-order Fréchet derivative of the relative entropy]
\label{example:entropy3}
    If $h$ is the relative entropy in Example \ref{example:relative-entropy}, using Example \ref{example:entropy2}, we can show that $\nabla^2_{\mathcal{F}} h^*(f)$ is Fréchet differentiable on $C_b\left(\mathcal{X}\right)$ with Fréchet derivative $\nabla^3_{\mathcal{F}} h^*(f).$ We differentiate the variance $\operatorname{Var}_{\varphi(f)}(g)$ with respect to $f$ in a direction $g'.$ Using the identity $$\frac{\mathrm{d}}{\mathrm{d}\varepsilon}\Bigg|_{\varepsilon=0}\left(\int_{\mathcal{X}} G(x) \varphi(f+\varepsilon g')(\mathrm{d}x)\right) = \operatorname{Cov}_{\varphi(f)}(G,g'),$$ we obtain
    \begin{align*}
        \nabla^3_{\mathcal{F}} h^*(f)[g,g][g'] &= \frac{\mathrm{d}}{\mathrm{d}\varepsilon}\Bigg|_{\varepsilon=0}\left(\int_{\mathcal{X}} g(x)^2 \varphi(f+\varepsilon g')(\mathrm{d}x) - \left(\int_{\mathcal{X}} g(x) \varphi(f+\varepsilon g')(\mathrm{d}x)\right)^2\right)\\
        &=\operatorname{Cov}_{\varphi(f)}(g^2,g') - 2\operatorname{Cov}_{\varphi(f)}(g,g') \int_\mathcal{X}g(x)\varphi(f)(\mathrm{d}x),
    \end{align*}
    for all $g,g' \in C_b(\mathcal{X}).$ If $g'=g,$ then
    \begin{equation*}
        \nabla^3_{\mathcal{F}} h^*(f)[g,g,g] = \int_{\mathcal{X}} \left(g(x)- \int_{\mathcal{X}} g(y) \varphi(f)(\mathrm{d}y)\right)^3\varphi(f)(\mathrm{d}x).
    \end{equation*}
\end{example}
\begin{example}[Third-order Fréchet derivative of the $\chi^2$-divergence]
\label{example:chi3}
    If $h$ is the $\chi^2$-divergence in Example \ref{example:chi-squared}, using Example \ref{example:chi2}, we observe that since $\nabla^2_{\mathcal{F}} h^*(f)$ is independent of $f,$ the third-order Fréchet derivative of $h^*$ is given by
    \begin{equation*}
        \nabla^3_{\mathcal{F}} h^*(f)[g,g,g] = 0,
    \end{equation*}
    for all $g \in C_b(\mathcal{X}).$
\end{example}
The motivation behind working with Fréchet instead of Gâteaux differentiability is that the higher-order derivatives in the case of the former could be identified with continuous symmetric multilinear maps. As proved in Section $3$ of Chapter $1$ from \cite{ambrosetti1995primer}, the space $\mathcal{L}\left(C_b(\mathcal{X}), \mathcal{L}\left(C_b(\mathcal{X}),\mathcal{L}\left(C_b(\mathcal{X}),\mathbb R\right)\right)\right)$ is isometrically isomorphic to $\mathcal{L}_3\left(C_b(\mathcal{X}), \mathbb R\right),$ i.e., the space of continuous trilinear maps from $C_b(\mathcal{X}) \times C_b(\mathcal{X}) \times C_b(\mathcal{X})$ to $\mathbb R,$ and therefore, we could naturally view the third-order Fréchet derivative of $H,$ if it exists, as a continuous trilinear map. Furthermore, due to \cite[Theorem $3.5$]{ambrosetti1995primer}, we have that the third-order Fréchet derivative is always symmetric. On the contrary, the second-order Gâteaux derivative is not necessarily symmetric as noted on page $78$ in \cite{ortega1970iterative}. 
\begin{remark}
    If we replace $C_b(\mathcal{X})$ with $\mathbb R^d$, then the first and second-order Fréchet derivatives are precisely the gradient and Hessian matrix of $H$ at $f.$
\end{remark}
 \begin{proposition}[Verification of Assumption \ref{assumption:lipschitz-h^*2} for the relative entropy in Example \ref{example:relative-entropy}]
\label{eq:verification-assumption}
    For $h$ being the relative entropy in Example \ref{example:relative-entropy}, its third Frechet derivative in Example \ref{example:entropy3} satisfies Assumption \ref{assumption:lipschitz-h^*2}. 
\end{proposition}
\begin{proof}
    Let $g \in C_b(\mathcal{X}).$ Recall that
    \begin{equation*}
        \nabla^3_{\mathcal{F}} h^*(f)[g,g,g] = \int_{\mathcal{X}} \left(g(x)- \int_{\mathcal{X}} g(y) \varphi(f)(\mathrm{d}y)\right)^3\varphi(f)(\mathrm{d}x).
    \end{equation*}
    Note that since $\varphi(f) \in \mathcal{P}(\mathcal{X})$ we have
    \begin{equation*}
        \left|g(x)- \int_{\mathcal{X}} g(y) \varphi(f)(\mathrm{d}y)\right| \leq 2\|g\|_\infty,
    \end{equation*}
    and hence
    \begin{equation*}
        |\nabla^3_{\mathcal{F}} h^*(f)[g,g,g]| \leq 8\|g\|_\infty^3.
    \end{equation*}
    Using the fact that
    \begin{equation*}
        \left\|\nabla^3_{\mathcal{F}} h^*(f)\right\|_{\mathcal{L}_3\left(C_b(\mathcal{X}), \mathbb R\right)} \coloneqq \sup_{\|g\|_\infty = 1}\frac{|\nabla^3_{\mathcal{F}} h^*(f)[g,g,g]|}{\|g\|_\infty^3},
    \end{equation*}
    we conclude that
    \begin{equation*}
        \left\|\nabla^3_{\mathcal{F}} h^*(f)\right\|_{\mathcal{L}_3\left(C_b(\mathcal{X}), \mathbb R\right)} \leq 8,
    \end{equation*}
    for all $f \in C_b(\mathcal{X}).$
 \end{proof}
 \begin{example}[Why self-concordance may not hold for the entropy mirror map]
 \label{example:self-concordance-fail}
    When \(h\) is the relative entropy in Example \ref{example:relative-entropy}, \(\nabla^2_{\mathcal F} h^*(f)[g,g]=\mathrm{Var}_{\varphi(f)}(g)\) and \(\nabla^3_{\mathcal F} h^*(f)[g,g,g]\) is the centered third moment of \(g\) under \(\varphi(f)\), i.e.,
    \begin{equation*}
        \nabla^3_{\mathcal{F}} h^*(f)[g,g,g] = \int_{\mathcal{X}} \left(g(x)- \int_{\mathcal{X}} g(y) \varphi(f)(\mathrm{d}y)\right)^3\varphi(f)(\mathrm{d}x).
    \end{equation*}
     In this case, self-concordance can fail globally because the normalized skewness
     \begin{equation*}
         \frac{\int_{\mathcal{X}} \left(g(x)- \int_{\mathcal{X}} g(y) \varphi(f)(\mathrm{d}y)\right)^3\varphi(f)(\mathrm{d}x)}{(\mathrm{Var}_{\varphi(f)}(g))^{3/2}}
     \end{equation*}
can be arbitrarily large as \(\varphi(f)\) becomes highly concentrated.
\end{example}
 \begin{proposition}[Verification of Assumption \ref{assumption:lipschitz-h^*2} for the relative $\chi^2$-divergence in Example \ref{example:chi-squared}]
\label{eq:verification-assumption-chi}
    For $h$ being the $\chi^2$-divergence in Example \ref{example:chi-squared}, its third Frechet derivative in Example \ref{example:chi3} satisfies Assumption \ref{assumption:lipschitz-h^*2}. 
\end{proposition}
\begin{proof}
    Recall that \begin{equation*}
        \nabla^3_{\mathcal{F}} h^*(f)[g,g,g] = 0,
    \end{equation*}
    hence the conclusion is immediate.
\end{proof}
\subsection{First variation}
Following Chapter $2$ from \cite{abraham2012manifolds}, we introduce the notion of first variation for Fréchet differentiable functions $H,$ relative to the duality pairing \eqref{eq:duality-pairing}.
\begin{definition}[First variation of $H$]
\label{def:first-variation}
    Let $H:C_b(\mathcal{X}) \to \mathbb R$ be Fréchet differentiable at $f \in C_b(\mathcal{X}).$ If it exists, the \textit{first variation} of $H$ at $f$ is the element $\frac{\delta H}{\delta f}(f) \in \mathcal{M}(\mathcal{X})$ such that, for all $g \in C_b(\mathcal{X}),$
    \begin{equation*}
        \left\langle g, \frac{\delta H}{\delta f}(f)\right\rangle \coloneqq \nabla_{\mathcal{F}}H(f)[g].
    \end{equation*}
\end{definition}
\begin{example}[First variation of the dual of relative entropy]
\label{example:entropy-dual-first-var}
    From Example \ref{example:entropy1}, we observe that the first variation $\frac{\delta h^*}{\delta f}(f) \in \mathcal{P}(\mathcal{X}) \subseteq \mathcal{M}(\mathcal{X})$ of $h^*$ at $f$ is given by
\begin{equation*}
    \frac{\delta h^*}{\delta f}(f)(\mathrm{d}z) = \varphi(f)(\mathrm{d}z).
\end{equation*}
\end{example}
\begin{example}[First variation of the dual of $\chi^2$-divergence]
\label{example:chi-dual-first-var}
    From Example \ref{example:chi1}, we observe that the first variation $\frac{\delta h^*}{\delta f} \in \mathcal{M}(\mathcal{X})$ of $h^*$ at $f$ is given by
\begin{equation*}
    \frac{\delta h^*}{\delta f}(f)(\mathrm{d}z) = \left(1+\frac{1}{2}f(z)\right)\pi(\mathrm{d}z).
\end{equation*}
\end{example}
Assuming that $H:C_b(\mathcal{X}) \to \mathbb R$ is Fréchet differentiable at $f \in C_b(\mathcal{X})$ with Fréchet derivative $\nabla_{\mathcal{F}}H(f),$ then it is Gâteaux differentiable (cf. Definition \ref{definition:gateaux-diff}) with the same derivative, and therefore the first variation of $H$ at $f$ can be characterized as
\begin{equation}
\label{eq:first-variation}
    \left\langle g, \frac{\delta H}{\delta f}(f)\right\rangle = \lim_{\varepsilon \downarrow 0} \frac{1}{\varepsilon}\left(H\left(f+\varepsilon g\right) - H\left(f\right)\right),
\end{equation}
for all $g \in C_b(\mathcal{X}).$

With the definition of first variation at hand, we can introduce necessary and sufficient conditions for $H$ to have an extremum at $f \in C_b(\mathcal{X}).$
\begin{lemma}[Necessary first-order condition on $C_b(\mathcal{X})$]
\label{lemma:FOC-cts-functions}
    Suppose $H:C_b(\mathcal{X}) \to \mathbb R$ admits first variation at $f.$ If $H$ has an extremum at $f^*,$ then it holds that
    \begin{equation*}
        \frac{\delta H}{\delta f}(f^*) = 0.
    \end{equation*}
\end{lemma}
\begin{proof}
    For a proof, see \cite[Proposition $2.4.22$]{abraham2012manifolds}.
\end{proof}
\begin{lemma}[Sufficient first-order condition on $C_b(\mathcal{X})$]
\label{lemma:suff-FOC-cts-functions}
Let $\mathcal{U} \subseteq C_b(\mathcal{X})$ be non-empty and convex. Suppose that $H:\mathcal{U} \to \mathbb R$ admits first variation on $\mathcal{U}$ and is convex in the sense that, for all $\lambda \in [0,1],$ and all $f,g \in \mathcal{U},$ it holds that $H\left((1-\lambda)f + \lambda g\right) \leq (1-\lambda)H(f) + \lambda H(g).$ If $\frac{\delta H}{\delta f}(f^*) = 0$, for some $f^* \in \mathcal{U},$ then $f^*$ is a global minimum of $H.$
\end{lemma}
\begin{remark}
    An analogous result can be identically proved for concave functions and global maxima, so we will give the proof only for the convex case.
\end{remark}
\begin{proof}
    Since $H$ is convex and admits first variation, following the argument in Lemma \ref{lemma:strong-convexity-h}, it can be showed that for any $f,g \in \mathcal{U}$
    \begin{equation*}
        H(g) \geq H(f) + \left\langle g-f, \frac{\delta H}{\delta f}(f)\right\rangle.
    \end{equation*}
    For $f=f^*$ and using the assumption that $\frac{\delta H}{\delta f}(f^*) = 0,$ we get
    \begin{equation*}
            H(g) \geq H(f^*),
    \end{equation*}
    for all $g \in \mathcal{U},$ i.e., $f^*$ is a global minimum.
\end{proof}

\section{Technical results on duality}
\label{appendix:technical-duality}
In this section we state and prove some technical results which are central to the proof technique via dual Bregman divergence that we developed in Subsection \ref{subsect:alternating}.
\begin{proposition}
\label{prop:first-var-argmax}
   Let Assumption \ref{assumption:assump-h} hold. Let $h^*:C_b(\mathcal{X}) \to \mathbb R$ be the convex conjugate of $h.$ Then, the following are equivalent:
   \begin{enumerate}
       \item The supremum of $\mathcal{M}(\mathcal{X}) \ni m \mapsto \langle g^*, m \rangle - h(m) \in \mathbb R$ is attained at $m = m^*$.
       \item We have the first-order condition $g^*(x) - \frac{\delta h}{\delta m}(m^*, x) = 0$, for all $x \in \mathcal{X}$, $m^*$-a.e.
       \item The supremum of $C_b(\mathcal{X}) \ni  g \mapsto \langle g, m^* \rangle - h^*(g) \in \mathbb R$ is attained at $g = g^*$.
       \item It holds that $m^* = \frac{\delta h^*}{\delta g}(g^*)$.
   \end{enumerate}
\end{proposition}
\begin{proof}
$(1) \implies (2)$: Suppose that $(1)$ holds. Then the supremum of $m \mapsto \langle g^*, m \rangle - h(m)$ is attained at the maximizer $m^* = \argmax_{m \in \mathcal{M}(\mathcal{X})}\left\{\langle g^*, m \rangle - h(m)\right\}.$ Hence, 
\begin{equation*}
    \langle g^*, m^*-m \rangle - \left(h(m^*)-h(m)\right)  \geq 0,
\end{equation*}
for all $m \in \mathcal{M}(\mathcal{X})$. Let $\Tilde{m} \in  \mathcal{M}(\mathcal{X})$ and set $m \coloneqq m^* + t(\Tilde{m}-m^*)$, for $t \in [0,1]$. Then
\begin{equation*}
    -t\langle g^*, \Tilde{m}-m^* \rangle + \left(h( m^* + t(\Tilde{m}-m^*))- h(m^*)\right) \geq 0.
\end{equation*}
Dividing by $t$ and letting $t \searrow 0$ gives
\begin{equation*}
    -\langle g^*, \Tilde{m}-m^* \rangle + \int_\mathcal{X} \frac{\delta h}{\delta m}(m^*,x)(\Tilde{m}-m^*)(\mathrm{d}x) \geq 0,
\end{equation*}
or equivalently
\begin{equation*}
    \left\langle -g^* + \frac{\delta h}{\delta m}(m^*,\cdot), \Tilde{m}-m^* \right\rangle \geq 0.
\end{equation*}
Since $\Tilde{m}$ is arbitrary, $m^*$ satisfies the first-order condition
\begin{equation*}
        g^*(x) - \frac{\delta h}{\delta m}(m^*, x) = 0,
    \end{equation*}
    for all $x \in \mathcal{X},$ $m^*$-a.e.

    $(2) \implies (1)$: Suppose that $(2)$ holds. Observe that the map $m \mapsto \langle g^*, m \rangle - h(m)$ is strictly concave due to the strict convexity of $h$ and the linearity of $m \mapsto \langle g^*, m \rangle.$ Therefore, $m^*$ is the maximizer of the map $\mathcal{M}(\mathcal{X}) \ni m \mapsto \langle g^*, m \rangle - h(m) \in \mathbb R,$ and so $(1)$ holds.

    $(3) \implies (4)$: Suppose that $(3)$ holds. Then the supremum in $g \mapsto \langle g, m^* \rangle - h^*(g)$ is attained at a maximizer $g^* \in \argmax_{g \in C_b(\mathcal{X})}\left\{\langle g, m^* \rangle - h^*(g)\right\}.$ Hence, by Lemma \ref{lemma:FOC-cts-functions}, it follows that $g^*$ satisfies the first-order condition
    \begin{equation*}
        m^* = \frac{\delta h}{\delta g}(g^*).
    \end{equation*}
    
    $(4) \implies (3)$: Suppose that $(4)$ holds. Observe that $C_b(\mathcal{X})$ is convex and the map $g \mapsto \langle g, m^* \rangle - h^*(g)$ is concave due to the convexity of $h^*$ and the linearity of $g \mapsto \langle g, m^* \rangle.$ Hence, by Lemma \ref{lemma:suff-FOC-cts-functions}, it follows that $g^*$ is a maximizer of the map $C_b(\mathcal{X}) \ni  g \mapsto \langle g, m^* \rangle - h^*(g) \in \mathbb R,$ and so $(3)$ holds. 
    
    $(1) \implies (3)$: Suppose that $(1)$ holds. Then, by Definition \ref{def: convex-conjugate}, we have that
    $h^*(g) = \langle g, m^* \rangle - h(m^*),$ and equivalently $h(m^*) = \langle g, m^* \rangle - h^*(g).$ Clearly, $\mathcal{M}(\mathcal{X})$ is convex and $\left(\mathcal{M}(\mathcal{X}), \operatorname{TV}\right)$ is Hausdorff since it is a metric space, hence we can apply the Fenchel-Moreau theorem \cite[Theorem $2.3.3$]{zalinescu2002convex} to conclude that $h^{**} = h$, i.e.,  $h(m^*) = \sup_{g \in C_b(\mathcal{X})} \{\langle g, m^* \rangle - h^*(g)\}.$ Therefore, $h(m^*)$ is the supremum of $g \mapsto \langle g, m^* \rangle - h^*(g)$ attained at $g=g^*.$
    
    $(3) \implies (1)$: Suppose $(3)$ holds. Then $h^{**}(m^*) = \langle g^*, m^* \rangle - h^*(g^*),$ or equivalently $h^*(g^*) = \langle g^*, m^* \rangle - h^{**}(m^*).$ Again, by the Fenchel-Moreau theorem \cite[Theorem $2.3.3$]{zalinescu2002convex}, $h^{**}(m) = h(m),$ for all $m \in \mathcal{M}(\mathcal{X}),$ and hence $h^*(g^*) = \langle g^*, m^* \rangle - h(m^*).$ Hence, by Definition \ref{def: convex-conjugate}, the supremum of $m \mapsto \langle g^*, m \rangle - h(m)$ is realized at $m=m^*.$
\end{proof}
\begin{corollary}
\label{corollary:first-var-argmax}
     Let $h^*:C_b(\mathcal{X}) \to \mathbb R$ be the convex conjugate of $h.$ If Assumption \ref{assumption:assump-h} holds and $h^*$ admits the first variation $\frac{\delta h^*}{\delta f}(f)$ (cf. \eqref{eq:first-variation}) on $C_b(\mathcal{X}),$ then
\begin{equation}
\label{eq:first-var-legendre}
    \frac{\delta h^*}{\delta f}(f) = \argmax_{m \in \mathcal{M}(\mathcal{X})}\left\{\langle f, m \rangle - h(m)\right\}.
\end{equation}
\end{corollary}
\begin{remark}[Bregman divergence via first variation]
    Definition \ref{def:dual-Bregman} can be relaxed as follows. Provided that $h^*$ admits a first variation (see Examples \ref{example:entropy-dual-first-var} and \ref{example:chi-dual-first-var}), Corollary \ref{corollary:first-var-argmax} shows that if Assumption \ref{assumption:assump-h} holds, then the first variation $\frac{\delta h^*}{\delta f}(f)$ of $h^*$ at $f$ is the unique maximizer of $m \mapsto \langle f, m \rangle - h(m).$ Consequently, from Definition \ref{def:first-variation}, since $f,f' \in C_b(\mathcal{X})$ and $\frac{\delta h^*}{\delta f}(f) \in \mathcal{M}(\mathcal{X}),$ it follows that $\nabla_\mathcal{F} h^*(f)[f'-f] = \langle f'-f,\frac{\delta h^*}{\delta f}(f)\rangle.$ Moreover, because $h^*$ is Fréchet-convex, $D_{h^*}(f',f) \geq 0,$ for all $f',f \in C_b(\mathcal{X}).$
\end{remark}
\begin{lemma}
\label{lemma:primal-dual-bregman}
Let Assumption \ref{assumption:assump-h} hold. Let $h^*:C_b(\mathcal{X}) \to \mathbb R$ be the convex conjugate of $h.$ Fix $f,g \in C_b(\mathcal{X})$ and $\mu, \mu' \in \mathcal{E}.$ If $f(z) = \frac{\delta h}{\delta m}(\mu, z)$ and $g(z) = \frac{\delta h}{\delta m}(\mu', z),$ for all $z \in \mathcal{X},$ $\mu$-a.e. and $\mu'$-a.e., respectively, then
    \begin{equation*}
        D_{h^*}(f,g) = D_h(\mu', \mu).
    \end{equation*}
\end{lemma}
\begin{proof}
By Definition \ref{def:dual-Bregman}, we have that
\begin{align*}
        &D_{h^*}(f,g) = h^*(f) - h^*(g) - \int_\mathcal{X}\left(f(z)-g(z)\right)\frac{\delta h^*}{\delta g}(g)(\mathrm{d}z)\\
        &= \langle f, \mu \rangle - h(\mu) - \langle g, \mu' \rangle + h(\mu') - \int_\mathcal{X}\left(f(z)-g(z)\right)\frac{\delta h^*}{\delta g}(g)(\mathrm{d}z)\\
        &= h(\mu') - h(\mu) + \int_{\mathcal{X}} \frac{\delta h}{\delta m}(\mu, z) \mu(\mathrm{d}z)\\ 
        &- \int_{\mathcal{X}} \frac{\delta h}{\delta m}(\mu', z)\mu'(\mathrm{d}z) - \int_\mathcal{X}\left(\frac{\delta h}{\delta m}(\mu, z)-\frac{\delta h}{\delta m}(\mu', z)\right)\mu'(\mathrm{d}z)\\
        &= h(\mu') - h(\mu) - \int_\mathcal{X}\frac{\delta h}{\delta m}(\mu, z)(\mu'-\mu)(\mathrm{d}z)
        = D_h(\mu', \mu),
    \end{align*}
    where the second and third equalities follow from Lemma \ref{prop:first-var-argmax} and Corollary \ref{corollary:first-var-argmax}, while the last equality follows from the definition of the Bregman divergence. 
\end{proof}
\begin{lemma}
\label{lemma:primal-dual-iterates}
    Consider Algorithms \ref{eq:mirror-sim-explicit} and \ref{eq:mirror-alt}. Let Assumption \ref{assumption:assump-h} hold. Let $h^*:C_b(\mathcal{X}) \to \mathbb R$ be the convex conjugate of $h.$ For each $n \geq 0,$ fix $f^n, g^n \in C_b(\mathcal{X}),$ $\nu^n \in \mathcal{C}$ and $\mu^n \in \mathcal{D}.$ If $f^n = \frac{\delta h}{\delta \nu}(\nu^n, \cdot)$ and $g^n = \frac{\delta h}{\delta \mu}(\mu^n, \cdot),$ then, for any $n \geq 0,$ we have that
    \begin{multline*}
        D_h(\nu^{n+1}, \nu^n) = D_{h^*}(f^n, f^{n+1}), \quad D_h(\nu^{n}, \nu^{n+1}) = D_{h^*}(f^{n+1}, f^{n}),\\ D_h(\mu^{n+1}, \mu^n) = D_{h^*}(g^n, g^{n+1}), \quad D_h(\mu^{n}, \mu^{n+1}) = D_{h^*}(g^{n+1}, g^{n}).
    \end{multline*}
\end{lemma}
\begin{proof}
    First, observe that due to Assumption \ref{assumption:assump-h}, the pairs $(\nu^{n+1}, \mu^{n+1})$ in \eqref{eq:mirror-sim-explicit} and \eqref{eq:mirror-alt} are unique. We will only present the proof for \eqref{eq:mirror-sim-explicit} since the argument for \eqref{eq:mirror-alt} is identical. The updates in \eqref{eq:mirror-sim-explicit} can be equivalently written as
\begin{align}
\label{eq:scheme2}
    &\nu^{n+1} = \argmin_{\nu \in \mathcal{C}} \left\{\int_{\mathcal{X}} \frac{\delta F}{\delta \nu}(\nu^{n}, \mu^{n}, x)(\nu-\nu^n)(\mathrm{d}x) + \frac{1}{\tau} D_h(\nu, \nu^n)\right\} \nonumber\\ 
    &= \argmin_{\nu \in \mathcal{C}} \left\{\int_{\mathcal{X}} \tau \frac{\delta F}{\delta \nu}(\nu^{n}, \mu^{n}, x)(\nu-\nu^n)(\mathrm{d}x) + h(\nu) - h(\nu^n) - \int_{\mathcal{X}} \frac{\delta h}{\delta \nu}(\nu^n,x)(\nu-\nu^n)(\mathrm{d}x)\right\} \nonumber\\
    &= \argmin_{\nu \in \mathcal{C}} \left\{\int_{\mathcal{X}} \left(\tau \frac{\delta F}{\delta \nu}(\nu^n, \mu^{n}, x) - \frac{\delta h}{\delta \nu}(\nu^n,x)\right)(\nu-\nu^n)(\mathrm{d}x) + h(\nu)\right\}\\
    &= \argmax_{\nu \in \mathcal{C}} \left\{\int_{\mathcal{X}} \left(\frac{\delta h}{\delta \nu}(\nu^n,x) - \tau \frac{\delta F}{\delta \nu}(\nu^n, \mu^{n}, x)\right)(\nu-\nu^n)(\mathrm{d}x) - h(\nu)\right\} \nonumber\\
    &= \argmax_{\nu \in \mathcal{C}} \left\{\int_{\mathcal{X}} \left(\frac{\delta h}{\delta \nu}(\nu^n,x) - \tau \frac{\delta F}{\delta \nu}(\nu^n, \mu^{n}, x)\right)\nu(\mathrm{d}x) - h(\nu)\right\} \nonumber,
\end{align}
and 
\begin{align}
\label{eq:scheme3}
    &\mu^{n+1} = \argmax_{\mu \in \mathcal{D}} \left\{\int_{\mathcal{X}} \frac{\delta F}{\delta \mu}(\nu^{n}, \mu^{n}, y)(\mu-\mu^n)(\mathrm{d}y) - \frac{1}{\tau} D_h(\mu, \mu^n)\right\} \nonumber\\ 
    &= \argmax_{\mu \in \mathcal{D}} \left\{\int_{\mathcal{X}} \tau\frac{\delta F}{\delta \mu}(\nu^{n}, \mu^{n}, y)(\mu-\mu^n)(\mathrm{d}y) - h(\mu) + h(\mu^n) + \int_{\mathcal{X}} \frac{\delta h}{\delta \mu}(\mu^n, y)(\mu-\mu^n)(\mathrm{d}y)\right\} \nonumber\\
    &= \argmax_{\mu \in \mathcal{D}} \left\{\int_{\mathcal{X}} \left(\frac{\delta h}{\delta \mu}(\mu^n, y) + \tau\frac{\delta F}{\delta \mu}(\nu^{n}, \mu^{n}, y)\right)(\mu-\mu^n)(\mathrm{d}y) - h(\mu)\right\}\\
    &= \argmax_{\mu \in \mathcal{D}} \left\{\int_{\mathcal{X}} \left(\frac{\delta h}{\delta \mu}(\mu^n, y) + \tau\frac{\delta F}{\delta \mu}(\nu^{n}, \mu^{n}, y)\right)\mu(\mathrm{d}y) - h(\mu)\right\} \nonumber.
\end{align}
Using the notation $f^n = \frac{\delta h}{\delta \nu}(\nu^n, \cdot)$ and $g^n = \frac{\delta h}{\delta \mu}(\mu^n, \cdot),$ for each $n \geq 0,$ the first-order conditions for \eqref{eq:mirror-sim-explicit} in Proposition \ref{prop:foc} can be equivalently written as
\begin{equation}
\label{eq:foc-dual-nu}
    f^{n+1}(x) - f^n(x) = -\tau \frac{\delta F}{\delta \nu}(\nu^n, \mu^{n}, x),
\end{equation}
\begin{equation}
\label{eq:foc-dual-mu}
    g^{n+1}(y) - g^n(y) = \tau \frac{\delta F}{\delta \mu}(\nu^n, \mu^{n}, y),
\end{equation}
for all $(x,y) \in \mathcal{X} \times \mathcal{X},$ $\nu^{n+1}$-a.e. and $\mu^{n+1}$-a.e., respectively. Then, using \eqref{eq:first-var-legendre}, \eqref{eq:scheme2} becomes
\begin{align}
\label{eq:h^*-inverse-nu}
    \nu^{n+1} &= \argmax_{\nu \in \mathcal{C}} \left\{\int_{\mathcal{X}} \left(f^n(x) -\tau \frac{\delta F}{\delta \nu}(\nu^n, \mu^n, x)\right)\nu(\mathrm{d}x) - h(\nu)\right\} \nonumber\\ 
    &= \argmax_{\nu \in \mathcal{C}} \left\{\int_{\mathcal{X}} f^{n+1}(x)\nu(\mathrm{d}x) - h(\nu)\right\} = \frac{\delta h^*}{\delta f}(f^{n+1}),
\end{align}
for all $n \geq 0.$ Similarly, from \eqref{eq:scheme3}, we have that
\begin{equation}
\label{eq:h^*-inverse-mu}
    \mu^{n+1} = \frac{\delta h^*}{\delta f}(g^{n+1}),
\end{equation}
for all $n \geq 0.$ The conclusion follows directly from Lemma \ref{lemma:primal-dual-bregman}.
\end{proof}
\section{Convergence of the continuous-time dynamics and the MDA implicit algorithm}
\label{sec:proof-implicit-game}
In this section, we provide a formal calculation showing that the continuous-time gradient flow obtained by taking the limit $\tau \to 0$ in the dual iterative MDA schemes of Proposition \ref{prop:foc} converges at rate $\mathcal{O}(1/t)$ in NI for the time-averaged flows.

Moreover, we show that an implicit Euler discretization of this gradient flow achieves a linear convergence rate $\mathcal{O}(1/N)$, matching the continuous-time rate under the same convexity–concavity assumptions on $F$. However, this implicit scheme is not practically implementable, unlike the explicit Algorithms \ref{eq:mirror-sim-explicit} and \ref{eq:mirror-alt}.

Formally letting $\tau \to 0$ in the updates of Proposition \ref{prop:foc} yields the continuous-time flow
\begin{multline}
\begin{aligned}
\label{eq:Bregman flow}
        &\partial_t \frac{\delta h}{\delta \nu}(\nu_t,x) = -\frac{\delta F}{\delta \nu}(\nu_t, \mu_t,x), \quad \partial_t \frac{\delta h}{\delta \mu}(\mu_t,y) = \frac{\delta F}{\delta \mu}(\nu_t, \mu_t,y), \quad t > 0,
\end{aligned}
\end{multline}
with initial condition $(\nu_0, \mu_0) \in \mathcal{C} \times \mathcal{D}.$ For convenience, we assume this flow is well-posed, i.e., it admits a unique solution $(\nu_t,\mu_t)_{t\geq0}.$

For any $(\nu, \mu) \in \mathcal{C} \times \mathcal{D}$, and assuming the interchange of derivatives and integrals is valid, a direct calculation gives
\begin{align*}
&\partial_t D_h(\nu, \nu_t) = \partial_t\left( h(\nu) - h(\nu_t) - \int_\mathcal{X} \frac{\delta h}{\delta \nu}(\nu_t,x)(\nu-\nu_t)(\mathrm{d}x)\right)\\
&=-\partial_t h(\nu_t) - \partial_t \int_\mathcal{X} \frac{\delta h}{\delta \nu}(\nu_t,x)(\nu-\nu_t)(\mathrm{d}x)\\
&= -\int_\mathcal{X} \frac{\delta h}{\delta \nu}(\nu_t,x)\partial_t \nu_t(\mathrm{d}x) - \int_\mathcal{X} \partial_t\frac{\delta h}{\delta \nu}(\nu_t,x)(\nu-\nu_t)(\mathrm{d}x) - \int_\mathcal{X} \frac{\delta h}{\delta \nu}(\nu_t,x)\partial_t (\nu-\nu_t)(\mathrm{d}x)\\
&=-\int_\mathcal{X} \frac{\delta h}{\delta \nu}(\nu_t,x)\partial_t \nu_t(\mathrm{d}x) - \int_\mathcal{X} \partial_t\frac{\delta h}{\delta \nu}(\nu_t,x)(\nu-\nu_t)(\mathrm{d}x) + \int_\mathcal{X} \frac{\delta h}{\delta \nu}(\nu_t,x)\partial_t\nu_t(\mathrm{d}x)\\
&= \int_\mathcal{X} \frac{\delta F}{\delta \nu}(\nu_t, \mu_t,x)(\nu-\nu_t)(\mathrm{d}x).
\end{align*}
Following the same calculation for $D_h(\mu, \mu_t)$ we obtain
\begin{equation*}
    \partial_t D_h(\mu, \mu_t) = -\int_\mathcal{X} \frac{\delta F}{\delta \mu}(\nu_t, \mu_t,y)(\mu-\mu_t)(\mathrm{d}y).
\end{equation*}
Adding these and applying the convexity–concavity of $F$ (Assumption \ref{def: def-F-conv-conc}) yields
\begin{equation*}
    \partial_t \left(D_h(\nu, \nu_t) + D_h(\mu, \mu_t)\right) \leq F(\nu, \mu_t) - F(\nu_t, \mu_t) + F(\nu_t, \mu_t) - F(\nu_t, \mu).
\end{equation*}
Integrating, dividing by $t$ and applying Jensen's inequality to $F$ gives
\begin{equation*}
    F\left(\frac{1}{t}\int_0^t \nu_s\mathrm{d}s, \mu\right) - F\left(\nu, \frac{1}{t}\int_0^t \mu_s\mathrm{d}s\right) \leq \frac{1}{t}\left(\sup_{\nu \in \mathcal{C}}D_h(\nu, \nu_0) + \sup_{\mu \in \mathcal{D}}D_h(\mu, \mu_0)\right).
\end{equation*}
Hence, maximizing over $(\nu, \mu),$ we conclude that
\begin{equation*}
    \operatorname{NI}\left(\frac{1}{t}\int_0^t \nu_s\mathrm{d}s, \frac{1}{t}\int_0^t \mu_s\mathrm{d}s\right) \leq \frac{1}{t}\left(\sup_{\nu \in \mathcal{C}}D_h(\nu, \nu_0) + \sup_{\mu \in \mathcal{D}}D_h(\mu, \mu_0)\right),
\end{equation*}
establishing the $\mathcal{O}(1/t)$ rate. 

We now turn to the implicit MDA scheme. For a given stepsize $\tau > 0,$ and fixed initial pair of strategies $(\nu_0, \mu_0) \in \mathcal{C} \times \mathcal{D},$ for $n \geq 0,$ the \textit{implicit} MDA algorithm is defined by
\begin{algorithm}[!h] 
   \caption{\textsc{Implicit MDA}}
   \label{eq:mirror-sim-fp}
   \KwIn{Objective function F, initial measures $(\nu_0, \mu_0)$, stepsize $\tau > 0$}
   \For{$n =0, 1, \dots ,N-1$}
   {
    $\nu^{n+1} = \argmin\limits_{\nu \in \mathcal{C}} \{\int_{\mathcal{X}} \frac{\delta F}{\delta \nu}(\nu^n, \mu^{n+1}, x)(\nu-\nu^n)(\mathrm{d}x) + \frac{1}{\tau} D_h(\nu, \nu^n)\},$
    
    $\mu^{n+1} = \argmax\limits_{\mu \in \mathcal{D}} \{\int_{\mathcal{X}} \frac{\delta F}{\delta \mu}(\nu^{n+1}, \mu^{n}, y)(\mu-\mu^n)(\mathrm{d}y) - \frac{1}{\tau} D_h(\mu, \mu^n)\}$
   }
\KwOut{$\left(\frac{1}{N}\sum_{n=0}^{N-1}\nu^{n+1},\frac{1}{N}\sum_{n=0}^{N-1}\mu^n\right)$}
\end{algorithm}
\begin{theorem}[Convergence of the implicit MDA Algorithm \ref{eq:mirror-sim-fp}]
    Let $(\nu^0, \mu^0)$ be such that $\sup_{\nu \in \mathcal{C}} D_h(\nu, \nu^0) + \sup_{\mu \in \mathcal{D}} D_h(\mu, \mu^0) < \infty.$ Let Assumption \ref{assumption:assump-h}, \ref{def: def-F-conv-conc} and \ref{def:relative-smoothness} hold. Suppose that $\tau L \leq 1,$ where $L \coloneqq \max\{L_{\nu}, L_{\mu}\}.$ Then, we have
    \begin{equation*}
        \operatorname{NI}\left(\frac{1}{N}\sum_{n=0}^{N-1}\nu^{n+1}, \frac{1}{N}\sum_{n=0}^{N-1}\mu^{n+1}\right) \leq \frac{1}{N\tau}\left(\sup_{\nu \in \mathcal{C}} D_h(\nu, \nu^0) + \sup_{\mu \in \mathcal{D}} D_h(\mu, \mu^0)\right).
    \end{equation*}
\end{theorem}
\begin{proof}
Since $\nu \mapsto \tau \int \frac{\delta F}{\delta \nu}(\nu^n,\mu^{n+1},x)(\nu-\nu^n)(\mathrm{d}x)$ is convex, applying Lemma \ref{lemma:Bregman-prox-ineq} with $\Bar{\nu} = \nu^{n+1}$ and $\mu = \nu^n$ implies that, for any $\nu \in \mathcal{C},$ we have
\begin{multline*}
    \tau \int \frac{\delta F}{\delta \nu}(\nu^n,\mu^{n+1},x)(\nu-\nu^n)(\mathrm{d}x) + D_h(\nu, \nu^n) \geq \tau \int \frac{\delta F}{\delta \nu}(\nu^n,\mu^{n+1},x)(\nu^{n+1}-\nu^n)(\mathrm{d}x)\\ + D_h(\nu^{n+1}, \nu^n) + D_h(\nu, \nu^{n+1}),
\end{multline*}
or, equivalently,
\begin{multline}
\label{eq: bregpp-nu1}
    -\tau \int \frac{\delta F}{\delta \nu}(\nu^n,\mu^{n+1},x)(\nu-\nu^n)(\mathrm{d}x) - D_h(\nu, \nu^n) \leq -\tau \int \frac{\delta F}{\delta \nu}(\nu^n,\mu^{n+1},x)(\nu^{n+1}-\nu^n)(\mathrm{d}x)\\ - D_h(\nu^{n+1}, \nu^n) - D_h(\nu, \nu^{n+1}).
\end{multline}
Similarly, since $\mu \mapsto -\tau \int \frac{\delta F}{\delta \mu}(\nu^{n+1},\mu^n,y)(\mu-\mu^n)(\mathrm{d}y)$ is convex, applying Lemma \ref{lemma:Bregman-prox-ineq} with $\Bar{\nu} = \mu^{n+1}$ and $\mu = \mu^n$ implies that, for any $\mu \in \mathcal{D},$ we have 
\begin{multline}
\label{eq: bregpp-mu1}
    \tau \int \frac{\delta F}{\delta \mu}(\nu^{n+1},\mu^n,y)(\mu-\mu^n)(\mathrm{d}y) - D_h(\mu, \mu^n) \leq \tau \int \frac{\delta F}{\delta \mu}(\nu^{n+1},\mu^n,y)(\mu^{n+1}-\mu^n)(\mathrm{d}y)\\ 
    - D_h(\mu^{n+1}, \mu^n) - D_h(\mu, \mu^{n+1}).
\end{multline}
Using the convexity of $\nu \mapsto F(\nu,\mu)$ in \eqref{eq: bregpp-nu1}, with $\nu= \nu^n$ and $\mu = \mu^{n+1},$
we have that
\begin{multline}
\label{eq:new-eq3}
    F(\nu^n, \mu^{n+1}) - F(\nu, \mu^{n+1}) - \frac{1}{\tau}D_h(\nu, \nu^n) \leq \int_\mathcal{X} \frac{\delta F}{\delta \nu}(\nu^n,\mu^{n+1},x)(\nu^{n}-\nu^{n+1})(\mathrm{d}x)\\ - \frac{1}{\tau}D_h(\nu^{n+1}, \nu^n) - \frac{1}{\tau} D_h(\nu, \nu^{n+1}).
\end{multline}
From $L_{\nu}$-relative smoothness and the fact that $\tau L \leq 1,$ it follows that
    \begin{multline}
    \label{eq:new-Lnu-smooth}
        F(\nu^{n+1}, \mu^{n+1}) \leq F(\nu^n, \mu^{n+1}) + \int_{\mathcal{X}}\frac{\delta F}{\delta \nu}(\nu^n, \mu^{n+1}, x)(\nu^{n+1}-\nu^n)(\mathrm{d}x) + L_{\nu}D_h(\nu^{n+1}, \nu^n)\\
        \leq F(\nu^n, \mu^{n+1}) + \int_{\mathcal{X}}\frac{\delta F}{\delta \nu}(\nu^n, \mu^{n+1}, x)(\nu^{n+1}-\nu^n)(\mathrm{d}x) + \frac{1}{\tau}D_h(\nu^{n+1}, \nu^n).
    \end{multline}
Hence, combining \eqref{eq:new-eq3} with \eqref{eq:new-Lnu-smooth}, we obtain that
\begin{multline}
\label{eq:new-firstpart}
    F(\nu^n, \mu^{n+1}) - F(\nu, \mu^{n+1}) - \frac{1}{\tau}D_h(\nu, \nu^n) \leq F(\nu^n, \mu^{n+1}) - F(\nu^{n+1}, \mu^{n+1}) - \frac{1}{\tau} D_h(\nu, \nu^{n+1}).
\end{multline}
Similarly, using concavity of $\mu \mapsto F(\nu, \mu)$ in \eqref{eq: bregpp-mu1},  with $\nu= \nu^{n+1}$ and $\mu = \mu^n,$ we have that
\begin{multline}
\label{eq:new-eq4}
    F(\nu^{n+1}, \mu) - F(\nu^{n+1}, \mu^n) - \frac{1}{\tau}D_h(\mu, \mu^n) \leq \int_\mathcal{X} \frac{\delta F}{\delta \mu}(\nu^{n+1},\mu^n,y)(\mu^{n+1}-\mu^n)(\mathrm{d}y)\\ - \frac{1}{\tau}D_h(\mu^{n+1}, \mu^n) - \frac{1}{\tau}D_h(\mu, \mu^{n+1}).
\end{multline}
From $L_{\mu}$-relative smoothness and the fact that $\tau L \leq 1,$ it follows that
    \begin{multline}
    \label{eq:new-Lmu-smooth}
        F(\nu^{n+1}, \mu^{n+1}) \geq F(\nu^{n+1}, \mu^n) + \int_{\mathcal{X}}\frac{\delta F}{\delta \mu}(\nu^{n+1}, \mu^n, y)(\mu^{n+1}-\mu^n)(\mathrm{d}y) - L_{\mu}D_h(\mu^{n+1}, \mu^n)\\ \geq F(\nu^{n+1}, \mu^n) + \int_{\mathcal{X}}\frac{\delta F}{\delta \mu}(\nu^{n+1}, \mu^n, y)(\mu^{n+1}-\mu^n)(\mathrm{d}y) - \frac{1}{\tau}D_h(\mu^{n+1}, \mu^n).
    \end{multline}
Hence, combining \eqref{eq:new-eq4} with \eqref{eq:new-Lmu-smooth}, we obtain that
\begin{equation}
\label{eq:new-secondpart}
    F(\nu^{n+1}, \mu) - F(\nu^{n+1}, \mu^n) - \frac{1}{\tau}D_h(\mu, \mu^n) \leq F(\nu^{n+1}, \mu^{n+1}) - F(\nu^{n+1}, \mu^{n}) - \frac{1}{\tau}D_h(\mu, \mu^{n+1}).
\end{equation}
Adding inequalities \eqref{eq:new-firstpart} and \eqref{eq:new-secondpart} implies that
\begin{multline*}
    F(\nu^{n+1}, \mu) - F(\nu, \mu^{n+1}) \leq F(\nu^{n+1}, \mu^{n+1}) - F(\nu^{n+1}, \mu^{n+1})\\ 
    + \frac{1}{\tau}D_h(\nu, \nu^n) + \frac{1}{\tau}D_h(\mu, \mu^n) - \frac{1}{\tau} D_h(\nu, \nu^{n+1}) - \frac{1}{\tau}D_h(\mu, \mu^{n+1})
\end{multline*}
Summing the previous inequality over $n=0,1,...,N-1,$ bounding the right-hand side from above by its supremum over $(\nu,\mu),$ dividing by $N,$ applying Jensen's inequality and taking maximum over $(\nu,\mu)$ in the left-hand side leads to
\begin{equation*}
    \operatorname{NI}\left(\frac{1}{N}\sum_{n=0}^{N-1}\nu^{n+1}, \frac{1}{N}\sum_{n=0}^{N-1}\mu^{n+1}\right) \leq \frac{1}{N\tau} \left(\sup_{\nu \in \mathcal{C}}D_h(\nu, \nu^0) + \sup_{\mu \in \mathcal{D}} D_h(\mu, \mu^0)\right),
\end{equation*}
where the last inequality follows since $D_h(\nu, \nu^{N}) + D_h(\mu, \mu^{N}) \geq 0,$ for all $(\nu,\mu) \in \mathcal{C} \times \mathcal{D}.$
\end{proof}
\section{Further related works}
\label{app:further-works}
\paragraph{Comparison with Wasserstein MD and other mean-field min-max methods.}
Our work is also related to recent developments on optimization over spaces of probability measures. A first point of comparison is \cite{bonet2024mirror}, who study mirror descent on $\mathcal P_2(\mathbb R^d)$ using the Wasserstein geometry. By contrast, we work with the flat geometry and study min-max problems over $\mathcal P(\mathbb R^d)$. These two geometries are complementary and are suited to different goals.

In the flat geometry, convexity-concavity and relative smoothness are formulated along linear interpolations of measures $[0,1]\ni \varepsilon \mapsto (1-\varepsilon)\mu_0+\varepsilon\mu_1$. This formulation is aligned with our MDA analysis, since it interacts naturally with the duality pairing between measures and functions. It also leads to assumptions that are relatively easy to verify in several motivating examples, including GANs and RLHF. In the Wasserstein geometry, by contrast, convexity and smoothness are typically formulated along transport curves, in particular along geodesics of the form $[0,1]\ni \varepsilon \mapsto \big((1-\varepsilon)\operatorname{Id}
+\varepsilon \operatorname{T}_{\mu_0}^{\mu_1}\big)_{\#}\mu_0,$ where $\operatorname{T}_{\mu_0}^{\mu_1}$ denotes an optimal transport map from $\mu_0$ to $\mu_1$. This structure is natural for transport-based algorithms, but it imposes substantially different requirements. Such maps need not exist in general. For example, a standard sufficient condition is that $\mu_0$ is absolutely continuous with respect to Lebesgue measure. Moreover, for bilinear objective, i.e., $F(\nu,\mu)=\iint f(x,y)\nu(\mathrm dx)\mu(\mathrm dy)$, Wasserstein geodesic convexity-concavity would impose Euclidean convexity-concavity conditions on $f$. This would exclude several of our main motivating examples, such as GANs, RLHF, and adversarial training of mean-field neural networks, unless additional restrictive assumptions were imposed.

More generally, Wasserstein regularity is stronger than flat regularity. The Wasserstein gradient is expressed in terms of $\nabla \frac{\delta F}{\delta \mu}$, and therefore requires differentiability of the flat derivative itself. This is often delicate and can be difficult to verify in applications such as RLHF. Furthermore, certain important functionals, including relative entropy, are not Wasserstein differentiable in general and require subgradient formulations, as discussed in \cite{bonet2024mirror}. For these reasons, the flat geometry is better suited to the assumptions used in our MDA analysis. Relatedly, Remark \ref{remark: proof of thm 2.1} explains why replacing the $\operatorname{TV}^2$ strong convexity assumption on the mirror potential $h$ by a $W_2^2$ condition would substantially reduce the generality of our results and would require stronger assumptions, such as Talagrand inequalities along the iterates.

The limitation of the flat geometry is that it does not include the transport structure intrinsic to Wasserstein space. In the Wasserstein setting, an update may be represented by a transport map $\operatorname{T}^{n+1}$ and written as $\mu^{n+1}=(\operatorname{T}^{n+1})_{\#}\mu^n$, so the algorithm directly moves particles from $\mu^n$ to $\mu^{n+1}$. This perspective is particularly natural for particle-based algorithms and for objectives involving Sinkhorn divergences, sliced-Wasserstein distances, or interaction energies, as in \cite{bonet2024mirror}. In contrast, the flat derivative $\frac{\delta F}{\delta \mu}$ describes infinitesimal signed perturbations of the measure. This is the appropriate object for our analysis, but it does not capture the geodesic transport structure exploited by Wasserstein methods.

Another related contribution is \cite{pmlr-v291-cai25a}, which studies an entropy-regularized bilinear game over distributions using min-max Langevin dynamics and finite-particle approximations. Their setting assumes a four-times differentiable, strongly convex-strongly concave payoff with Lipschitz gradient. In contrast, we study simultaneous and alternating MDA schemes for unregularized convex-concave min-max problems on spaces of measures, under relative smoothness and mirror geometry. In particular, in the bilinear case we do not assume convexity-concavity of $f$, nor do we require $f$ to be differentiable or to have Lipschitz gradient. A sufficient condition for Assumption \ref{assumption:bddflatF} is simply boundedness of $f$. As discussed in Remark \ref{remark:bilinear-games}, relative smoothness is automatic for bilinear games in our framework, without additional regularity assumptions on $f$. Thus, the regularity requirements in our setting are substantially weaker.

The guarantees are also different in nature. Our results concern simultaneous and alternating MDA, with performance measured by the Nikaido--Isoda error. A main contribution of our paper is to show that, in our measure space setting, alternating MDA improves the rate from $\mathcal O(N^{-1/2})$ for simultaneous MDA to $\mathcal O(N^{-2/3})$. In this sense, the results of \cite{pmlr-v291-cai25a} neither imply nor are implied by ours since they address a different algorithmic and analytical regime.

Finally, \cite{kim2024symmetric} study an entropy-regularized distributional minimax problem and analyze Wasserstein descent-ascent dynamics with historical averaging. Their setting differs from ours in several respects. First, they consider a regularized Wasserstein dynamics, whereas we study unregularized convex-concave games on spaces of measures under mirror geometry. Second, their Theorem 3.4 gives continuous-time rates, that is, the unweighted averaging scheme achieves an $\mathcal O(\log T/T)$ rate, while the weighted averaging scheme achieves an $\mathcal O(1/T)$ rate. These continuous-time guarantees are not directly comparable to our discrete-time MDA rates, especially because their method is a history-averaged Wasserstein dynamics rather than an MDA scheme. As discussed further in Appendix \ref{sec:proof-implicit-game}, continuous-time rates should be interpreted as idealized limits and do not automatically translate into achievable rates for implementable discrete-time algorithms.

Moreover, the discrete-time result of \cite{kim2024symmetric}, stated in their Theorem 3.7 is a propagation-of-chaos result controlling time and particle approximations of the continuous-time dynamics. Our paper does not address propagation-of-chaos guarantees. Instead, our contribution is to establish convergence rates for explicit simultaneous and alternating MDA schemes on measure spaces, and to prove that alternating updates yield a faster rate than simultaneous updates in this framework. Thus, while \cite{kim2024symmetric} provides important results for regularized Wasserstein dynamics, it does not directly imply that the rates obtained here are suboptimal.

\paragraph{Comparison with infinite-dimensional Mirror Prox.} Besides the MDA algorithm, \cite{pmlr-v97-hsieh19b} considers the entropic Mirror Prox algorithm, which is computationally more expensive than simultaneous MDA since it requires two mirror descent steps per player per iteration, effectively doubling the number of gradient evaluations. Although \cite{pmlr-v97-hsieh19b} proves it achieves a $\mathcal{O}(N^{-1})$ convergence rate with deterministic gradients, in the stochastic case it attains only the same $\mathcal{O}(N^{-1/2})$ rate as simultaneous MDA. Since the stochastic setting is precisely the one arising in the particle-based algorithmic implementation, even if the theoretical infinite-dimensional Mirror Prox analysis yielded a $\mathcal{O}(N^{-1})$ rate, the corresponding particle system would still converge at best at $\mathcal{O}(N^{-1/2})$. 

Another approach based on reproducing kernel Hilbert spaces (RKHS) is developed in \cite{dvurechensky2024analysis} and achieves the same convergence rates $\mathcal{O}\left(N^{-1}\right)$ and $\mathcal{O}\left(N^{-1/2}\right)$ for the deterministic and stochastic Mirror Prox algorithm, respectively. Finally, we note that the convergence analysis of an alternating Mirror Prox scheme remains, to our knowledge, open. Designing how to alternate the four mirror descent steps in a way that improves convergence rates appears to be technically delicate. We view this as a potential direction for future work.

Extending our analysis to the Mirror Prox scheme is not straightforward due to the following technical issue. To replicate the assumptions on the objective $F$ used in \cite[Section 5.2.3]{bubeck2015convex} or \cite[Assumption 4.2]{dvurechensky2024analysis}, one would need to impose the following $\operatorname{TV}$-Lipschitz condition on the flat derivatives of $F$. Specifically, there exist $C_{11}, C_{12}, C_{21}, C_{22} > 0$ such that 
\begin{equation*}
    \left|\frac{\delta F}{\delta \nu}(\nu',\mu,x') - \frac{\delta F}{\delta \nu}(\nu,\mu,x)\right| \leq C_{11}(\operatorname{TV}(\nu',\nu) + |x'-x|),
\end{equation*}
\begin{equation*}
    \left|\frac{\delta F}{\delta \nu}(\nu,\mu',x') - \frac{\delta F}{\delta \nu}(\nu,\mu,x)\right| \leq C_{12}(\operatorname{TV}(\mu',\mu) + |x'-x|),
\end{equation*}
\begin{equation*}
    \left|\frac{\delta F}{\delta \mu}(\nu',\mu,y') - \frac{\delta F}{\delta \mu}(\nu,\mu,y)\right| \leq C_{21}(\operatorname{TV}(\nu',\nu) + |y'-y|).
\end{equation*}
\begin{equation*}
    \left|\frac{\delta F}{\delta \mu}(\nu,\mu',y') - \frac{\delta F}{\delta \mu}(\nu,\mu,y)\right| \leq C_{22}(\operatorname{TV}(\mu',\mu) + |y'-y|).
\end{equation*}
This $\operatorname{TV}$-Lipschitz assumption is substantially stronger than the relative smoothness condition (Assumption \ref{def:relative-smoothness}), which is needed only for the alternating scheme but not for the simultaneous one. For instance, the bilinear game
\begin{equation*}
    F(\nu, \mu) = \int_\mathcal{X} \int_\mathcal{X} f(x,y)\nu(\mathrm{d}x)\mu(\mathrm{d}y)
\end{equation*} 
satisfies Assumption \ref{def:relative-smoothness} for $L_\nu = L_\mu = 0,$ yet it satisfies the $\operatorname{TV}$-Lipschitz conditions above only when $f$ is both bounded and Lipschitz. Thus, showing that Mirror Prox achieves an $\mathcal{O}(N^{-1})$ convergence rate under the weaker relative smoothness assumption remains a challenging open question. 

\paragraph{Comparison with Euclidean alternating updates.}
A closely related line of work studies alternating methods for finite-dimensional bilinear zero-sum games. In particular, \cite{katona} analyze alternating mirror descent (AMD) from a symplectic/Hamiltonian perspective. Their $\mathcal{O}(N^{-4/5})$ bound on the average-iterate duality gap is obtained by viewing AMD as a symplectic Euler discretization, constructing a modified Hamiltonian, and exploiting high-order smoothness properties of the convex conjugate $h^*$ associated with the Bregman potential.

The techniques in \cite{katona}, however, appear to rely in an essential way on finite-dimensional Euclidean structure. The symplectic Euler interpretation uses a finite-dimensional representation in which AMD is related to a symplectic scheme through an invertible linear transformation. Moreover, the modified Hamiltonian is constructed using Poisson brackets, which are defined for sufficiently smooth functions on Euclidean spaces. Their argument also involves an infinite-series representation of the modified Hamiltonian and requires delicate control of this series, established only in certain regimes. Consequently, extending this approach to our setting would require infinite-dimensional analogues of several nontrivial objects, including a suitable notion of symplectic integrator on spaces of measures, an associated Hamiltonian formalism, and an appropriate version of Poisson brackets. Developing such a framework is an interesting direction, but it is not a direct consequence of the methods used here. Thus, the sharper rate in \cite{katona} exploits finite-dimensional bilinear structure that is not available in our general mean-field convex-concave setting.

Another related result is due to \cite{nan2026on}, who study alternating projected gradient descent-ascent (Alt-GDA) for bilinear games over compact convex sets. They prove a global $\mathcal{O}(1/T)$ rate under the assumption that an interior Nash equilibrium exists, and a local $\mathcal{O}(1/T)$ rate near a maximum-support equilibrium when this assumption is removed. This result is specific to Euclidean projected GDA rather than mirror descent. In particular, it corresponds to the quadratic Bregman potential $\frac12\|\cdot\|^2$ and does not immediately extend to general Bregman geometries, even in finite dimensions. In our setting, the updates are defined on spaces of measures, where there is no canonical Euclidean gradient step and the analysis must instead be formulated through the duality pairing between measures and functions. Hence the connection between Euclidean projected GDA and mirror descent-type dynamics on measure spaces is far from immediate.

Furthermore, the analysis of \cite{nan2026on} uses structural properties of finite-dimensional simplices. Concepts such as interior Nash equilibria and maximum support equilibria play a central role in both their assumptions and their proofs, and their admissible stepsize condition also depends on this finite-dimensional simplex structure. Therefore, as in the comparison with \cite{katona}, the sharper guarantees in \cite{nan2026on} are obtained in a substantially more specialized setting using tools that do not transfer to infinite-dimensional mean-field games over measures.

Overall, the results of \cite{katona} and \cite{nan2026on} show that sharper convergence rates for alternating methods are possible in special finite-dimensional bilinear regimes. Our setting is different in that it concerns general convex-concave games over spaces of measures. Extending such improved rates to this level of generality would require new geometric and analytic tools, and remains a highly nontrivial open direction.


\newpage
\section*{NeurIPS Paper Checklist}
\begin{enumerate}

\item {\bf Claims}
    \item[] Question: Do the main claims made in the abstract and introduction accurately reflect the paper's contributions and scope?
    \item[] Answer: \answerYes{} 
    \item[] Justification: See Sections \ref{section:2}, \ref{subsect:alternating} and \ref{section:proofs}.
    \item[] Guidelines:
    \begin{itemize}
        \item The answer \answerNA{} means that the abstract and introduction do not include the claims made in the paper.
        \item The abstract and/or introduction should clearly state the claims made, including the contributions made in the paper and important assumptions and limitations. A \answerNo{} or \answerNA{} answer to this question will not be perceived well by the reviewers. 
        \item The claims made should match theoretical and experimental results, and reflect how much the results can be expected to generalize to other settings. 
        \item It is fine to include aspirational goals as motivation as long as it is clear that these goals are not attained by the paper. 
    \end{itemize}

\item {\bf Limitations}
    \item[] Question: Does the paper discuss the limitations of the work performed by the authors?
    \item[] Answer: \answerYes{} 
    \item[] Justification: See Section \ref{app:further-works}.
    \item[] Guidelines:
    \begin{itemize}
        \item The answer \answerNA{} means that the paper has no limitation while the answer \answerNo{} means that the paper has limitations, but those are not discussed in the paper. 
        \item The authors are encouraged to create a separate ``Limitations'' section in their paper.
        \item The paper should point out any strong assumptions and how robust the results are to violations of these assumptions (e.g., independence assumptions, noiseless settings, model well-specification, asymptotic approximations only holding locally). The authors should reflect on how these assumptions might be violated in practice and what the implications would be.
        \item The authors should reflect on the scope of the claims made, e.g., if the approach was only tested on a few datasets or with a few runs. In general, empirical results often depend on implicit assumptions, which should be articulated.
        \item The authors should reflect on the factors that influence the performance of the approach. For example, a facial recognition algorithm may perform poorly when image resolution is low or images are taken in low lighting. Or a speech-to-text system might not be used reliably to provide closed captions for online lectures because it fails to handle technical jargon.
        \item The authors should discuss the computational efficiency of the proposed algorithms and how they scale with dataset size.
        \item If applicable, the authors should discuss possible limitations of their approach to address problems of privacy and fairness.
        \item While the authors might fear that complete honesty about limitations might be used by reviewers as grounds for rejection, a worse outcome might be that reviewers discover limitations that aren't acknowledged in the paper. The authors should use their best judgment and recognize that individual actions in favor of transparency play an important role in developing norms that preserve the integrity of the community. Reviewers will be specifically instructed to not penalize honesty concerning limitations.
    \end{itemize}

\item {\bf Theory assumptions and proofs}
    \item[] Question: For each theoretical result, does the paper provide the full set of assumptions and a complete (and correct) proof?
    \item[] Answer: \answerYes{} 
    \item[] Justification: See Section \ref{section:2}, \ref{subsect:alternating}, \ref{section:proofs} and the rest of the appendix.
    \item[] Guidelines:
    \begin{itemize}
        \item The answer \answerNA{} means that the paper does not include theoretical results. 
        \item All the theorems, formulas, and proofs in the paper should be numbered and cross-referenced.
        \item All assumptions should be clearly stated or referenced in the statement of any theorems.
        \item The proofs can either appear in the main paper or the supplemental material, but if they appear in the supplemental material, the authors are encouraged to provide a short proof sketch to provide intuition. 
        \item Inversely, any informal proof provided in the core of the paper should be complemented by formal proofs provided in appendix or supplemental material.
        \item Theorems and Lemmas that the proof relies upon should be properly referenced. 
    \end{itemize}

    \item {\bf Experimental result reproducibility}
    \item[] Question: Does the paper fully disclose all the information needed to reproduce the main experimental results of the paper to the extent that it affects the main claims and/or conclusions of the paper (regardless of whether the code and data are provided or not)?
    \item[] Answer: \answerYes{} 
    \item[] Justification: Details of the algorithm and the numerical experiments are provided in Appendix \ref{section:numerical-example}.
    \item[] Guidelines:
    \begin{itemize}
        \item The answer \answerNA{} means that the paper does not include experiments.
        \item If the paper includes experiments, a \answerNo{} answer to this question will not be perceived well by the reviewers: Making the paper reproducible is important, regardless of whether the code and data are provided or not.
        \item If the contribution is a dataset and\slash or model, the authors should describe the steps taken to make their results reproducible or verifiable. 
        \item Depending on the contribution, reproducibility can be accomplished in various ways. For example, if the contribution is a novel architecture, describing the architecture fully might suffice, or if the contribution is a specific model and empirical evaluation, it may be necessary to either make it possible for others to replicate the model with the same dataset, or provide access to the model. In general. releasing code and data is often one good way to accomplish this, but reproducibility can also be provided via detailed instructions for how to replicate the results, access to a hosted model (e.g., in the case of a large language model), releasing of a model checkpoint, or other means that are appropriate to the research performed.
        \item While NeurIPS does not require releasing code, the conference does require all submissions to provide some reasonable avenue for reproducibility, which may depend on the nature of the contribution. For example
        \begin{enumerate}
            \item If the contribution is primarily a new algorithm, the paper should make it clear how to reproduce that algorithm.
            \item If the contribution is primarily a new model architecture, the paper should describe the architecture clearly and fully.
            \item If the contribution is a new model (e.g., a large language model), then there should either be a way to access this model for reproducing the results or a way to reproduce the model (e.g., with an open-source dataset or instructions for how to construct the dataset).
            \item We recognize that reproducibility may be tricky in some cases, in which case authors are welcome to describe the particular way they provide for reproducibility. In the case of closed-source models, it may be that access to the model is limited in some way (e.g., to registered users), but it should be possible for other researchers to have some path to reproducing or verifying the results.
        \end{enumerate}
    \end{itemize}

\item {\bf Open access to data and code}
    \item[] Question: Does the paper provide open access to the data and code, with sufficient instructions to faithfully reproduce the main experimental results, as described in supplemental material?
    \item[] Answer: \answerNA{} 
    \item[] Justification: The experiments conducted are limited to toy simulations.
    \item[] Guidelines:
    \begin{itemize}
        \item The answer \answerNA{} means that paper does not include experiments requiring code.
        \item Please see the NeurIPS code and data submission guidelines (\url{https://neurips.cc/public/guides/CodeSubmissionPolicy}) for more details.
        \item While we encourage the release of code and data, we understand that this might not be possible, so \answerNo{} is an acceptable answer. Papers cannot be rejected simply for not including code, unless this is central to the contribution (e.g., for a new open-source benchmark).
        \item The instructions should contain the exact command and environment needed to run to reproduce the results. See the NeurIPS code and data submission guidelines (\url{https://neurips.cc/public/guides/CodeSubmissionPolicy}) for more details.
        \item The authors should provide instructions on data access and preparation, including how to access the raw data, preprocessed data, intermediate data, and generated data, etc.
        \item The authors should provide scripts to reproduce all experimental results for the new proposed method and baselines. If only a subset of experiments are reproducible, they should state which ones are omitted from the script and why.
        \item At submission time, to preserve anonymity, the authors should release anonymized versions (if applicable).
        \item Providing as much information as possible in supplemental material (appended to the paper) is recommended, but including URLs to data and code is permitted.
    \end{itemize}

\item {\bf Experimental setting/details}
    \item[] Question: Does the paper specify all the training and test details (e.g., data splits, hyperparameters, how they were chosen, type of optimizer) necessary to understand the results?
    \item[] Answer: \answerYes{} 
    \item[] Justification: All details are provided in Appendix \ref{section:numerical-example}.
    \item[] Guidelines:
    \begin{itemize}
        \item The answer \answerNA{} means that the paper does not include experiments.
        \item The experimental setting should be presented in the core of the paper to a level of detail that is necessary to appreciate the results and make sense of them.
        \item The full details can be provided either with the code, in appendix, or as supplemental material.
    \end{itemize}

\item {\bf Experiment statistical significance}
    \item[] Question: Does the paper report error bars suitably and correctly defined or other appropriate information about the statistical significance of the experiments?
    \item[] Answer: \answerNo{} 
    \item[] Justification: The experiments are toy illustrative examples only meant to demonstrate the qualitative behavior of the proposed approach, so I do not report error bars, confidence intervals, or statistical significance tests.
    \item[] Guidelines:
    \begin{itemize}
        \item The answer \answerNA{} means that the paper does not include experiments.
        \item The authors should answer \answerYes{} if the results are accompanied by error bars, confidence intervals, or statistical significance tests, at least for the experiments that support the main claims of the paper.
        \item The factors of variability that the error bars are capturing should be clearly stated (for example, train/test split, initialization, random drawing of some parameter, or overall run with given experimental conditions).
        \item The method for calculating the error bars should be explained (closed form formula, call to a library function, bootstrap, etc.)
        \item The assumptions made should be given (e.g., Normally distributed errors).
        \item It should be clear whether the error bar is the standard deviation or the standard error of the mean.
        \item It is OK to report 1-sigma error bars, but one should state it. The authors should preferably report a 2-sigma error bar than state that they have a 96\% CI, if the hypothesis of Normality of errors is not verified.
        \item For asymmetric distributions, the authors should be careful not to show in tables or figures symmetric error bars that would yield results that are out of range (e.g., negative error rates).
        \item If error bars are reported in tables or plots, the authors should explain in the text how they were calculated and reference the corresponding figures or tables in the text.
    \end{itemize}

\item {\bf Experiments compute resources}
    \item[] Question: For each experiment, does the paper provide sufficient information on the computer resources (type of compute workers, memory, time of execution) needed to reproduce the experiments?
    \item[] Answer: \answerYes{} 
    \item[] Justification: The experiments were run on an Apple M4 laptop, with each experiment taking on the order of tens of minutes.
    \item[] Guidelines:
    \begin{itemize}
        \item The answer \answerNA{} means that the paper does not include experiments.
        \item The paper should indicate the type of compute workers CPU or GPU, internal cluster, or cloud provider, including relevant memory and storage.
        \item The paper should provide the amount of compute required for each of the individual experimental runs as well as estimate the total compute. 
        \item The paper should disclose whether the full research project required more compute than the experiments reported in the paper (e.g., preliminary or failed experiments that didn't make it into the paper). 
    \end{itemize}
    
\item {\bf Code of ethics}
    \item[] Question: Does the research conducted in the paper conform, in every respect, with the NeurIPS Code of Ethics \url{https://neurips.cc/public/EthicsGuidelines}?
    \item[] Answer: \answerYes{} 
    \item[] Justification: The authors have verified that the research complies with the NeurIPS Code of Ethics.
    \item[] Guidelines:
    \begin{itemize}
        \item The answer \answerNA{} means that the authors have not reviewed the NeurIPS Code of Ethics.
        \item If the authors answer \answerNo, they should explain the special circumstances that require a deviation from the Code of Ethics.
        \item The authors should make sure to preserve anonymity (e.g., if there is a special consideration due to laws or regulations in their jurisdiction).
    \end{itemize}

\item {\bf Broader impacts}
    \item[] Question: Does the paper discuss both potential positive societal impacts and negative societal impacts of the work performed?
    \item[] Answer: \answerNA{} 
    \item[] Justification: The paper is primarily theoretical, and no immediate societal impact is anticipated.
    \item[] Guidelines:
    \begin{itemize}
        \item The answer \answerNA{} means that there is no societal impact of the work performed.
        \item If the authors answer \answerNA{} or \answerNo, they should explain why their work has no societal impact or why the paper does not address societal impact.
        \item Examples of negative societal impacts include potential malicious or unintended uses (e.g., disinformation, generating fake profiles, surveillance), fairness considerations (e.g., deployment of technologies that could make decisions that unfairly impact specific groups), privacy considerations, and security considerations.
        \item The conference expects that many papers will be foundational research and not tied to particular applications, let alone deployments. However, if there is a direct path to any negative applications, the authors should point it out. For example, it is legitimate to point out that an improvement in the quality of generative models could be used to generate Deepfakes for disinformation. On the other hand, it is not needed to point out that a generic algorithm for optimizing neural networks could enable people to train models that generate Deepfakes faster.
        \item The authors should consider possible harms that could arise when the technology is being used as intended and functioning correctly, harms that could arise when the technology is being used as intended but gives incorrect results, and harms following from (intentional or unintentional) misuse of the technology.
        \item If there are negative societal impacts, the authors could also discuss possible mitigation strategies (e.g., gated release of models, providing defenses in addition to attacks, mechanisms for monitoring misuse, mechanisms to monitor how a system learns from feedback over time, improving the efficiency and accessibility of ML).
    \end{itemize}
    
\item {\bf Safeguards}
    \item[] Question: Does the paper describe safeguards that have been put in place for responsible release of data or models that have a high risk for misuse (e.g., pre-trained language models, image generators, or scraped datasets)?
    \item[] Answer: \answerNA{} 
    \item[] Justification: The paper is primarily theoretical.
    \item[] Guidelines:
    \begin{itemize}
        \item The answer \answerNA{} means that the paper poses no such risks.
        \item Released models that have a high risk for misuse or dual-use should be released with necessary safeguards to allow for controlled use of the model, for example by requiring that users adhere to usage guidelines or restrictions to access the model or implementing safety filters. 
        \item Datasets that have been scraped from the Internet could pose safety risks. The authors should describe how they avoided releasing unsafe images.
        \item We recognize that providing effective safeguards is challenging, and many papers do not require this, but we encourage authors to take this into account and make a best faith effort.
    \end{itemize}

\item {\bf Licenses for existing assets}
    \item[] Question: Are the creators or original owners of assets (e.g., code, data, models), used in the paper, properly credited and are the license and terms of use explicitly mentioned and properly respected?
    \item[] Answer: \answerNA{} 
    \item[] Justification: The paper does not use existing assets.
    \item[] Guidelines:
    \begin{itemize}
        \item The answer \answerNA{} means that the paper does not use existing assets.
        \item The authors should cite the original paper that produced the code package or dataset.
        \item The authors should state which version of the asset is used and, if possible, include a URL.
        \item The name of the license (e.g., CC-BY 4.0) should be included for each asset.
        \item For scraped data from a particular source (e.g., website), the copyright and terms of service of that source should be provided.
        \item If assets are released, the license, copyright information, and terms of use in the package should be provided. For popular datasets, \url{paperswithcode.com/datasets} has curated licenses for some datasets. Their licensing guide can help determine the license of a dataset.
        \item For existing datasets that are re-packaged, both the original license and the license of the derived asset (if it has changed) should be provided.
        \item If this information is not available online, the authors are encouraged to reach out to the asset's creators.
    \end{itemize}

\item {\bf New assets}
    \item[] Question: Are new assets introduced in the paper well documented and is the documentation provided alongside the assets?
    \item[] Answer: \answerNA{} 
    \item[] Justification: The paper does not release new assets.
    \item[] Guidelines:
    \begin{itemize}
        \item The answer \answerNA{} means that the paper does not release new assets.
        \item Researchers should communicate the details of the dataset\slash code\slash model as part of their submissions via structured templates. This includes details about training, license, limitations, etc. 
        \item The paper should discuss whether and how consent was obtained from people whose asset is used.
        \item At submission time, remember to anonymize your assets (if applicable). You can either create an anonymized URL or include an anonymized zip file.
    \end{itemize}

\item {\bf Crowdsourcing and research with human subjects}
    \item[] Question: For crowdsourcing experiments and research with human subjects, does the paper include the full text of instructions given to participants and screenshots, if applicable, as well as details about compensation (if any)? 
    \item[] Answer: \answerNA{} 
    \item[] Justification: The paper does not involve crowdsourcing nor research with human subjects.
    \item[] Guidelines:
    \begin{itemize}
        \item The answer \answerNA{} means that the paper does not involve crowdsourcing nor research with human subjects.
        \item Including this information in the supplemental material is fine, but if the main contribution of the paper involves human subjects, then as much detail as possible should be included in the main paper. 
        \item According to the NeurIPS Code of Ethics, workers involved in data collection, curation, or other labor should be paid at least the minimum wage in the country of the data collector. 
    \end{itemize}

\item {\bf Institutional review board (IRB) approvals or equivalent for research with human subjects}
    \item[] Question: Does the paper describe potential risks incurred by study participants, whether such risks were disclosed to the subjects, and whether Institutional Review Board (IRB) approvals (or an equivalent approval/review based on the requirements of your country or institution) were obtained?
    \item[] Answer: \answerNA{} 
    \item[] Justification: The paper does not involve crowdsourcing nor research with human subjects.
    \item[] Guidelines:
    \begin{itemize}
        \item The answer \answerNA{} means that the paper does not involve crowdsourcing nor research with human subjects.
        \item Depending on the country in which research is conducted, IRB approval (or equivalent) may be required for any human subjects research. If you obtained IRB approval, you should clearly state this in the paper. 
        \item We recognize that the procedures for this may vary significantly between institutions and locations, and we expect authors to adhere to the NeurIPS Code of Ethics and the guidelines for their institution. 
        \item For initial submissions, do not include any information that would break anonymity (if applicable), such as the institution conducting the review.
    \end{itemize}

\item {\bf Declaration of LLM usage}
    \item[] Question: Does the paper describe the usage of LLMs if it is an important, original, or non-standard component of the core methods in this research? Note that if the LLM is used only for writing, editing, or formatting purposes and does \emph{not} impact the core methodology, scientific rigor, or originality of the research, declaration is not required.
    \item[] Answer: \answerNA{} 
    \item[] Justification: The core method development in this research does not involve LLMs as any important, original, or non-standard components.
    \item[] Guidelines:
    \begin{itemize}
        \item The answer \answerNA{} means that the core method development in this research does not involve LLMs as any important, original, or non-standard components.
        \item Please refer to our LLM policy in the NeurIPS handbook for what should or should not be described.
    \end{itemize}

\end{enumerate}

\end{document}